\documentclass[10pt]{article}

\usepackage{amssymb,amsmath,amsthm,amsfonts}
\usepackage{graphicx}
\usepackage{comment}
\usepackage{algorithm,algorithmic}
\graphicspath{{./pics/}}
\usepackage{epstopdf}
\usepackage{verbatim}
\usepackage{bookmark}
\usepackage{indentfirst}

\numberwithin{equation}{section}
\theoremstyle{plain}
\newtheorem{proposition}{Proposition}

\theoremstyle{remark}
\newtheorem{remark}{Remark}
\theoremstyle{definition}
\newtheorem{exam}{Example}

\usepackage{xcolor}

\renewcommand{\epsilon}{\varepsilon}
\DeclareMathOperator*{\argmax}{arg\,max}
\usepackage{subfig}

\title{The Linearized Inverse Problem in Multifrequency Electrical Impedance Tomography}
\author{Giovanni S.\ Alberti\thanks{Department of Mathematics,
ETH Z\"{u}rich, R\"{a}mistrasse 101, 8092 Z\"{u}rich, Switzerland. (\texttt{giovanni.alberti@sam.math.ethz.ch},
\texttt{habib.ammari@math.ethz.ch})}
\and Habib Ammari\footnotemark[1]
 \and Bangti Jin\thanks{Department of Computer Science, University College London,
Gower Street, London WC1E 6BT, UK. (\texttt{bangti.jin@gmail.com,b.jin@ucl.ac.uk})}
\and Jin-Keun Seo\thanks{Department of Computational Science and Engineering, Yonsei University, 50 Yonsei-Ro,
Seodaemun-Gu, Seoul 120-749, Korea. (\texttt{seoj@yonsei.ac.kr})}
\and Wenlong Zhang\thanks{Department of Mathematics and Applications, \'Ecole Normale Sup\'{e}rieure, 45 Rue d'Ulm, 75005
Paris, France. (\texttt{wenlong.zhang@ens.fr})}
}

\begin{document}
\maketitle

\begin{abstract}
This paper provides an analysis of the linearized inverse problem in multifrequency electrical
impedance tomography. We consider an isotropic conductivity distribution with a finite number of unknown
inclusions with different frequency dependence, as is often seen in biological tissues. We discuss reconstruction
methods for both fully known and partially known spectral profiles, and demonstrate in the latter case the
successful employment of difference imaging. We also study the reconstruction with an imperfectly known boundary,
and show that the multifrequency approach can eliminate modeling errors and recover almost all inclusions.
In addition, we develop an efficient group sparse recovery algorithm for the robust solution of related linear
inverse problems. Several numerical simulations are presented to illustrate and validate the approach.
\end{abstract}


\section{Introduction}\label{sect:intro}

Electrical impedance tomography (EIT) is a diffusive imaging modality that allows recovering the conductivity
 of an electrically conducting object by using electrodes to measure the resulting voltage
on its boundary, induced by multiple known injected currents. It is safe, cheap and portable, and is
potentially applicable to clinical imaging in a range of areas. However, the EIT
inverse problem is severely ill-posed, and has thus shown only modest image quality when compared with
other modalities \cite{borcea}. This has motivated numerous mathematical studies on EIT imaging
techniques including small anomaly conductivity imaging \cite{anomaly1, anomaly3, anomaly2, anomaly4}
and hybrid conductivity imaging \cite{hybrid1, hybrid2,hybrid3,hybrid4,hybrid5,hybrid6,hybrid8}.

Static imaging aims at recovering absolute conductivity values. Apart from the popular linearization approach, a
number of static imaging algorithms have been developed, e.g., least-squares method \cite{RondiSantosa:2001,ChungChan:2005,LechleiterRieder:2006,JinKhanMaass:2012,JinMaass:2012}, direct methods \cite{SiltanenMueller:2000,ChowItoZou:2014,Lechleiter:2015}, and statistical methods \cite{KaipioKolehmainen:2000,
GehreJin:2014}; see also the overviews \cite{borcea,Lionheart:2004}.
However, static imaging has so far achieved only limited success in practice, since
electrode voltages are insensitive  to localized conductivity changes, but sensitive to forward  modeling
errors, e.g., boundary shape and electrode positions. Hence, apart from accurate data, a very accurate forward
model is  required for its success; however, this is often difficult to obtain in practice. A prominent idea is
to use difference imaging, in the hope of canceling out modeling errors due to, e.g., boundary shape. A
traditional approach is time difference imaging, which produces  an image of the conductivity change  by inverting
a linearized sensitivity model. A second approach is multifrequency EIT (mfEIT), which has also attracted  attention in recent years.

Imaging by mfEIT  exploits the frequency dependence of the conductivity. Experimental
research has found that the conductivity of many biological tissues varies strongly with
the frequency \cite{GeddesBaker:1967,GabrielyLauGabriel:1996,LauferIvorra:2010}. In \cite{laure},
the authors analytically exhibited fundamental mechanisms underlying the fact that effective
electrical properties of biological tissue and their frequency dependence reflect the tissue
composition and physiology, and  a homogenization theory was developed.
In mfEIT, boundary voltages are recorded simultaneously, while varying the modulation
frequency of the injected current. It is expected to be especially useful for the diagnostic imaging
of conditions such as acute stroke, brain injury, and breast cancer, because patients are admitted
into care after the onset of the pathology and thus lack a baseline record for healthy tissue, whence
time difference imaging may not be used.

There have been several studies on frequency-difference imaging \cite{GriffithsAhmed:1987,SchlappaAnnese:2000,Yerworth:2003}. An mfEIT experimental design for head imaging was given in \cite{Yerworth:2003}. In these works, the simple frequency difference (between two neighboring frequencies) was often employed. Seo et al.\ \cite{SeoLeeKim:2008} proposed a weighted frequency difference imaging technique, based on a suitable weighted voltage difference between any \textit{two} sets of data. It was numerically shown that the approach can accommodate geometrical errors, including imperfectly known boundary. This approach can improve the imaging quality when the background is frequency dependent. Recently, Malone et al.\ \cite{MaloneSato:2014,MaloneSato:2015} proposed a nonlinear reconstruction scheme, which uses all multifrequency data directly to recover the volume fractions of the tissues, and validated the approach on phantom experimental data. Harrach and Seo \cite{HarrachSeo:2009}
developed a direct method for detecting inclusions from frequency-difference data. See also \cite{KimTamasan:2014} for a recovery algorithm at low frequencies.

This work analyzes mfEIT in the linearized regime, by linearizing the forward model around a constant
conductivity, as customarily adopted in practice. We shall discuss both the mathematically convenient continuum model and the practically popular complete electrode model. Our main contributions are as follows. First, we discuss mfEIT imaging for spectral profiles that are  known, or partially known, or unknown, and also  generalize existing studies, especially \cite{SeoLeeKim:2008}. Second, we rigorously justify  mfEIT for handling geometrical errors. Third, we present a novel group sparse reconstruction algorithm of iterative shrinkage type, which is easily implemented
and converges quickly. Extensive numerical experiments  confirm our discussions. All these findings shed new valuable insights into mfEIT, which are expected to be of great interest to the engineering community.

This paper is organized as follows. In  Section~\ref{sect:continuum}, we mathematically formulate
mfEIT using a continuum model, and analyze three important scenarios, depending on the
knowledge of the spectral profiles. Then, in  Section~\ref{sec:unknownboundary}, we illustrate the
potential of mfEIT in handling the modeling errors due to an imperfectly known boundary shape.
These analyses are then extended to the complete electrode model in  Section~\ref{sec:cem}.
In  Section~\ref{sect:sparsity}, we present a novel group sparse reconstruction algorithm.
In  Section~\ref{sec:numer}, extensive numerical experiments are presented to illustrate
the approach. Finally, some concluding remarks are discussed in Section~\ref{sec:conclusion}.

\section{The Continuum Model}\label{sect:continuum}
In this section, we mathematically formulate mfEIT in the  continuum model with
a known boundary.
Let $\Omega\subset\mathbb{R}^d$ ($d=2,3$) be a bounded domain with a smooth boundary $\partial\Omega$.
The forward problem reads: for an input current $f\in L^2_\diamond(\partial\Omega):=\{ g\in L^2(\partial\Omega):
\int_{\partial\Omega}g\,ds=0 \}$ and $\sigma(x,\omega)$, find $u(\cdot,\omega)\in H_\diamond(\Omega):=
\{ v\in H^1(\Omega): \int_{\partial\Omega}vds=0 \}$:
\begin{equation}\label{eqn:eit}
  \left\{\begin{aligned}
    -\nabla\cdot(\sigma(x,\omega)\nabla u(x,\omega)) & = 0\quad \mbox{ in }\Omega,\\
    \sigma(x,\omega)\frac{\partial u}{\partial\nu}& = f(x)\quad \mbox{on }\partial\Omega,
  \end{aligned}\right.
\end{equation}
where $\omega$ is the frequency and $\nu$ is the unit outward normal vector to $\partial\Omega$. The weak formulation of problem \eqref{eqn:eit} is to  find
$u=u(\cdot,\omega)\in H^1_\diamond(\Omega)$ such that
\begin{equation*}
  \int_\Omega \sigma\nabla u\cdot\nabla v\,dx = \int_{\partial\Omega}fv\,ds,\qquad  v\in H^1(\Omega).
\end{equation*}

Throughout, we assume that the conductivity $\sigma(x,\omega)$ takes a separable form
\begin{equation}\label{eqn:gamma}
  \sigma(x,\omega) = \sum_{k=0}^K\sigma_k(x)s_k(\omega),
\end{equation}
where $K+1$ is the number of spectral profiles,  $\{s_k(\omega)\}_{k=0}^K$ are the (possibly only
partially known) material spectra, a.k.a. endmembers, and $\{\sigma_k(x)\}_{k=0}^K$ are scalar functions
representing the corresponding proportions, a.k.a. abundances in the
hyperspectral unmixing literature \cite{KeshavaMustard:2002}. Further, we shall assume
\begin{equation*}
  \begin{aligned}
    \sigma_0(x) &= 1 + \delta\sigma_0(x), \\
    \sigma_k(x) &= \delta\sigma_k(x),\quad k=1,\ldots,K,
  \end{aligned}
\end{equation*}
where the $\delta\sigma_k$s, i.e., $\{\delta\sigma_k\}_{k=0}^K$, are small (in suitable $L^p(\Omega)$
norms) so that a linearized model is valid. The $\delta\sigma_k$s, including the background $\delta
\sigma_0$, are all unknown, represent the small inclusions/anomalies in the object $\Omega$, and have compact
spatial supports  that are disjoint from each other.  We also assume that the background frequency
$s_0(\omega)$ is known.

Now we apply $N$ linearly independent input currents $\{f_n\}_{n=1}^N\subset L_\diamond^2(\partial\Omega)$. Let $\{u_n\equiv
u_n(x,\omega)\}_{n=1}^N\subset H_\diamond^1(\Omega)$ be the corresponding solutions to \eqref{eqn:eit}, i.e.,
\begin{equation}\label{eqn:us}
\int_\Omega \sigma \nabla u_n \cdot\nabla v \,dx = \int_{\partial\Omega}f_n v \,ds,\qquad  v\in H^1(\Omega).
\end{equation}
The inverse problem is to recover $\delta\sigma_k$s
from  $\{u_n(x,\omega)\}_{n=1}^N$ on $\partial\Omega$ at the frequencies $\{\omega_q\}_{q=1}^Q$.

Next we derive the linearized model for the inverse problem, based on an integral representation. Let
$v_m\in H^1_\diamond(\Omega)$ be the potential corresponding to the
unperturbed conductivity $\sigma_0(x,\omega)\equiv s_0(\omega)$ with the input current $f_m\in L^2_\diamond(\partial\Omega)$, namely
\begin{equation}\label{eqn:vs}
\int_\Omega \sigma_0\nabla v_m \cdot\nabla v\,dx = \int_{\partial\Omega}f_mv\,ds,\qquad  v\in H^1(\Omega).
\end{equation}
Then $v_m = v^*_m / s_0(\omega)$, where $v^*_m$ is the solution of  \eqref{eqn:vs} corresponding to the case $s_0 \equiv 1$. Namely, the dependence of $v_m$ on the frequency $\omega$ is explicit. Using \eqref{eqn:us} and  \eqref{eqn:vs}, we obtain
\begin{equation*}
   \sum_{k=0}^Ks_k(\omega) \int_\Omega \delta\sigma_k\nabla u_n\cdot\nabla v_m\,dx =  \int_{\partial\Omega}( f_nv_m-f_mu_n)\,ds.
\end{equation*}
Hence, using the approximation $\nabla u_n(x,\omega)\approx \nabla v_n(x,\omega)$ in $\Omega$
(valid in the linear regime), and the identity $v_m = v^*_m / s_0(\omega)$, we arrive at a linearized model:
\begin{equation}\label{eqn:integral}
  \sum_{k=0}^K s_k(\omega)\int_\Omega \delta\sigma_k\nabla v_n^*\cdot \nabla v_m^*\,dx =s_0(\omega)^2\int_{\partial\Omega}(f_nv_m-f_mu_n)\,ds.
\end{equation}
The right hand side of \eqref{eqn:integral} can be treated as a known quantity: $u_m$ is the measured voltage data
(and thus depends on $\omega$), and $v_m$ is computable.
Next, we triangulate $\Omega$ into  a shape-regular quasi-uniform mesh
$\{\Omega_l\}_{l=1}^L$, and consider a piecewise constant approximation of $\delta\sigma_k$s:
\begin{equation}\label{eqn:piece-const}
  \delta\sigma_k(x) \approx \sum_{l=1}^L (\delta\sigma_k)_l\chi_{\Omega_l}(x),\quad k=0,1,\ldots,K,
\end{equation}
where  $\chi_{\Omega_l}$ is the characteristic function of the $l$th element $\Omega_l$, and $(\delta\sigma_k)_l$
denotes the corresponding value of $\delta\sigma_k$.
Thus we get a finite-dimensional linear inverse problem
\begin{equation*}
  \sum_{k=0}^K s_k(\omega)\sum_{l=1}^L(\delta\sigma_k)_l\int_{\Omega_l} \nabla v^*_n\cdot \nabla v^*_m\,dx = s_0(\omega)^2 \int_{\partial\Omega}(f_nv_m-f_mu_n)\,ds.
\end{equation*}

Throughout, we shall focus on the \textit{finite-dimensional} linear inverse problem,
where the discretization is always assumed to be adequate. We refer interested readers to \cite{PoschlResmeritaScherzer:2010}
for discussions on the interplay between regularization, discretization and noise level.

Last, we introduce the sensitivity matrix $M$ and the data vector $X$.
We use a single index $j=1,\ldots,J$ with  $J=N^2$ for the index pair $(m,n)$
with $j=N(m-1)+ n$, and introduce the frequency-independent sensitivity matrix $M=[M_{jl}]
\in\mathbb{R}^{J\times L}$ with its entries $M_{jl}$ given by
\begin{equation*}
  M_{jl}=\int_{\Omega_l} \nabla v^*_n\cdot\nabla v^*_m \,dx \quad (j\leftrightarrow (m,n)).
\end{equation*}
Likewise, we introduce a vector $X(\omega)\in\mathbb{R}^J$ with its $j$th entry $X_j(\omega)$ given by
\begin{equation*}
  X_j(\omega) = s_0(\omega)^2 \int_{\partial\Omega}(f_nv_m(\omega)-f_mu_n(\omega))\,ds \quad (j\leftrightarrow (m,n)).
\end{equation*}
By writing $A_k = (\delta\sigma_k)_l\in\mathbb{R}^L$, $k=0,\ldots,K$, we obtain a linear system (parameterized by $\omega$)
\begin{equation}\label{eqn:lin-inverse}
  M \sum_{k=0}^Ks_k(\omega) A_k  = X(\omega).
\end{equation}

\begin{remark}
In \eqref{eqn:lin-inverse}, the sensitivity matrix $M$ is identical with that in
static imaging, and hence mfEIT does not lead to improved resolution.
Namely, in mfEIT the diffusive nature of the  modality does not change with
the frequency $\omega$. But as we shall see below,
in the presence of spectral contrast, mfEIT does allow recovering
$\{A_k\}_{k=0}^K$ and removing modeling errors.
\end{remark}

In the mfEIT, $A_k$s are of primary interest.
Depending on the a priori spectral knowledge, we
discuss the following three cases separately: (a) All  $s_k$s are known;
(b) $s_k$s may not be fully known, but with substantially different frequency dependence;
(c) $s_k$s are only partially known, and we aim at a partial recovery of $A_k$s.
They are of different degree of challenge.

\subsection{Case (a): Known Spectral Profiles}\label{sub:case(a)}

First we consider the case when $s_k$s are all known.
In some applications, this is feasible, since the spectral profiles
of many materials can be measured (see e.g.\ \cite{GabrielPeyman:2009} for
tissues). Suppose that we can measure
$X(\omega)$ at $Q$ distinct frequencies $\{\omega_q\}_{q=1}^Q$. By writing
$S=(S_{kq})\in\mathbb{R}^{(K+1)\times Q}$, with $S_{kq}=s_k(\omega_q)$, we get from
\eqref{eqn:lin-inverse}
\begin{equation}\label{eqn:lin-inv-knownprofile}
  MAS = X,
\end{equation}
where the matrix $X=[X(\omega_1)\ \ldots \ X(\omega_Q)]\in\mathbb{R}^{J\times Q}$. In equation
\eqref{eqn:lin-inv-knownprofile}, the matrix $M$ can be precomputed, and the matrix
$S$ and the data $X$ are known: only $A=[A_0\ \ldots\ A_K]\in \mathbb{R}^{L\times(K+1)}$
is unknown. It is natural to assume that a sufficient number of frequencies are taken so that
$S$ is incoherent, namely $Q\geq K+1$ and $\mathrm{rank}(S)=K+1$ (and presumably $S$
is also well-conditioned). Then $S$ admits a right inverse $S^{-1}$. By letting $Y=XS^{-1}$ we obtain
\begin{equation*}
  MA = Y.
\end{equation*}
These are $K+1$ decoupled linear system. By letting
$Y=[Y_0\ \ldots \ Y_K]\in\mathbb{R}^{J\times(K+1)}$, we have
\begin{equation}\label{eqn:Y}
  MA_k=Y_k,\quad k = 0,\ldots,K,
\end{equation}
where $A_k$ represents the $k$th abundance. Here each linear system determines one
and only one abundance $A_k$. The stable and accurate solution of \eqref{eqn:Y}
will be discussed in  Section~\ref{sect:sparsity}.

The condition $\mathrm{rank}(S)=K+1$ is necessary and
sufficient for a full decoupling, and the well-conditioning of $S$
ensures a stable decoupling. It specifies the condition under which the abundance unmixing is practically feasible, and
also the proper selection of $\{\omega_q\}_{q=1}^Q$ such that $\mathrm{rank}(S)=K+1$.
It depends essentially on the incoherence of
$\{s_k(\omega)\}_{k=0}^K$, without which a full decoupling
is impossible. For example, consider the simple case of two endmembers, with $s_0(\omega)=1+\omega$, $s_1(\omega)=2+2\omega$.
Then for any $Q>1$, $S$ is always of rank one.

The right inverse $Y=XS^{-1}$ can also be viewed as a least-squares
procedure
\begin{equation*}
  \min_{Y\in\mathbb{R}^{J\times (K+1)}}\|X-YS\|_F.
\end{equation*}
Thus, for a rank-deficient $S$, our approach
yields the minimum-norm matrix $Y$ compatible with the data, and for an inconsistent $S$,
it yields a best approximation via projection. By the
perturbation theory for least-squares problems \cite{Grcar:2010}, the well-conditioning of $S$
implies that the procedure is stable with respect to small perturbations in the spectral profiles.

This approach generalizes the weighted frequency difference EIT (fdEIT) method
 \cite{SeoLeeKim:2008}.
\begin{exam}\label{ex:fdEIT}
Consider the case with $K=1$ and $Q=2$, i.e., two frequencies. We write
\begin{equation*}
  X = [X(\omega_1) \ X(\omega_2)] \quad \mbox{and}\quad S = \left[\begin{array}{cc}
    s_0(\omega_1) & s_0(\omega_2)\\ s_1(\omega_1) & s_1(\omega_2)
  \end{array}\right].
\end{equation*}
Therefore, if $S$ is invertible, we obtain
\begin{equation*}
  Y = XS^{-1} = \frac{s_0(\omega_1)}{\det S}\left[\frac{s_1(\omega_2)}{s_0(\omega_1)}X(\omega_1)
  -\frac{s_1(\omega_1)}{s_0(\omega_1)}X(\omega_2)\quad X(\omega_2)-\frac{s_0(\omega_2)}{s_0(\omega_1)}X(\omega_1)\right].
\end{equation*}
The second column of $Y$ recovers the weighted fdEIT method \cite{SeoLeeKim:2008}.
Thus our method generalizes \cite{SeoLeeKim:2008}.
Our approach directly incorporates multifrequency data, which
improves the numerical stability, especially in the presence of strong correlation between neighboring frequencies and
imprecisely known spectral profiles. Further, it
enables  decoupling multiple inclusions.
In the special case $s_0(\omega_1)=
s_0(\omega_2)$, it recovers the usual frequency difference. This delineates the region of validity of
frequency difference for multifrequency data.
\end{exam}

\begin{remark}
The minimal number $Q$ of frequencies is equal to $K+1$, provided that with $\{\omega_q\}_{q=1}^Q$,
$S$ is sufficiently incoherent, i.e.\ $\mathrm{rank}(S)=K+1$. With an inadvertently poor choice of
$\{\omega_q\}_{q=1}^Q$, it may require more  than $K+1$  frequencies to achieve the desired incoherence.
\end{remark}

\subsection{Case (b): Spectral Profiles with Substantially Different Frequency Dependence}\label{sub:case(b)}

Next we consider the case when some of (or, possibly, all) $s_k(\omega)$s are not known, but
do not change rapidly with $\omega$, when compared to the remaining ones.  Thus,
instead of using $X(\omega)$ directly, it is natural to differentiate \eqref{eqn:lin-inverse}
with respect to $\omega$ to eliminate the contributions from those $s_k(\omega)$s that do not vary
much with $\omega$. This discriminating effect is useful in practice. For example,
the conductivity of malign tissues is more sensitive with respect to frequency variations in
a certain frequency range \cite{1988-surowiec,LauferIvorra:2010}, even though that of healthy tissues in the
background may exhibit fairly complex structure.

More precisely, let $\mathcal{P}\subseteq \{0,1,\dots,K\}$ be such that
\begin{equation}\label{eqn:ass-freq}
\left|s_p^\prime(\omega_q)\right|\gg \left|s_k^\prime(\omega_q)\right|, \qquad p\in \mathcal{P},k\in \{0,1,\dots,K\}\setminus \mathcal{P}.
\end{equation}

By differentiating \eqref{eqn:lin-inverse} with respect to $\omega$ and
invoking the assumption \eqref{eqn:ass-freq}, we obtain
\begin{equation}\label{eq:caseb}
  M\sum_{p\in \mathcal{P}} A_ps_p^\prime(\omega) \approx X^\prime(\omega).
\end{equation}
Thus the contributions from the remaining profiles are negligible.
Different reconstruction schemes should be used depending on whether the spectral profiles $\{s_p(\omega)\}_{p\in \mathcal{P}}$ are known.

\subsubsection{Case (b1): The Spectral Profiles $\{s_p(\omega)\}_{p\in \mathcal{P}}$ are not Known}\label{subsub:b1}

In the case when the spectral profiles $\{s_p(\omega)\}_{p\in \mathcal{P}}$ are not known,
\eqref{eq:caseb} cannot be simplified further. By solving \eqref{eq:caseb}, we can
recover at most $\sum_{p\in \mathcal{P}} s_p'(\omega) A_p$, namely a linear combination of the inclusions.
Since the weights $\{s_p'(\omega)\}_{p\in\mathcal{P}}$ are unknown, it is impossible to separate $\{A_p,p\in\mathcal{P}\}$. However,
 when $\mathcal{P} = \{p\}$  (i.e., $|\mathcal{P}|=1$), $\delta\sigma_p$
may be recovered up to a multiplicative constant, which gives the support information.
We illustrate the technique with an example.
\begin{exam}
Consider the case $K=1$, and two linear  frequency profiles, i.e.,
$s_0(\omega) = \alpha_0 + \beta_0 \omega$ and $s_1(\omega) = \alpha_1 + \beta_1 \omega$,
with $\beta_0\ll\beta_1$. Then the differentiation imaging amounts to
\begin{equation*}
  \beta_0 MA_0 + \beta_1 MA_1 = X^\prime(\omega).
\end{equation*}
If $MA_0$ and $MA_1$ are comparable, then $\beta_0\ll\beta_1$ implies
that the contribution of $\beta_0MA_0$  to the data is negligible.
Hence, the technique allows to recover the component $\beta_1MA_1$, which
upon linear inversion yields $\beta_1A_1$. In particular, it gives the support
$\mathrm{supp}(A_1)$, and also its magnitude up to a multiplicative constant. Further, for
known $\beta_1$, it allows recovering $A_1$.
\end{exam}

\subsubsection{Case (b2): The Spectral Profiles $\{s_p(\omega)\}_{p\in \mathcal{P}}$ are Known}\label{subsub:b2}

If the spectral profiles $\{s_p(\omega)\}_{p\in \mathcal{P}}$ are known, it is possible to perform
the same analysis of Case (a) to  \eqref{eq:caseb}. Taking
measurements at $Q$ distinct frequencies $\omega_1,\dots,\omega_Q$, we have
\begin{equation*}
  M\sum_{p\in\mathcal{P}} A_p s_p'(\omega_q) \approx X'(\omega_q),\qquad q=1,\dots,Q.
\end{equation*}
Then, with
$\widetilde S=(\widetilde S_{pq})\in\mathbb{R}^{|\mathcal{P}| \times Q}$, $\widetilde S_{pq}=
s_p'(\omega_q)$, $ X'=[X'(\omega_1)\ \ldots \ X'(\omega_Q)]\in\mathbb{R}^{J\times Q}$, we get
$MA\widetilde S =  X'$. Then the inversion
step is completely analogous to that in  Section~\ref{sub:case(a)}, if
$\mathrm{rank} \,\widetilde S = |\mathcal{P}|$ (and well-conditioning). All the
inclusions $A_p$, $p\in \mathcal{P}$, can be recovered.

\subsubsection{Numerical Implementation}

In the implementation, we take
\begin{equation}\label{eqn:freqdiff}
  M\sum_{k=0}^KA_k\frac{s_k(\omega_{q+1})-s_k(\omega_q)}{\omega_{q+1}-\omega_q}=\frac{X(\omega_{q+1})-X(\omega_q)}{\omega_{q+1}-\omega_q}.
\end{equation}
It approximates the derivative $s_k^\prime(\omega_q)$ with the forward difference
$
  s_k^\prime(\omega_q) \approx (s_k(\omega_{q+1})-s_k(\omega_q))/(\omega_{q+1}-\omega_q).
$
One can also use higher-order difference formulas, and they represent different ways to
perform difference imaging. Their robustness with respect to noise
might differ due to the ill-posed nature of numerical differentiation.
In this work, we shall use  \eqref{eqn:freqdiff}.

\subsection{Case (c): Partially Known Spectral Profiles, Partial Recovery of the Abundances}\label{sub:case(c)}

In practice, it is also of interest to recover some information about $\{A_k\}$ when $\{s_k(\omega)\}$
are only partially known. Generally, this is infeasible. But, one can still obtain some information under
certain a priori knowledge. To this end, recall the notation $Y_k=MA_k$, cf. \eqref{eqn:Y}. Then
\begin{equation}\label{eqn:YQ}
  Y_0 s_0(\omega_q) + \ldots + Y_K s_K(\omega_q) = X(\omega_q),\qquad q=1,\dots,Q.
\end{equation}
Now suppose the frequency dependence of $\{s_k(\omega)\}_{k=0}^K$ are of polynomial
type, namely $s_k(\omega) = \sum_{n=0}^N\alpha_k^n\omega^n.$
Inserting this expression into \eqref{eqn:YQ} yields
$
  \sum_{n=0}^N\sum_{k=0}^K(\alpha_k^n Y_k)\omega^n = X(\omega).
$
By taking a sufficiently large number of modulating frequencies $\{\omega_q\}_{q=1}^Q$ (to
be more precise, $Q\geq N+1$), and using the
identity principle for polynomials, we can compute
$B_n:= \sum_{k=0}^K \alpha_k^n Y_k$, $n=0,\ldots, N.$
Note that adding more frequencies does not tell more about $Y_k$ and $\alpha_k^j$ than
$\{B_n\}_{n=0}^N$. Namely, $\{B_n\}_{n=0}^N$ contain the essential information in
$\{X(\omega_q)\}_{q=1}^Q$. Depending on $K$, $N$ and further a prior knowledge,
some $Y_k$ can be recovered without knowing the corresponding spectral profiles.
Instead of providing a  general analysis of all possible cases, we present two examples that
explain the different situations that may appear.
\begin{exam}\label{exam30}
Consider the case $K=1$. For every $n$ we have $B_0 = \alpha_0^0 Y_0 + \alpha^0_1 Y_1$ and $B_n = \alpha_0^n Y_0 + \alpha^n_1 Y_1$, whence
$
Y_1 = (\alpha^0_0\alpha_1^n - \alpha^0_1\alpha^n_0)^{-1}(\alpha^0_0 B_n - \alpha^n_0 B_0).
$
Since $s_0$ is known, so are $\alpha^0_0$
and $ \alpha^n_0$. Hence, $Y_1$  may be recovered up to a multiplicative constant $c$, if
$\alpha^0_0\alpha_1^n - \alpha^0_1\alpha^n_0\neq 0$, without
assuming any knowledge of $s_1(\omega)$. This condition
simply represents the incoherence between $s_0$ and $s_1$. Finally, by solving
$MA_1 = cY_1$, $\delta\sigma_1$ can be recovered up
to some constant.

Further, assuming a unique recovery of the linearized inverse problem, the
knowledge of $B_0$ allows recovering an unknown
linear combination of $A_0$ and $A_1$, especially the union of their supports.
Since the supports of $A_0$ and $A_1$ are assumed to be disjoint from
each other, this  allows recovering $\mathrm{supp}(A_0)$,
given that $\mathrm{supp}(A_1)$ has already been recovered.
\end{exam}

\begin{exam}
Note that if $K=2$ and $N=1$, we get only
\begin{equation*}
  \alpha_0^0 Y_0 + \alpha_1^0 Y_1 +\alpha_2^0 Y_2 = B_0 \quad\mbox{and}\quad
  \alpha_0^1Y_0 + \alpha_1^1Y_1 + \alpha_2^1Y_2 = B_1,
\end{equation*}
which is vastly insufficient to determine all the unknowns. However, a calculation similar to
Example \ref{exam30} shows that the support of $Y_2$ can be
determined if $K=N=2$ and $s_1$ is known, if a
certain nonzero condition is satisfied. Like before, by solving the underdetermined system
$MA_2 = cY_2$, we can recover the support of $\delta\sigma_2$. Further, assuming a unique recovery with the
linearized inverse problem, $\mathrm{supp}(\delta\sigma_0)$ and $\mathrm{supp}(\delta\sigma_1)$ may be
determined.
\end{exam}

With obvious modifications, the preceding discussion is also valid for more
general basis functions $\phi_n(\omega)$ which form a unisolvent system on the set  $\{\omega_q\}_{q=1}^Q$ \cite[pp. 31--32]{Davis:1963}.

\section{Imperfectly Known Boundary}\label{sec:unknownboundary}

Now we illustrate the potentials of mfEIT for
handling modeling errors, e.g., an imperfectly known boundary,
which has long been one of the main
obstacles in practice \cite{AdlerGuardo:1996,KolehmainenLassas:2005,
KolehmainenLassasOla:2008}. The use of a slightly incorrect
boundary can lead to large reconstruction errors, and mfEIT is one strategy
to overcome the challenge \cite{SeoLeeKim:2008}.
Here we present an analysis of the approach in the linearized regime to justify these numerical findings.

We denote the true but unknown physical domain by $\widetilde \Omega$, and the
computational domain by $\Omega$.
Next we introduce a forward map $F:\widetilde \Omega\to \Omega$, $\widetilde x\to x$, which
is assumed to be a smooth orientation preserving map with a smooth inverse map $F^{-1}: \Omega
\to\widetilde \Omega$. We denote  the Jacobian of the map $F$ by $J_F$, and the Jacobian
of $F$ with respect to the surface integral by $J^S_{F}$.

Now suppose that the function $\widetilde u_n(\widetilde x,\omega)\in H^1_\diamond(\widetilde\Omega)$
satisfies \eqref{eqn:eit} in the true domain $\widetilde \Omega$ with a conductivity
$\widetilde\sigma(\widetilde x,\omega)$ and input current $\widetilde f_n\in L^2_\diamond(\partial\widetilde\Omega)$, namely
\begin{equation}\label{eqn:eit-deformed}
  \left\{\begin{aligned}
    -\nabla_{\widetilde x}\cdot(\widetilde\sigma(\widetilde x,\omega)\nabla_{\widetilde x}\widetilde u_n(\widetilde x,\omega)) & = 0\quad \mbox{ in }\widetilde\Omega,\\
    \widetilde \sigma(\widetilde x,\omega)\frac{\partial \widetilde u_n(\widetilde x,\omega)}{\partial\widetilde\nu}& = \widetilde f_n\quad \mbox{on }\partial\widetilde\Omega,
  \end{aligned}\right.
\end{equation}
Here the conductivity $\widetilde\sigma(\widetilde x,\omega)$ takes a separable form (cf. \eqref{eqn:gamma})
\begin{equation}\label{eqn:sigma-tilde}
  \widetilde\sigma(\widetilde x,\omega) = \sum_{k=0}^K s_k(\omega)\widetilde \sigma_k(\widetilde x),
\end{equation}
with $\widetilde \sigma_0(\widetilde x) = 1 + \delta \widetilde \sigma_0(\widetilde x)$, and $\widetilde \sigma_k(\widetilde x) = \delta
\widetilde \sigma_k(\widetilde x)$, $k=1,\ldots,K$, where $\delta\widetilde \sigma_k$ are small and their supports are disjoint and stay
away from $\partial\widetilde\Omega$.
The weak formulation (by suppressing the dependence on
$\omega$) is given by: find $\widetilde u_n(\cdot,\omega) \in H_\diamond^1(\widetilde\Omega)$ such that
\begin{equation}\label{eqn:un-perturbed}
  \int_{\widetilde\Omega}\widetilde\sigma\nabla_{\widetilde x} \widetilde u_n \cdot\nabla_{\widetilde x} \widetilde v d\widetilde x = \int_{\partial\widetilde\Omega}\widetilde f_n\widetilde v d\widetilde s,\qquad  \widetilde v\in H^1(\widetilde\Omega).
\end{equation}

Let us now discuss the experimental setup. The practitioner chooses a current density
$f_n\in L_\diamond^2(\partial\Omega)$ defined on $\partial\Omega$. It is then applied
to the unknown boundary $\partial\widetilde\Omega$. The applied
current $\widetilde f_n$ on $\partial\widetilde\Omega$ results to be
\begin{equation}\label{eq:transform_f}
\widetilde f_n = (f_n\circ F)|\det J_F^S|.
\end{equation}
This implies $\int_{\partial\widetilde\Omega}\widetilde f_nd\widetilde s =0$, whence
problem \eqref{eqn:eit-deformed} is well-posed.
This induces the potential $\widetilde u_n \in H_\diamond^1(\widetilde\Omega)$ given by
 \eqref{eqn:un-perturbed}, which should be measured on $\partial\widetilde\Omega$. However, due to the incorrect
knowledge of the boundary, the measured quantity is in fact  $u_n:=\widetilde u_n \circ F^{-1}$ restricted
to  $\partial\Omega$.

\begin{remark}
The current density on $\partial\widetilde\Omega$
is locally defined by $\widetilde J=I/{\rm area}(\widetilde A)$, where $I$ is the current injected through a
small surface $\widetilde A\subseteq \partial\widetilde\Omega$. Thus
$
\widetilde J =\frac{I}{{\rm area}(\widetilde A)} = \frac{I}{{\rm area}(A)}\frac{{\rm area}(A)}{{\rm area}(\widetilde A)} = J\, \frac{{\rm area}(A)}{{\rm area}(\widetilde A)},
$
where $J$ is the current density on $A:=F(\widetilde A)\subseteq\partial\Omega$. Hence,
$|\det J_F^S|$ is the infinitesimal version of $\frac{{\rm area}(A)}{{\rm area}
(\widetilde A)}$. Since
$\int_{\partial\widetilde\Omega} \widetilde f_n \widetilde u_n d\widetilde s = \int_{\partial\Omega} f_nu_n\,ds$
and $\int_{\partial\Omega}f_nu_n\,ds$ denotes the power needed to maintain the
potential $u_n$ on $\partial\Omega$, the choice \eqref{eq:transform_f}
preserves the needed power for the data.
\end{remark}

We consider only the case that  $\Omega$ is a small variation of $\widetilde\Omega$ (but comparable with
$\delta\sigma_k$s) so that the linearized
regime is valid. We write $F:\widetilde\Omega\to\Omega$ by $F(\widetilde x)=\widetilde x + \epsilon
\widetilde\phi(\widetilde x)$, where $\epsilon>0$ is small
and the smooth function $\widetilde\phi(\widetilde x)$ characterizes the deformation. Let $F^{-1}(x)
=x+\epsilon \phi(x)$ be the inverse, which is also smooth.
To examine its influence on the linearized inverse problem, we introduce
 $v_m\in H_\diamond^1(\Omega)$ corresponding to
$\sigma_0(x,\omega)=s_0(\omega)$ in $\Omega$ and the flux $f_m$, i.e.,
\begin{equation}\label{eqn:vs-perturbed}
\int_\Omega \sigma_0\nabla v_m \cdot\nabla v\,dx = \int_{\partial\Omega}f_mv\,ds,\qquad  v\in H^1(\Omega).
\end{equation}

We can now state the corresponding linearized inverse problem.
The proof shows that even for an isotropic $\widetilde\sigma$ in $\widetilde\Omega$,
cf. \eqref{eqn:sigma-tilde}, in $\Omega$ the equivalent $\sigma$ is
generally anisotropic.

\begin{proposition}\label{prop:linear-perturb}
Set $\delta\sigma_k = \delta \widetilde \sigma_k \circ F^{-1}$ for $k=0,1,\ldots,K$ and $v^*_m =  s_0(\omega) v_m$ for $m=1,\dots,N$.
The linearized inverse problem on the computational domain $\Omega$ is given by
\begin{equation}\label{eqn:int-domaindeform}
  s_0(\omega)\epsilon\int_\Omega\!\Psi\nabla v^*_n\cdot\nabla v^*_m\,dx+ \sum_{k=0}^Ks_k(\omega)\!\int_{\Omega}\delta\sigma_k\nabla
  v^*_n\cdot\nabla v^*_m\,dx =\! s_0(\omega)^2\int_{\partial\Omega}(f_nv_m-f_mu_n)\,ds,
\end{equation}
for some smooth function $\Psi:\Omega\to\mathbb{R}^{d\times d}$, which is independent of the frequency $\omega$.
\end{proposition}
\begin{proof}
First, we derive the governing equation for the variable $u_n=\widetilde u_n\circ F^{-1}$
in the domain $\Omega$ from \eqref{eqn:un-perturbed}. Denote by $v=\widetilde v\circ F^{-1} \in H^1(\Omega)$.
By the chain rule we have $\nabla_{\widetilde x}\widetilde u_n\circ F^{-1}=(J_F^t\circ F^{-1})\nabla_x
u_n$, where the superscript $t$ denotes the matrix transpose. Thus, we deduce
\begin{equation*}
   \quad \int_{\widetilde\Omega}\widetilde \sigma\nabla_{\widetilde x} \widetilde u_n
   \cdot\!\nabla_{\widetilde x} \widetilde v d\widetilde x  =  \int_\Omega
  \sigma  \nabla u_n\cdot  \nabla v\,dx,
  \end{equation*}
  where the transformed conductivity $\sigma(x,\omega)$ is given by \cite{Sylvester:1990,KolehmainenLassas:2005,KolehmainenLassasOla:2008}
\begin{equation}\label{eqn:cond-aniso}
  \sigma(x,\omega) = \left(\frac{J_F(\cdot)\widetilde\sigma(\cdot,\omega)J_F^t(\cdot)}{|\det J_F(\cdot)|}\circ F^{-1}\right)(x).
\end{equation}
Moreover, by \eqref{eq:transform_f} we have
$\int_{\partial\widetilde \Omega}\widetilde f_n \widetilde vd\widetilde s  = \int_{\partial\Omega}f_nv \,ds.$
Hence, by  \eqref{eqn:un-perturbed} the potential $u_n$ satisfies
\begin{equation}\label{eqn:us-deformed}
\int_\Omega \sigma\nabla u_n \cdot\nabla v\,dx = \int_{\partial\Omega}f_nv\,ds, \qquad  v\in H^1(\Omega).
\end{equation}
Then by choosing $v=v_m$ in \eqref{eqn:us-deformed} and $v=u_n$ in \eqref{eqn:vs-perturbed}, we arrive at
\begin{equation}\label{eqn:uv}
  \int_\Omega (\sigma-\sigma_0)\nabla u_n\cdot\nabla v_m \,dx = \int_{\partial\Omega}(f_nv_m-f_mu_n)\,ds.
\end{equation}
Note that $J_F= I + \epsilon J_{\widetilde\phi}$, and
$J_{F^{-1}}=I+\epsilon J_\phi = I-\epsilon J_{\widetilde\phi}\circ F^{-1}+O(\epsilon^2)$, since $\epsilon$ is small.
Since $|\det J_F|= 1+\epsilon\mathrm{div}\widetilde\phi +O(\epsilon^2)$
\cite[equation (2.10)]{Hettlich:1995}, $\sigma(x,\omega)$
can be written as
\begin{equation*}
  \begin{split}
    \sigma(x,\omega) & = \widetilde\sigma(\cdot,\omega)(1+\epsilon\mathrm{div}\widetilde\phi(\cdot))^{-1}(I+\epsilon(J_{\widetilde\phi}(\cdot)+J_{\widetilde\phi}^t(\cdot)))\circ F^{-1}(x)+O(\epsilon^2)\\
    &= \widetilde\sigma(\cdot,\omega)((1-\epsilon\mathrm{div}\widetilde\phi(\cdot))I + \epsilon(J_{\widetilde\phi}(\cdot)+J^t_{\widetilde\phi}(\cdot)))\circ F^{-1}(x) + O(\epsilon^2) \\
    &= \widetilde\sigma(\cdot,\omega) \circ F^{-1}(x) + \Psi(x) \epsilon + O(\epsilon^2),
  \end{split}
\end{equation*}
where $\Psi=(J_{\widetilde\phi}+J_{\widetilde\phi}^t- \mathrm{div}\widetilde\phi I)\circ F^{-1}$ is
independent of $\omega$. Thus we obtain
\begin{equation}\label{eqn:cond-deformed}
  \sigma(x,\omega) \approx s_0(\omega)I +\epsilon s_0(\omega)\Psi(x)  +\sum_{k=0}^K\delta\sigma_k(x)s_k(\omega)I.
\end{equation}
Substituting it into \eqref{eqn:uv} and invoking the approximation
$\nabla u_n\approx\nabla v_n$ complete the proof.
\end{proof}

By Proposition~\ref{prop:linear-perturb}, in the presence of an imperfectly known boundary with the deformation magnitude
$\epsilon$ comparable with $\{\delta\sigma_k\}_{k=0}^K$, there is a dominant source of
errors in \eqref{eqn:int-domaindeform}: it contains an additional anisotropic
component $\epsilon \Psi$. Thus a direct inversion of \eqref{eqn:int-domaindeform} is unsuitable. This explains
the observation that a slightly incorrect boundary can lead to erroneous
recoveries \cite{AdlerGuardo:1996,GersingHofmannOsypka:1996}.
This issue can be resolved by mfEIT. Indeed, by rearranging  \eqref{eqn:int-domaindeform} we obtain
\begin{equation}\label{eqn:int-domaindeform-2}
  s_0(\omega)\int_\Omega(\epsilon\Psi+\delta\sigma_0)\nabla v^*_n\cdot\nabla v^*_m\,dx + \sum_{k=1}^K
  s_k(\omega)\int_{\Omega}\delta\sigma_k\nabla v^*_n\cdot\nabla v^*_m\,dx
  =s_0(\omega)^2\! \int_{\partial\Omega}(f_nv_m-f_mu_n)\,ds.
\end{equation}
This is analogous to \eqref{eqn:integral}, with the only difference lying in the
extra term $\epsilon\Psi$. Hence, all methods in  Section~\ref{sect:continuum}
are applicable. The inclusion
$\delta\sigma_0$ will never be properly recovered, due to the pollution of the term $\epsilon\Psi$.
However, $\{\delta\sigma_k(\omega)\}_{k=1}^K$ may be reasonably recovered, since
they are affected slightly by the deformation only through $\delta\sigma_k = \delta \widetilde \sigma_k
\circ F^{-1}$. Thus, mfEIT can effectively eliminate modeling errors caused by the boundary
uncertainty.

\section{The Complete Electrode Model}\label{sec:cem}
In this section we adapt the approach discussed in Sections~\ref{sect:continuum} and \ref{sec:unknownboundary} to the more
realistic complete electrode model (CEM).

\subsection{Perfectly Known Boundary}
First we consider  the case of a perfectly known boundary.
Let $\Omega$ be an open bounded domain in $\mathbb{R}^d$ ($d=2,3$), with a smooth boundary
$\partial\Omega$. We denote the set of electrodes by $\{e_j\}_{j=1}^E\subset\partial\Omega$, which are disjoint from
each other, i.e., $\bar{e}_i\cap\bar{e}_k=\emptyset$ if $i\neq k$. The applied current on the $j$th electrode $e_j$ is
denoted by $I_j$, and the current vector $I=(I_1,\ldots,I_E)^{t}$ satisfies $\sum_{j=1}^EI_j=0$ by the law of
charge conservation. Let the space $\mathbb{R}_\diamond^E$ be the subspace of $\mathbb{R}^E$ with zero
mean, i.e., $I\in\mathbb{R}_\diamond^E$. The electrode voltages $U=(U_1,\ldots,U_E)^{t}$ are also grounded so
that $U\in\mathbb{R}_\diamond^E$. Then the CEM reads
\cite{ChengIsaacsonNewellGisser:1989,SomersaloCheneyIsaacson:1992}: given the conductivity
$\sigma(x,\omega)$, positive
contact impedances $\{z_j\}_{j=1}^E$ and an input current $I\in\mathbb{R}_\diamond^E$, find the potential $u(\cdot,\omega)\in
H^1(\Omega)$ and the electrode voltages $U\in\mathbb{R}_\diamond^E$ such that
\begin{equation}\label{eqn:cem}
\left\{\begin{aligned}
-\nabla\cdot(\sigma(x,\omega)\nabla u(x,\omega))&=0  \mbox{ in }\Omega,\\ 
u+z_j\frac{\partial u}{\partial \nu_\sigma}&=U_j \mbox{ on } e_j,\; j=1,2,\ldots,E,\\ 
\int_{e_j}\frac{\partial u}{\partial \nu_\sigma}\,ds &=I_j \mbox{ for } j=1,2,\ldots, E,\\ 
\frac{\partial u}{\partial \nu_\sigma}&=0\mbox{ on } \partial\Omega\backslash\cup_{j=1}^E e_j,
\end{aligned}\right.
\end{equation}
where $\frac{\partial u}{\partial \nu_\sigma}$ denotes the co-normal derivative
 $(\sigma\nabla u)\cdot \nu$. The second line describes the contact impedance effect. In practice, the contact
impedances $\{z_j\}_{j=1}^E$ are observed to be inversely proportional to the conductivity of the object
\cite{Holm:1967,Hwang:1997}, i.e.,
\begin{equation}\label{eq:zj}
   z_j = s_0(\omega)^{-1} c_j,
\end{equation}
for some constants $c_j > 0$ independently of $\omega$, since by assumption, near $\partial\Omega$
we have $\sigma(x,\omega) = s_0(\omega)$. The weak formulation
is given by: find $(u,U)\in\mathbb{H} := H^1(\Omega)\times\mathbb{R}_\diamond^E$
such that \cite{GehreJinLu:2014}
\begin{equation*}
  \int_{\Omega}\sigma\nabla u\cdot\nabla v\,dx  + \sum_{j=1}^Ez_j^{-1}
  \int_{e_j}(u-U_j)(v-V_j)\,ds = \sum_{j=1}^EI_jU_j,   \qquad (v,V)\in\mathbb{H}.
\end{equation*}
The bilinear form defined on the left hand side is coercive and continuous on
$\mathbb{H}$, and thus by Lax-Milgram theorem there exists a unique solution $(u(\cdot,\omega),U(\omega))\in \mathbb{H}$.

Consider $N$ input currents $\{I_n\}_{n=1}^N\subset\mathbb{R}_\diamond^E$, and let $\{(u_n,U_n)\}_{n=1}^N\subset \mathbb{H}$
be the corresponding solutions \eqref{eqn:cem}, i.e., for all $(v,V)\in \mathbb{H}$
\begin{equation}\label{eqn:us-cem}
\int_\Omega \sigma\nabla u_n\cdot\nabla v \,dx + \sum_{j=1}^E z_j^{-1}\int_{e_j}(u_n-U_{n,j})(v-V_{j})\,ds = \sum_{j=1}^EI_{n,j}V_{j}.
\end{equation}
The electrode voltages $U_n\in\mathbb{R}_\diamond^E$ can be measured, and are used to recover the
conductivity $\sigma(x,\omega)$. To derive a linearized model, let $(v_m,V_m)\in \mathbb{H}$ be
the solution corresponding to the reference conductivity $\sigma_0(x,\omega)=s_0(\omega)$: for every $(v,V)\in \mathbb{H}$ we have
\begin{equation}\label{eqn:vs-cem}
\int_\Omega \sigma_0\nabla v_m\cdot\nabla v \,dx
 + \sum_{j=1}^E z_j^{-1}\int_{e_j}(v_m-V_{m,j})(v-V_j)\,ds = \sum_{j=1}^EI_{m,j}V_j.
\end{equation}
By  \eqref{eq:zj}, we can write $(v^*_m,V^*_m) = s_0(\omega)(v_m,V_m)$
for the solution $(v^*_m,V^*_m)$ corresponding to $\sigma_0 \equiv 1$.
Now we assume that $\sigma(x,\omega)$ follows
\eqref{eqn:gamma}. Using \eqref{eqn:us-cem} and \eqref{eqn:vs-cem},
we deduce
\begin{equation*}
  \sum_{k=0}^Ks_k(\omega)\int_\Omega \delta\sigma_k\nabla u_n\cdot\nabla v_m\,dx = \sum_{j=1}^E(I_{n,j}V_{m,j}-I_{m,j}U_{n,j}).
\end{equation*}
Then, under the approximation $\nabla u_n\approx\nabla v_n$ in the
domain $\Omega$, and the approximation \eqref{eqn:piece-const} of the inclusions
$\delta\sigma_k$s on the triangulation $\{\Omega_l\}_{l=1}^L$, we have
\begin{equation}\label{eqn:cem_lin}
  \sum_{k=0}^K s_k(\omega)\sum_{l=1}^L(\delta\sigma_k)_l\int_{\Omega_l}\nabla v^*_n\cdot
  \nabla v^*_m\,dx = s_0(\omega)^2 \sum_{j=1}^E(I_{n,j}V_{m,j}-I_{m,j}U_{n,j}).
\end{equation}
This formula is almost identical with \eqref{eqn:integral}, and formally their only
difference lies in the computation of $X(\omega)$. Hence, all the discussions in
Section~\ref{sect:continuum} can be adapted to the CEM \eqref{eqn:cem}. In particular,
all inversion methods therein can be directly applied  to this case.

\subsection{Imperfectly Known Boundary}\label{sub:cem-unknown}
Now we consider the case of an imperfectly known boundary. Like before, let $\widetilde\Omega$
be the unknown true domain with a smooth boundary $\partial\widetilde\Omega$, and $\Omega$ be the computational domain
with a smooth boundary $\partial\Omega$. Accordingly, let $\{\widetilde e_j\}_{j=1}^E\subset\partial\widetilde
\Omega$ and $\{e_j\}_{j=1}^E\subset \partial\Omega$ be the real and computational electrodes, respectively, and
assume they satisfy the usual conditions discussed above. Then we introduce a smooth orientation
preserving forward map $F:\widetilde \Omega\to \Omega$,  with a smooth inverse $F^{-1}: \Omega \to
\widetilde \Omega$, and we denote  the restriction of $F$ to the boundary $\partial\widetilde \Omega$ by $f:\partial\widetilde \Omega\to
\partial\Omega$. We write $F^{-1}(x)=x+\epsilon\phi(x)$, where $\epsilon>0$ denotes the deformation magnitude.
Further, it is assumed that
there is no further electrode movement, i.e., $e_j=f(\widetilde e_j)$, $j=1,\ldots,E$.  With the
conductivity $\widetilde\sigma(\widetilde x,\omega)$ of the
form \eqref{eqn:sigma-tilde} and input current $I_n\in\mathbb{R}_\diamond^E$, by \eqref{eqn:cem}, the quantity $(\widetilde
u_n(\widetilde x,\omega),\widetilde U_n(\omega))\in\widetilde{\mathbb{H}}\equiv H^1(\widetilde\Omega)\times \mathbb{R}_\diamond^E$ satisfies
\begin{equation}\label{eqn:cem-deformed}
  \left\{\begin{aligned}
    -\nabla_{\widetilde x}\cdot(\widetilde\sigma(\widetilde x,\omega)\nabla_{\widetilde x} \widetilde u_n(\widetilde x,\omega)) &= 0\quad \mbox{in } \widetilde\Omega,\\
    \int_{\widetilde e_j}\frac{\partial \widetilde u_n}{\partial \widetilde\nu_{\widetilde\sigma}}d\widetilde{s}&=I_{n,j}\quad \mbox{on }\widetilde e_j,j=1,2,\ldots,E,\\
    z_j\frac{\partial \widetilde u_n}{\partial\widetilde \nu_{\widetilde\sigma}} +\widetilde u_n &= \widetilde U_{n,j} \quad \mbox{on }\widetilde e_j,j=1,2,\ldots,E,\\
    \frac{\partial\widetilde u_n}{\partial\widetilde \nu_{\widetilde\sigma}} & = 0 \quad \mbox{on } \partial\widetilde\Omega\setminus \cup_{j=1}^E\widetilde e_j.
  \end{aligned}\right.
\end{equation}
The weak formulation is given by: find $(\widetilde u_n,\widetilde U_n)\in\widetilde{\mathbb{H}}$ such that for every $(\widetilde v,\widetilde V)\in \widetilde{\mathbb{H}}$
\[
  \int_{\widetilde\Omega} \widetilde \sigma\nabla_{\widetilde x}\widetilde u_n
  \cdot \nabla_{\widetilde x} \widetilde v \,d\widetilde x
  + \sum_{j=1}^Ez_j^{-1}\int_{\widetilde e_j}(\widetilde u_n-\widetilde U_{n,j})(\widetilde
  v-\widetilde V_j)\,d\widetilde s = \sum_{j=1}^E  I_{n,j}\widetilde V_j.
\]
In the experimental setting, on $\Omega$, the injected current $I_n\in\mathbb{R}_\diamond^E$
on the electrodes $\{e_j\}_{j=1}^E$ is known, and the corresponding voltage $\widetilde U_n(\omega)\in
\mathbb{R}_\diamond^E$ can be measured. The inverse problem is to recover
$\{\delta\widetilde\sigma_k\}_{k=0}^K$ from the voltages $\{\widetilde U_n(\omega)\}_{n=1}^N\subset\mathbb{R}_\diamond^E$ at $\{\omega_q\}_{q=1}^Q$.

Now we can state the corresponding linearized inverse problem for \eqref{eqn:cem-deformed}.
Consider the potential $u_n(\cdot,\omega)=\widetilde{u}_n
(\cdot,\omega)\circ F^{-1}$, and the associated electrode voltages $U_n=\widetilde U_n$.

\begin{proposition}\label{prop:cem-deformed}
Let the reference solutions $(v_m,V_m)\in \mathbb{H}$ be defined by \eqref{eqn:vs-cem} and the conductivity
$\widetilde\sigma$ be of the form \eqref{eqn:sigma-tilde}. Set $z=|\det J_{F^{-1}}^S|$, $\delta\sigma_k = \delta \widetilde \sigma_k
\circ F^{-1}$ for  $k=0,1,\ldots,K$ and $(v^*_m,V^*_m) = s_0(\omega)(v_m,V_m)$ for $m=1,\dots,N$. The linearized inverse problem on $\Omega$ is given by
\begin{multline}\label{eqn:lin-inv-cem-deformed}
   s_0(\omega)\epsilon\int_\Omega \Psi \nabla v^*_n\cdot\nabla v^*_m\,dx + \sum_{k=0}^Ks_k(\omega)\int_{\Omega}\delta\sigma_k\nabla v^*_n\cdot\nabla v^*_m\,dx\\
    = s_0(\omega)^2 \sum_{j=1}^E (I_{n,j}V_{m,j}-I_{m,j}U_{n,j})
    -s_0(\omega) \sum_{j=1}^E c_j \int_{e_j}(z-1)\left(\frac{\partial v^*_m}{\partial\nu}\right)^2 ds,
\end{multline}
for some smooth function $\Psi:\Omega\to\mathbb{R}^{d\times d}$, which is independent of the frequency $\omega$.
\end{proposition}
\begin{proof}
Proceeding as in the proof of Proposition~\ref{prop:linear-perturb}, by a change of variables
(and suppressing the variable $\omega$), since $e_j=f(\widetilde e_j)$ we deduce
\[
    \int_{\widetilde\Omega}\! \widetilde \sigma \nabla_{\widetilde x}\widetilde u_n\cdot
    \nabla_{\widetilde x} \widetilde v  d\widetilde x
     = \int_\Omega\! (\widetilde\sigma\circ {F^{-1}}) (J_F^t\circ F^{-1})\nabla u_n\cdot (J_F^t\circ
    F^{-1})\nabla v|\!\det\! J_{F^{-1}}|dx
\]
  and
  \begin{equation*}
    \int_{\widetilde e_j}(\widetilde u_n-\widetilde U_{n,j})(\widetilde v-\widetilde V_j)d\widetilde s  = \int_{e_j}(u_n-U_{n,j})(v-V_j)|\det J_{F^{-1}}^S|\,ds,
\end{equation*}
where  $v=\widetilde v\circ F^{-1} \in H^1(\Omega)$ and $V_j = \widetilde V_j$. Hence, $(u_n(\cdot,\omega),U_n(\omega))$
 satisfies for every $(v,V)\in \mathbb{H}$
\begin{equation*}
  \int_\Omega \sigma\nabla u_n\cdot\nabla v\,dx + \sum_{j=1}^Ez_j^{-1}\int_{e_j}(u_n-U_{n,j})(v-V_j)z\,\,ds = \sum_{j=1}^EI_{n,j}V_j,
\end{equation*}
where $\sigma(x,\omega)$ is given by \eqref{eqn:cond-aniso}.
By combining this identity with \eqref{eqn:vs-cem}, we obtain
\begin{equation*}
  \int_\Omega (\sigma-\sigma_0)\nabla u_n\cdot\nabla v_m\,dx  = \sum_{j=1}^E(I_{n,j}V_{m,j}-I_{m,j}U_{n,j})
  + \sum_{j=1}^E\int_{e_j}(z-1)(u_n-U_{n,j})\frac{\partial v_m}{\partial\nu_{\sigma_0}}\,ds.
\end{equation*}
In view of  \cite{Hettlich:1995,Hettlich:1998},
$
z= 1 + \epsilon (\mathrm{Div}\phi_t -(d-1)H\phi_\nu) + O(\epsilon^2),
$
where $\mathrm{Div}$ denotes the surface divergence, $\phi_t$ and $\phi_\nu$ denote the tangential
and normal components of the vectorial function $\phi$ on $\partial\Omega$, respectively,
 and $H$ is the mean curvature of $\partial\Omega$. In particular, $z-1 = O(\epsilon)$. Thus, by linearization we have
\begin{equation*}
\int_{e_j}(z-1)(u_n-U_{n,j})\frac{\partial v_m}{\partial\nu_{\sigma_0}}\,ds \approx\!
\int_{e_j}\!(z-1)(v_n-V_{n,j})\frac{\partial v_m}{\partial\nu_{\sigma_0}}\,ds
=- z_j \int_{e_j}\!(z-1)\left(\frac{\partial v_m}{\partial\nu_{\sigma_0}}\right)^2 \!ds.
\end{equation*}
Inserting this approximation in the above identity we obtain
\begin{equation*}
  \int_\Omega (\sigma-\sigma_0)\nabla u_n\cdot\nabla v_m\,dx  = \sum_{j=1}^E(I_{n,j}V_{m,j}-I_{m,j}U_{n,j})
  - \sum_{j=1}^E z_j \int_{e_j}(z-1)\left(\frac{\partial v_m}{\partial\nu_{\sigma_0}}\right)^2 ds.
\end{equation*}
Using \eqref{eq:zj}, the rest of the proof follows as in Proposition~\ref{prop:linear-perturb}, and thus it is omitted.
\end{proof}

By proceeding as in the continuum model, we can rewrite \eqref{eqn:lin-inv-cem-deformed} as
\begin{multline}\label{eqn:lin-inv-cem-deformed-2}
   s_0(\omega)\int_\Omega (\epsilon\Psi+ \delta\sigma_0) \nabla v^*_n\cdot\nabla v^*_m\,dx + \sum_{k=1}^Ks_k(\omega)
   \int_{\Omega}\delta\sigma_k\nabla v^*_n\cdot\nabla v^*_m\,dx\\
    = s_0(\omega)^2 \sum_{j=1}^E (I_{n,j}V_{m,j}-I_{m,j}U_{n,j})
    -s_0(\omega) \sum_{j=1}^E c_j \int_{e_j}(z-1)\left(\frac{\partial v^*_m}{\partial\nu}\right)^2 \,ds.
\end{multline}
When compared with \eqref{eqn:int-domaindeform-2},
we observe the presence of the additional error term $s_0(\omega)C_m$, where
$
C_m := -  \sum_{j=1}^E c_j \int_{e_j}(z-1)\left(\frac{\partial v^*_m}{\partial\nu}\right)^2 \,ds,
$
which comes from the boundary deformation. The formula \eqref{eqn:lin-inv-cem-deformed-2} is
consistent with \eqref{eqn:int-domaindeform-2}:
in the continuum case, the contact impedance effect is not present, and $u_n = U_n$ on the
electrodes, namely $c_j=0$, whence $C_m=0$.

All the preceding analysis easily carries over to the case $c_j>0$. Before treating
the general case, let us consider the simple scenario where $z\equiv 1$ on the electrodes $\cup_j e_j$.

\begin{exam}\label{ex:z=1}
Recall that $z(x)=|\det J_{F^{-1}}^S(x)|$ for $x\in\partial\Omega$. Physically, the factor $z$ represents
the length/area deformation relative to the map $F^{-1}\colon \partial\Omega\to \partial\widetilde\Omega$. Thus,
it may be reasonable to assume that the parametrization of the electrodes $\{e_j\}_{j=1}^E$ is known,
which implies $z\equiv1$ on the electrodes $\cup_j e_j$.
Then we have $C_m\equiv 0$, whence
\[
   s_0(\omega)\!\int_\Omega\! (\epsilon\Psi+ \delta\sigma_0) \nabla v^*_n\cdot\nabla v^*_m\,dx + \sum_{k=1}^Ks_k(\omega)
   \!\int_{\Omega}\delta\sigma_k\nabla v^*_n\cdot\nabla v^*_m\,dx = s_0(\omega)^2 \sum_{j=1}^E (I_{n,j}V_{m,j}-I_{m,j}U_{n,j}).
\]
This identity is similar to \eqref{eqn:int-domaindeform-2}, and the comments on the
recovery issue remain valid, since the right hand side is known. Thus, by applying
any of the techniques in Section~\ref{sect:continuum}, it is
possible to eliminate the error  $\epsilon\Psi$ due to the domain deformation, as this affects
only the inclusion $\delta\sigma_0$. All the other inclusions $\{\delta\sigma_k\}_{k=1}^K$
 may be successfully recovered.
\end{exam}

Now we consider the general case $z\not \equiv 1$ on $\cup_j e_j$, i.e.,
the length (or the area) of the electrodes is not precisely known.
However, since the error term $C_m$ is independent of $\omega$, the difference
imaging in  Section~\ref{sub:case(b)} may be applied, provided that $0\notin
\mathcal{P}$, i.e., if the frequency profile $s_0(\omega)$ does not vary much with respect to
 $\omega$. Then  $s_0(\omega) C_m$ disappears upon
differentiating \eqref{eqn:lin-inv-cem-deformed-2}, and the inversion step may be performed as in  Section~\ref{sub:case(b)}.

The method of  Section~\ref{sub:case(a)}
may  also be applied, since the error term $s_0
(\omega) C_m$ depends only on $s_0(\omega)$.
Namely, its influence on the inversion step is lumped into $\delta
\sigma_0$, like the conductivity perturbation $\epsilon\Psi$. Thus, all
the inclusions $\{\delta\sigma_k\}_{k=1}^K$ may be recovered. Alternatively, one may see
this as follows. When multiplying
the system of equations associated to \eqref{eqn:lin-inv-cem-deformed-2}
by $S^{-1}$, the error term $s_0(\omega) C_m$ cancels out in all the systems $MA_k = Y_k$, for
$k=1,\dots,K$:
\begin{equation*}
  \begin{split}
[s_0(\omega_1)C\,,\,\dots\,,\,s_0(\omega_Q)C]\,S^{-1} &=C\,[s_0(\omega_1)\,,\,\dots\,,\,s_0(\omega_Q)]
\begin{bmatrix}
s_0(\omega_1)  &  \cdots  &  s_0(\omega_Q) \\
\vdots  &  \vdots  & \vdots \\
s_K(\omega_1)  &  \cdots  &  s_K(\omega_Q)
\end{bmatrix}^{-1}\\
&= [C\,,\,0\,,\,\dots\,,\,0],
\end{split}
\end{equation*}
where $C$ denotes the column vector corresponding to  $C_m$.

\section{Group Sparse Reconstruction Algorithm}\label{sect:sparsity}

For all the scenarios discussed in the previous sections, one arrives at a number of (decoupled) linear systems
\begin{equation}\label{eqn:lin}
  MA_k = Y_k,\quad k=0,\ldots, K,
\end{equation}
where $M\in\mathbb{R}^{J\times L}$, $A_k\in\mathbb{R}^{L}$,
and $Y_k\in\mathbb{R}^J$. The linear systems are often under-determined, and severely
ill-conditioned, due to the ill-posed nature of the EIT inverse problem. Below we describe a heuristic
and yet very effective strategy
for the stable and accurate solution of \eqref{eqn:lin}; we refer to \cite{ScherzerGrasmair:2009,
SchusterKaltenbacher:2012,ItoJin:2014} for general discussions on regularization methods.

There are several natural aspects for the regularization term, especially sparsity, grouping,
 disjoint sparsity and bound constraints.
\begin{enumerate}
  \item[(1)] For every $k$, the abundance $A_k=(\delta\sigma_k)_l\in\mathbb{R}^L$ is sparse
  with respect to the pixel basis.  This suggests minimizing
  \begin{equation*}
    \min_{A_k\in{\Lambda}} \|A_k\|_1 \quad \mbox{subject to } \|MA_k-Y_k\|\leq \epsilon_k
  \end{equation*}
  for each $k=0,\ldots,K$. Here $\|\cdot\|_1$ denotes the $\ell^1$ norm of a vector.
  The set ${\Lambda}$ represents a box constraint on $A_k$, since
  $\sigma$ is bounded from below and above by positive constants, due to
   physical constraint, and $\epsilon_k>0$ is the estimated  noise level of $Y_k$.

  \item[(2)] In EIT applications, each $A_k$ is often clustered, and this
  refers to the concept of group sparsity. The grouping can remove undesirable spikes often observed
  in the recoveries using the $\ell^1$ penalty alone. This can be achieved by e.g., elastic net \cite{JinLorenzSchiffler:2009}. In this work, we shall exploit the dynamic group sparsity (DGS) \cite{HuangHuang:2009},
  which dynamically realizes group sparsity without knowing the supports of the $A_k$s.

  \item[(3)] The $\mathrm{supp}(A_k)$s are disjoint from each other. The disjoint
  supports of $A_k$s can be promoted, e.g., by penalizing the
  scalar product of the absolute values of the $A_k$s \cite{Vervier:2014}.
\end{enumerate}

Next we develop an algorithm, termed as group iterative soft thresholding (GIST), for achieving
the above goals. It combines the strengths of iterative soft thresholding (IST)
\cite{DaubechiesDefrise:2004} and DGS \cite{HuangHuang:2009}:
IST is easy to implement and has a built-in regularizing effect, whereas  DGS encourages the
group sparsity pattern. It is a simple modification of the IST (by omitting the subscript $k$): given an initial guess $A^0$,
construct an approximation iteratively by
\begin{equation*}
  A^{j+1} = S_{s^j\alpha}(g^j),
\end{equation*}
where the proxy $g^j$ is defined by
\begin{equation}\label{eqn:proxy}
  g^j=A^j-s^jM^t(MA^j-Y).
\end{equation}
Note that $M^t(MA^j-Y)$ is the gradient of $\frac12\|MA-Y\|^2$ at $A^j$,
The scalar $\alpha>0$ is a regularization parameter and $s^j>0$ is the step length. One
simple choice of $s^j$ is the constant one $s^j=1/\|M\|^2$, which ensures the convergence
of IST \cite{DaubechiesDefrise:2004}. The operator $S_\lambda$ for $\lambda>0$ is defined by
$S_\lambda(t) = \max(|t|-\lambda,0)\,\mathrm{sign}(t),$
and applied componentwise for a vectorial argument.

In GIST, instead of performing the thresholding on
$g^j$ directly, we take into account the neighboring influence.
This can be achieved by computing a generalized proxy $d^j_l$ by \cite{HuangHuang:2009}
\begin{equation}\label{eqn:gproxy}
   d^j_l = |g^j_l|^2 + \sum_{k\in \mathcal{N}_l}w_{lk}|g^j_k|^2, \qquad l=1,\ldots,L,
\end{equation}
where $w_{lk}\geq0$ are weights, and $\mathcal{N}_l$ denotes the neighborhood of the $l$th element.
The weights $w_{lk}$ determine the correlation strength: the smaller
$w_{lk}$ is, the weaker the correlation between the $l$th and $k$th components is, and
if $w_{lk}=0$ for all $k\in\mathcal{N}_l$, it does not promote grouping at all. In our
implementation, we take $w_{lk}=\beta$, for some $\beta>0$, for all $k\in\mathcal{N}_l$,
with a default value $\beta=0.5$, and $\mathcal{N}_l$ consists of all elements in the
triangulation that share one edge with the $l$th element. Then $d^j$ is used to reweigh the
thresholding step by
\begin{equation}\label{eqn:d-bar}
  \bar d^j= \max(d^j)^{-1}d^j.
\end{equation}
It indicates a normalized grouping effect: the larger  $\bar d^j_l$ is, the more likely
the $l$th element belongs to the group, and thus less thresholding should be applied. This
can be achieved by rescaling $\alpha$ to be proportional to
$(\bar d^j_l)^{-1}$, with a spatially variable regularization parameter
\begin{equation}\label{eqn:adapt-alpha}
  \bar{\alpha}_l^j=\alpha/\bar d^j_l,\quad l=1,\ldots,L,
\end{equation}
and last perform the projected thresholding with $\bar\alpha^j$ 
\begin{equation}\label{eqn:gist}
  A^{j+1} = P_{\Lambda}(S_{s^j\bar\alpha^j} (g^j)),
\end{equation}
where $P_\Lambda$ denotes the pointwise projection onto the set $\Lambda$.
The complete procedure is listed in Algorithm~\ref{alg:gist}. Here $N\in\mathbb{N}$ is the maximum
number of iterations, and the initial guess $A^0$ is the zero vector. The parameter $\alpha$ plays a crucial role in the performance of
the algorithm: the larger $\alpha$ is, the sparser the recovered $A$ is.
There are several strategies available for its choice, e.g., discrepancy
principle and balancing principle \cite{ItoJin:2014}.
One can terminate the algorithm by monitoring the relative change of the iterates.

\begin{algorithm}[hbt!]
  \caption{Group iterative soft thresholding.}
  \label{alg:gist}
  \begin{algorithmic}[1]
    \STATE Input $M$, $Y$, $W$, $\mathcal{N}$, $\alpha$, $N$ and $A^0$.
    \FOR {$j=1,\ldots,N$}
     \STATE Compute the proxy $g^j$ by \eqref{eqn:proxy}.
     \STATE Compute the generalized proxy $d^j$ by \eqref{eqn:gproxy}.
     \STATE Compute the normalized proxy $\bar d^j$ by \eqref{eqn:d-bar}.
     \STATE Adapt the regularization parameter $\bar\alpha^j$ by \eqref{eqn:adapt-alpha}.
     \STATE Update the abundance $A^{j+1}$ by the group thresholding \eqref{eqn:gist}.
     \STATE Check the stopping criterion.
    \ENDFOR
  \end{algorithmic}
\end{algorithm}

Last, disjoint sparsity can also be enforced in Algorithm~\ref{alg:gist}. Specifically,
we first compute $\bar d^{k,j}$ for $A_k$ separately according to
\eqref{eqn:d-bar}, and then at each $l=1,\ldots,L$, update them by
\begin{equation*}
  \bar d_l^{k,j} =
  \begin{cases}
  \bar d_l^{k,j} & \mbox{ if } k = k^*_l, \\
  \varepsilon &  \mbox{ otherwise},
  \end{cases}
  \qquad\quad
   k^*_l = \argmax_{k=0,\ldots,K}\,\bar d_l^{k,j},
\end{equation*}
where $\varepsilon>0$ is a small number to avoid numerical overflow. It only retains the most
likely abundance (with the likelihood for $A_k$ given by $\bar d^{k,j}$), and hence
enforces the disjoint sparsity.

\begin{remark}
The theoretical analysis of the dynamic group sparse recovery is still unavailable,
except for compressed sensing problems \cite{HuangHuang:2009}. However, it does not cover
the EIT inverse problem, due to a lack of the restricted isometry property.
\end{remark}

\section{Numerical Experiments and Discussions}\label{sec:numer}
Now we present  numerical results to illustrate the analytic study. We consider
only the CEM \eqref{eqn:cem}, since the results for \eqref{eqn:eit}
are similar. The experimental setup is as follows. The
computational domain $\Omega$ is taken to be the unit circle $\Omega=\{(x_1,x_2):
x_1^2+x_2^2<1\}$. There are sixteen electrodes $\{e_j\}_{j=1}^E$ (with $E=16$) evenly placed along
the boundary $\partial\Omega$, each of length $\pi/16$, thus occupying one half of
$\partial\Omega$, cf. Fig.~\ref{fig:electrodea}.
Unless otherwise specified, the contact impedances $\{z_j\}_{j=1}^E$ on the electrodes $\{e_j\}_{j=1}^E$ are all set
to unit, and $\sigma_0\equiv1$.
Further, we assume that $s_0(\omega)$ for the
background is $s_0(\omega)\equiv1$. This is not a restriction, since $s_0(\omega)$
is known, and one can rescale $s_k(\omega)$s so that $s_0\equiv1$.
We measure $U$ for all 15 sinusoidal input currents. The
model \eqref{eqn:cem} is discretized using a piecewise
linear FEM on a shape regular quasi-uniform triangulation of $\Omega$ \cite{GehreJinLu:2014}.
For the inversion, the conductivity is represented on a coarser mesh using a piecewise constant basis.
Then the noisy data
$U^\delta$ is generated by adding Gaussian noise to the exact data $U^\dag:= U(\sigma^\dag)$
corresponding to the true conductivity $\sigma^\dag(x,\omega)$ as follows
\begin{equation*}
  U_j^\delta = U_j^\dag + \epsilon\max_{l} |U_l^\dag-U_l(\sigma_0)| \varepsilon_j,\quad j=1,\ldots,E,
\end{equation*}
where $\epsilon$ is the relative noise level, and $\varepsilon_j$ follows the standard normal distribution.

\begin{remark}
Colton and Kress \cite[pp. 121, 289]{ColtonKress:1992}
coined the term \textit{inverse crime} to denote the act of employing the \textit{same} model to
generate and to invert synthetic data. Inverse crime often leads to excellent reconstructions
without revealing the ill-posed nature of inverse problems, and hence has to be avoided in numerical
experiments. In Section \ref{sec:numer-perf}, we have employed a finer mesh to generate the data than
 for inversion, in order to alleviate the inverse crime; and in
Section \ref{sec:numer-imperf}, the meshes for generating the data and inversion are completely different.
\end{remark}

We shall present numerical results for the cases of a perfectly known and of an imperfectly known boundary
separately, and discuss only cases a) and b) with spectral profiles that are either fully known or have
substantially different frequency dependence. Case c) will not be discussed since the inversion is
analogous to case a). To solve \eqref{eqn:lin},
we use Algorithm~\ref{alg:gist} with a constant step size. The scalar $\alpha$ was determined in
a trial-and-error manner, and set to $10^{-2}$ for all examples below, unless otherwise specified. We did not implement disjoint
sparsity, since in all examples below the recoveries are already very satisfactory. The algorithm is
always initialized with a zero vector. Numerically, it converges steadily and fast, and for
the examples presented below, it takes about 8 seconds per recovery. All the
computations were performed using MATLAB 2013a on a 2.5G Hz and 6G RAM personal laptop.

\begin{figure}[hbt!]
  \centering
  \subfloat[computational domain $\Omega$]{\includegraphics[width=.4\textwidth]{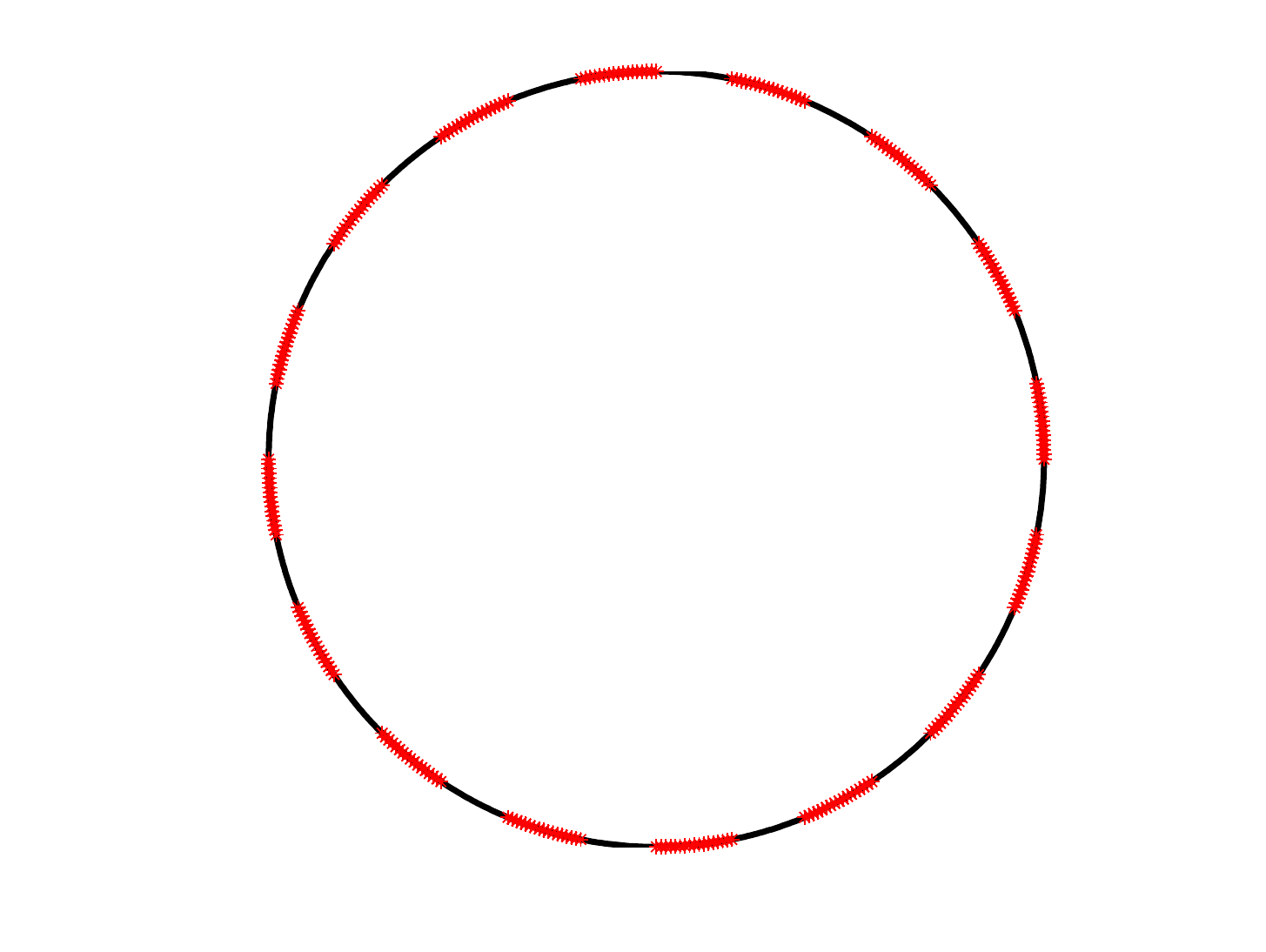}\label{fig:electrodea}}
  \subfloat[imperfectly known electrode positions]{\includegraphics[width=.4\textwidth]{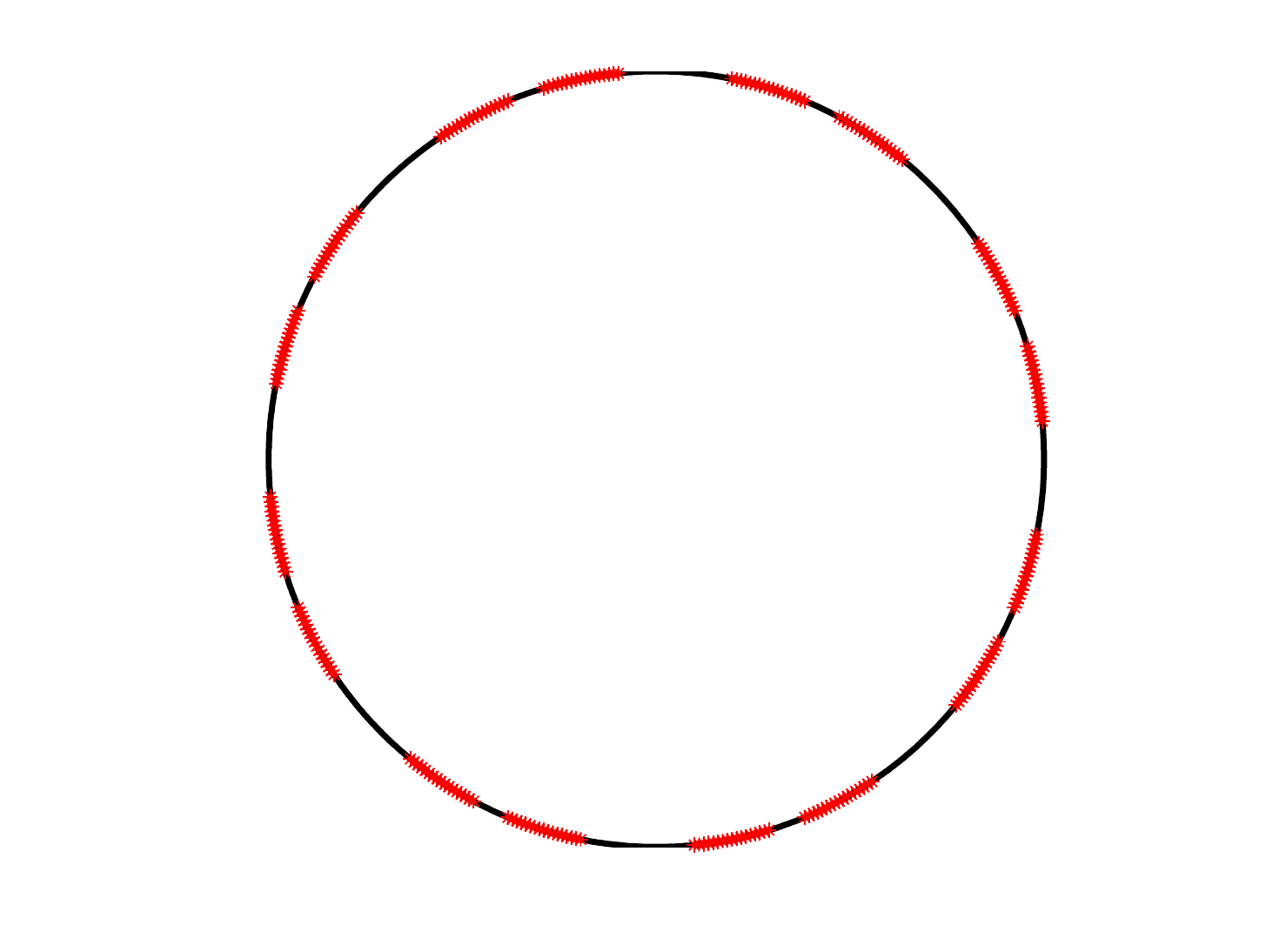}\label{fig:electrodeb}}
  \caption{Electrode arrangement for the computational domain $\Omega$ and for imperfectly known electrode positions
  (used in Example~\ref{exam4}). The curved
  segments in red denote electrodes.}\label{fig:electrode}
\end{figure}

\subsection{Perfectly Known Boundary}\label{sec:numer-perf}

First, we consider the case of a known boundary.
\begin{exam}\label{exam1}
Consider three square inclusions: the two inclusions on the top share the same spectral profile $s_1(\omega)$,
and the one on the bottom has a second spectral profile $s_2(\omega)$; cf.\ Fig.~\ref{fig:exam1ia} for an
illustration. In the experiments, we consider the following two cases:
\begin{itemize}
    \item[(i)] The spectral profiles are $s_1(\omega)=0.1\omega+0.1$
       and $s_2(\omega)=0.2\omega$;
    \item[(ii)] The spectral profiles  are $s_1(\omega)=0.1\omega+0.1$ and $s_2(\omega)=0.02\omega$.
\end{itemize}
In either case, we   take  $Q=3$ frequencies, $\omega_1=0$, $\omega_2=0.5$ and $\omega_3=1$.
\end{exam}

\begin{figure}[htp!]
\subfloat[true $\delta\sigma_k$s]{\includegraphics[width=0.3\columnwidth]{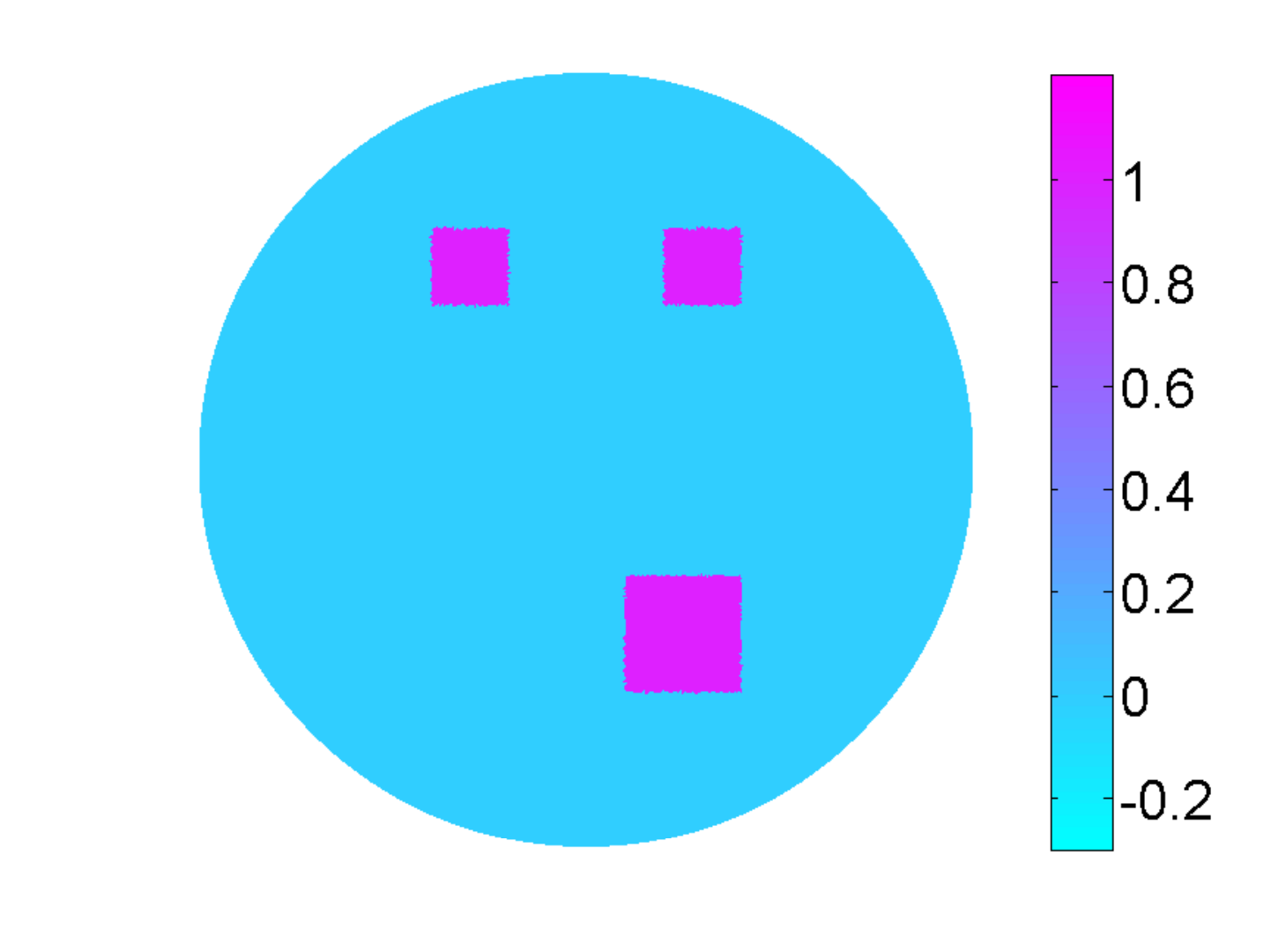}\label{fig:exam1ia}}\hfill
\subfloat[recovered $\delta\sigma_1$]{\includegraphics[clip,width=0.3\columnwidth]{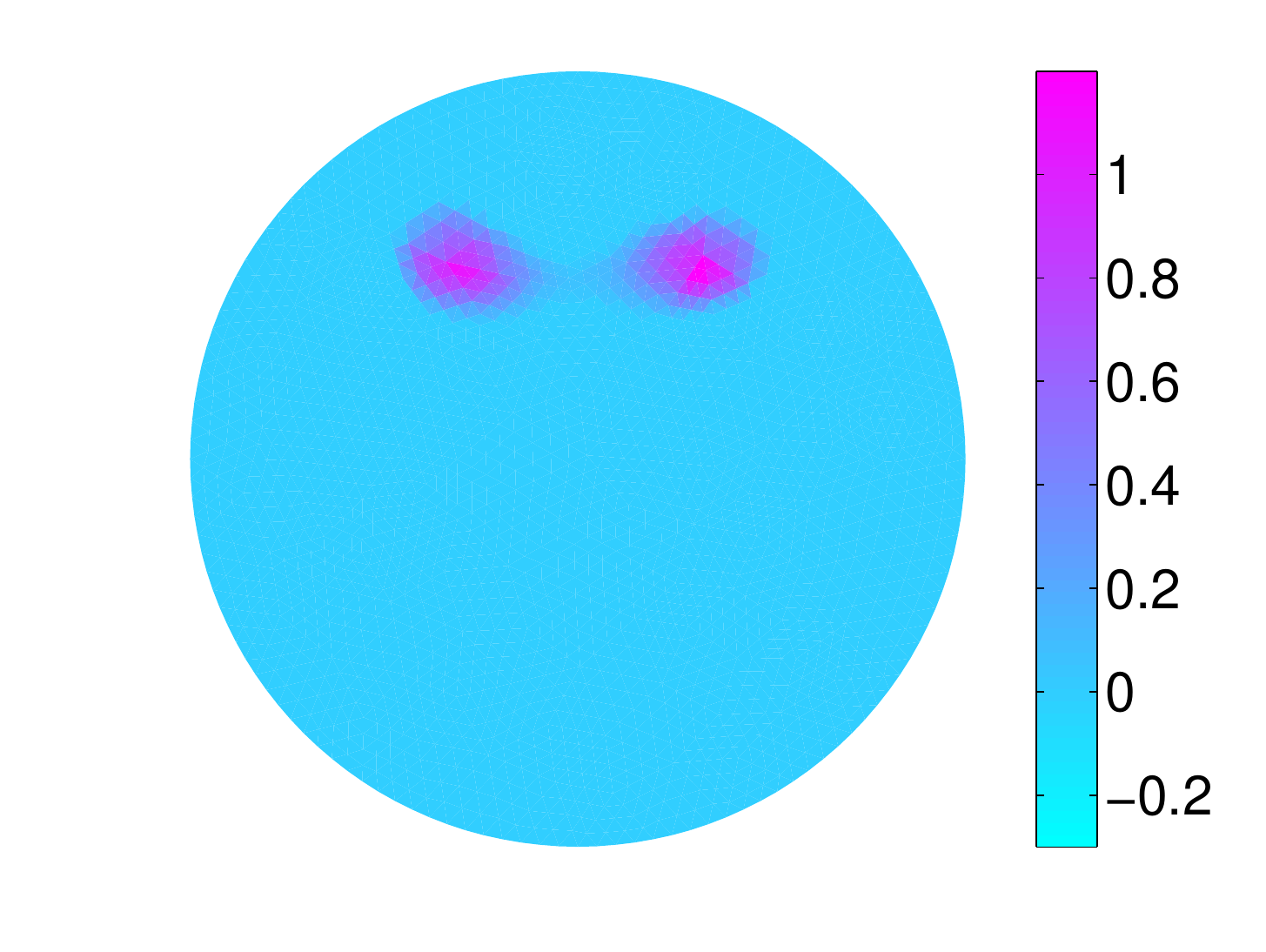}\label{fig:exam1ib}}\hfill
\subfloat[recovered $\delta\sigma_2$]{\includegraphics[width=0.3\columnwidth]{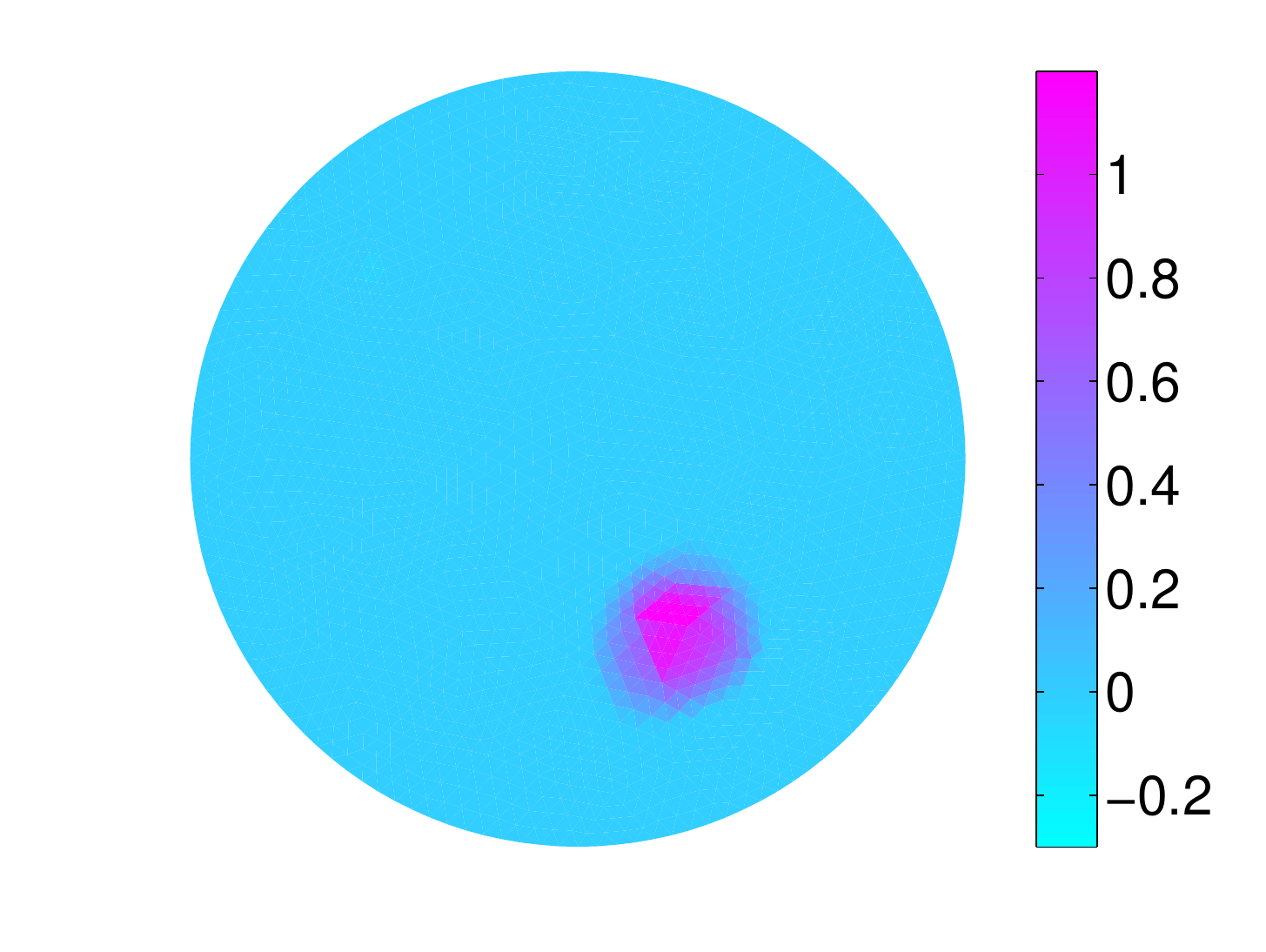}\label{fig:exam1ic}}
\caption{Numerical results for Example~\ref{exam1}(i) with $1\%$ data noise, and fully known $s_k(\omega)$s. The recoveries
are obtained using the direct approach.}
\label{fig:exam1i}
\end{figure}

The results for Example~\ref{exam1} with $\epsilon=1\%$ data noise are shown in Figs.~\ref{fig:exam1i}
and \ref{fig:exam1ii} for cases (i) and (ii), respectively. In case (i), the two frequencies have about the same
magnitude, and the matrix $S$ is nonsingular. The direct approach in Section~\ref{sub:case(a)} separates  the two
sets of inclusions  well thanks to the spectral incoherence. The recovery is very localized within a clean
background, the supports match closely the true ones (and are clearly disjoint from each other) and their magnitudes
 are well retrieved. The latter observation is a distinct feature of the proposed GIST
in  Section~\ref{sect:sparsity}. Hence, for known incoherent profiles, the
inclusions can be fairly recovered. It is noteworthy that our approach is insensitive to model parameters:
see Fig.~\ref{fig:exam1i-z} for the recoveries  with different contact impedance constants. Case (ii) is
similar, except that the variation of
$s_2(\omega)$ is much smaller. The preceding observations remain largely valid, except that the
inclusion $\delta\sigma_2$ has minor spurious oscillations. This is attributed to
the presence of data noise: the noise is comparable with inclusion contributions. Hence, for
the accurate recovery, the data should be reasonably
accurate.

\begin{figure}[htp!]
\subfloat[recovered $\delta\sigma_1$]{\includegraphics[trim = 1cm 0.5cm 0.5cm .5cm, clip=true,width = .23\textwidth]{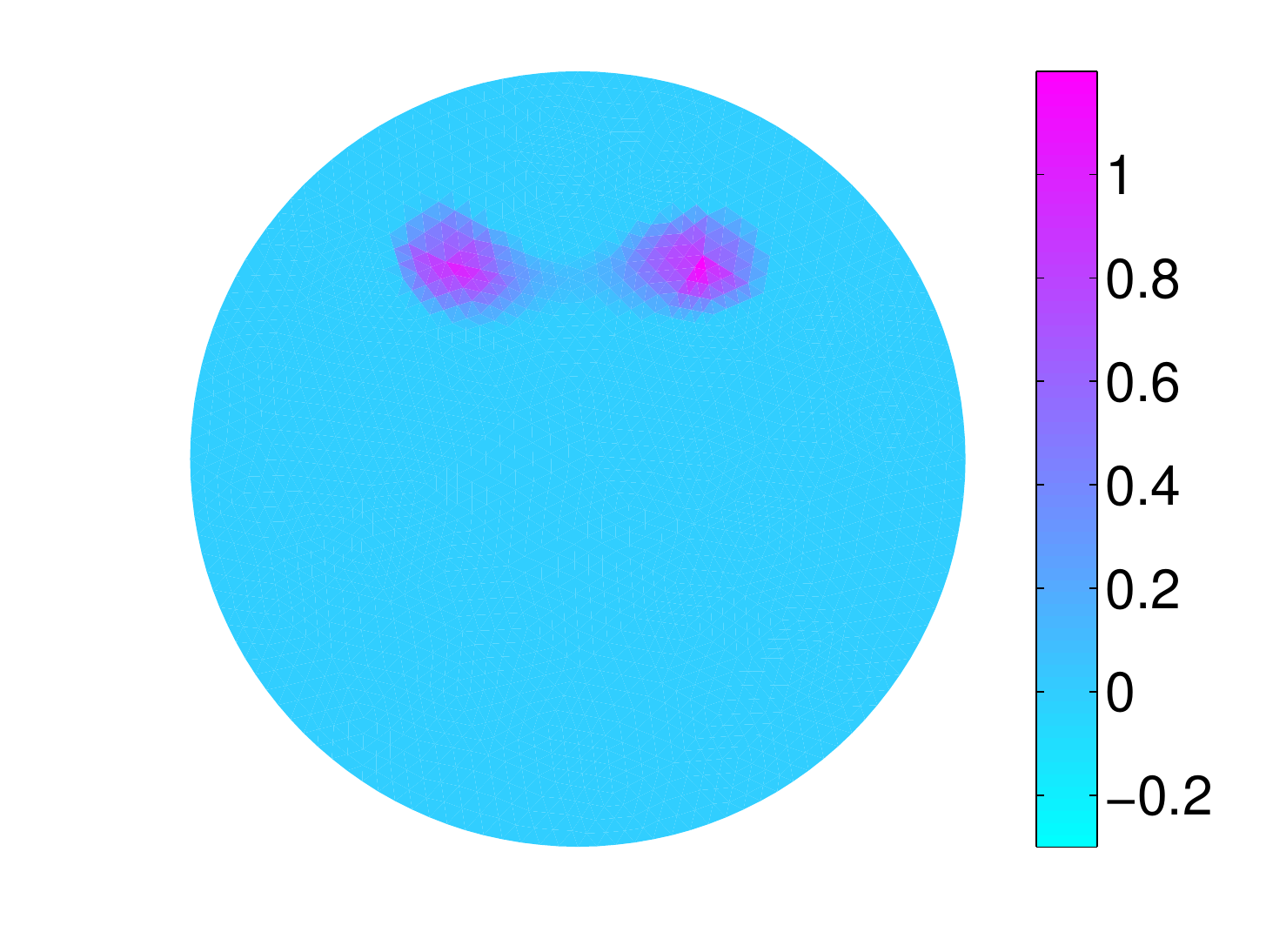}}\hfill
\subfloat[recovered $\delta\sigma_2$]{\includegraphics[trim = 1cm 0.5cm 0.5cm .5cm, clip=true,width = .23\textwidth]{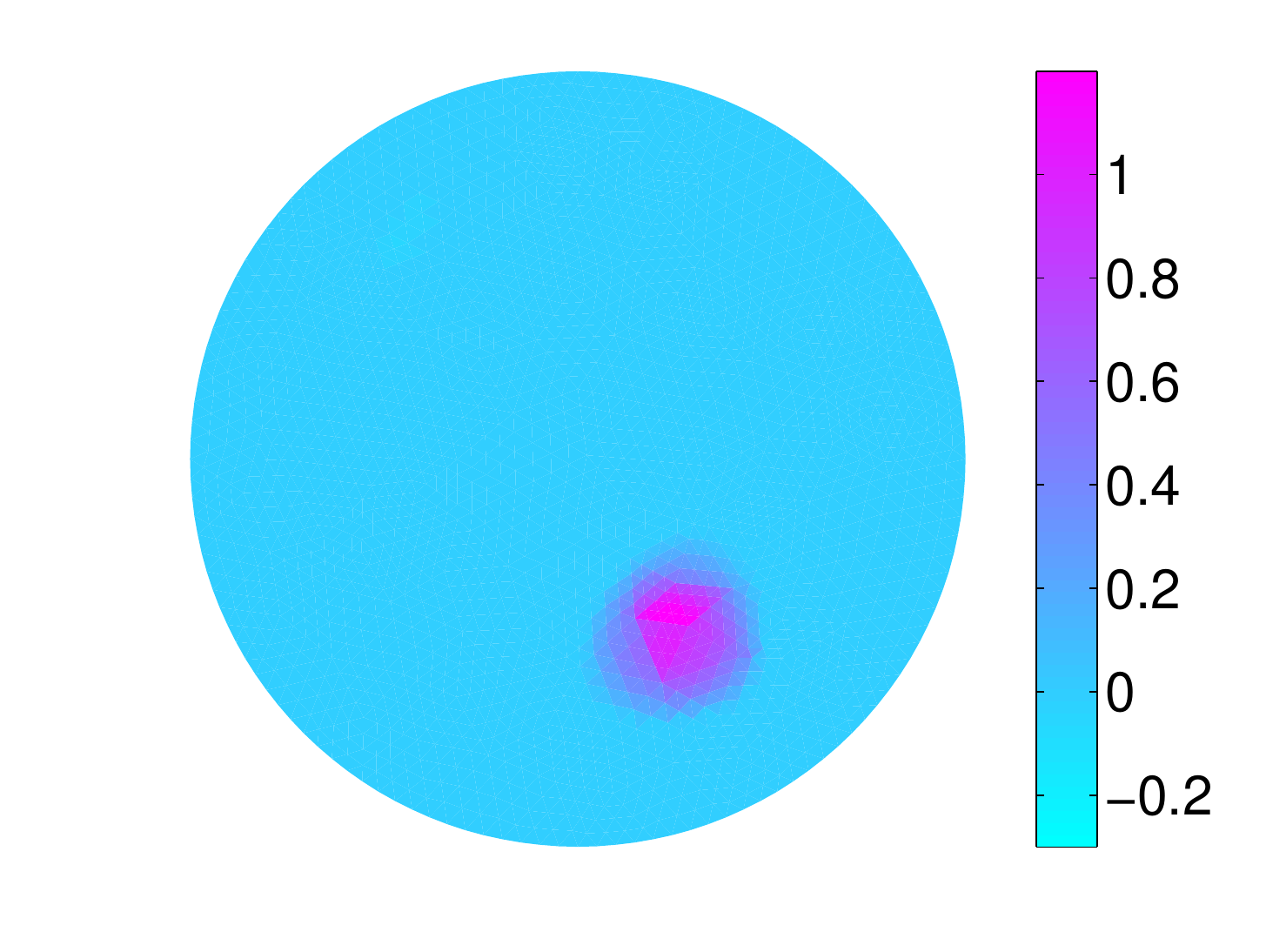}}\hfill
\subfloat[recovered $\delta\sigma_1$]{\includegraphics[trim = 1cm 0.5cm 0.5cm .5cm, clip=true,width = .23\textwidth]{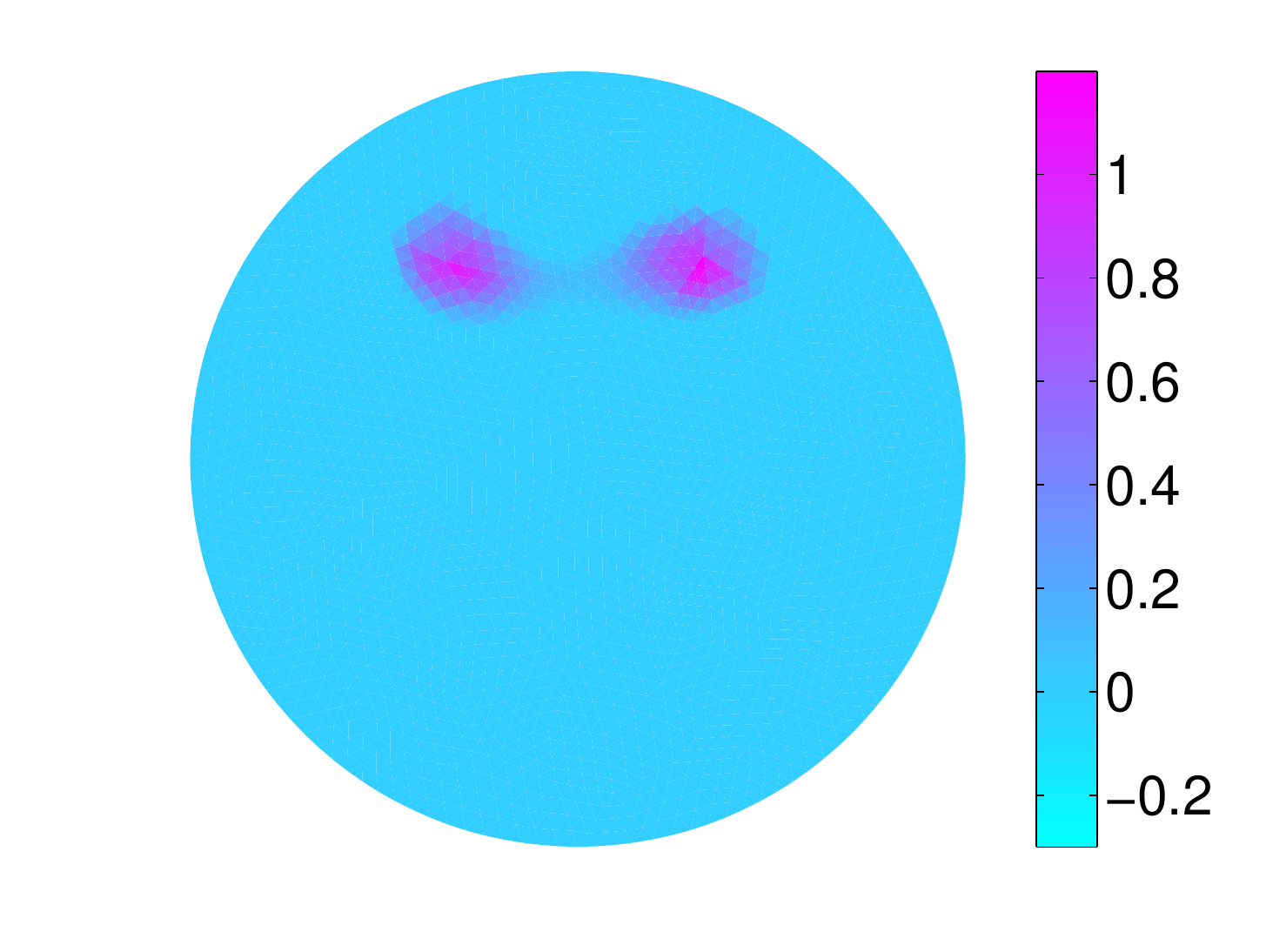}}\hfill
\subfloat[recovered $\delta\sigma_2$]{\includegraphics[trim = 1cm 0.5cm 0.5cm .5cm, clip=true,width = .23\textwidth]{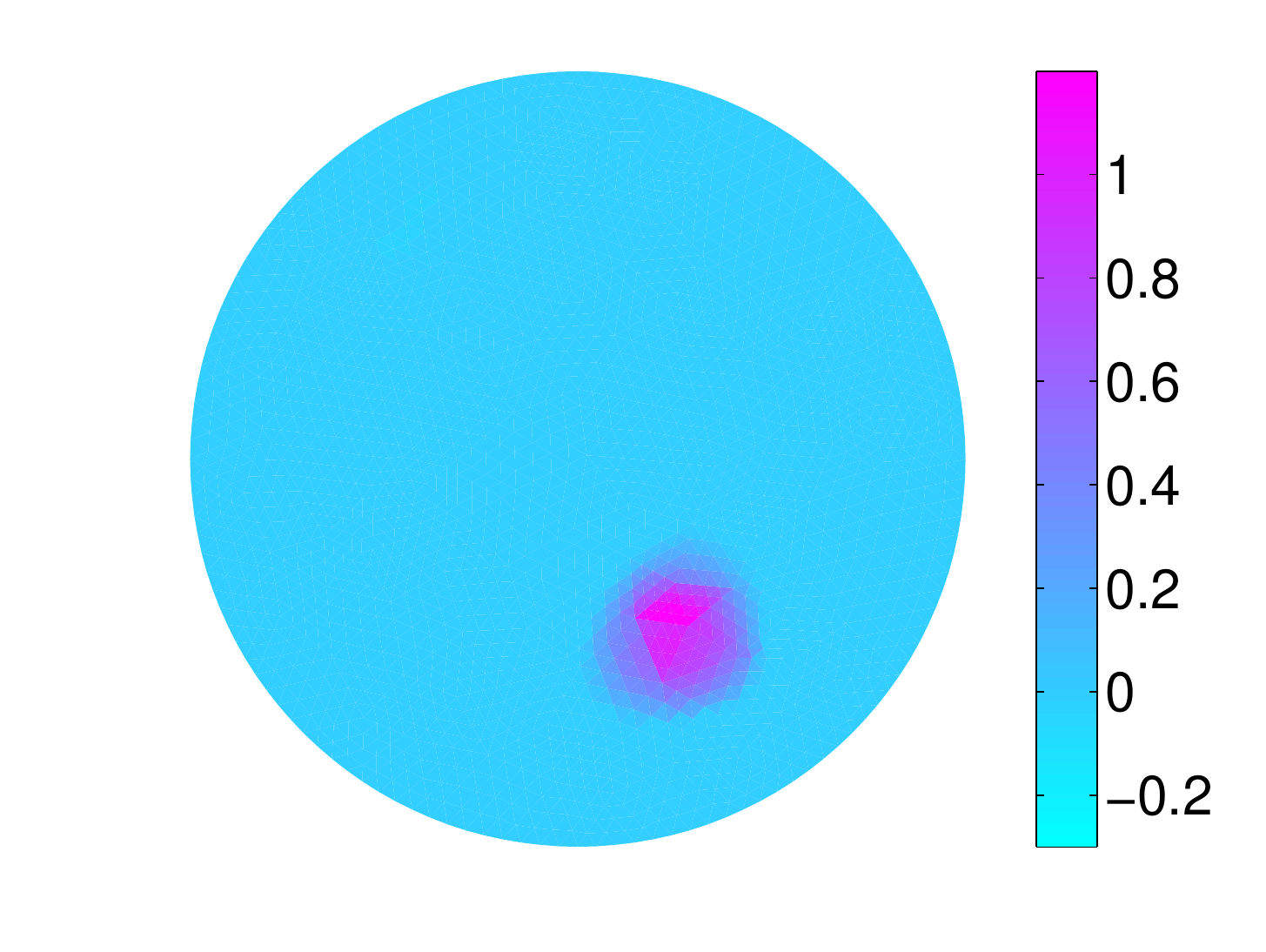}}
\caption{Numerical results for Example~\ref{exam1}(i) with different contact impedance
constants, $1\%$ data noise, and fully known $s_k(\omega)$s. The recoveries in (a) and (b) are obtained with $z_j=0.1$,
$j=1,\ldots,E$, and those in (c) and (d) with $z_j=0.01$, $j=1,\ldots,E$, by the direct approach.}
\label{fig:exam1i-z}
\end{figure}

The well-conditioning of $S$ implies the robustness of the direct approach
with respect to spectral profile perturbations, cf. Section~\ref{sub:case(a)}.
We present in Fig.~\ref{fig:exam1i-perturbed} the recoveries using imprecise
spectral profiles for Example~\ref{exam1}(i), where the spectral matrix is perturbed by additive Gaussian noise with a zero mean and
standard deviation proportional to the entry magnitude. Even only with three frequencies, the
recoveries remain stable up to 20\% spectral perturbation, indicating the
robustness of the approach, concurring with the findings in \cite{MaloneSato:2015}.

\begin{figure}[htp!]
\subfloat[recovered $\delta\sigma_1$]{\includegraphics[trim = 1cm 0.5cm 0.5cm .5cm, clip=true,width = .23\textwidth]{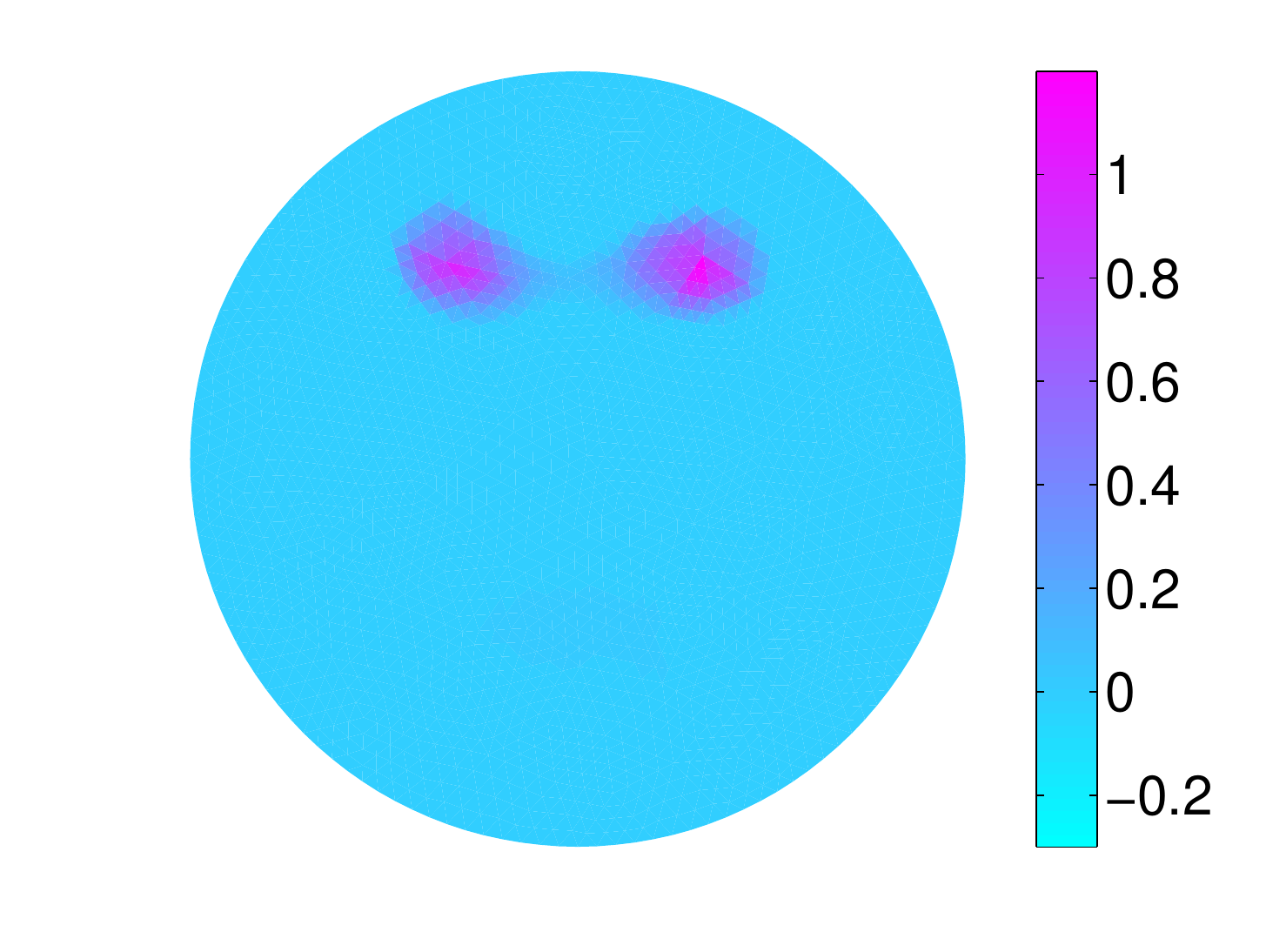}}\hfill
\subfloat[recovered $\delta\sigma_2$]{\includegraphics[trim = 1cm 0.5cm 0.5cm .5cm, clip=true,width = .23\textwidth]{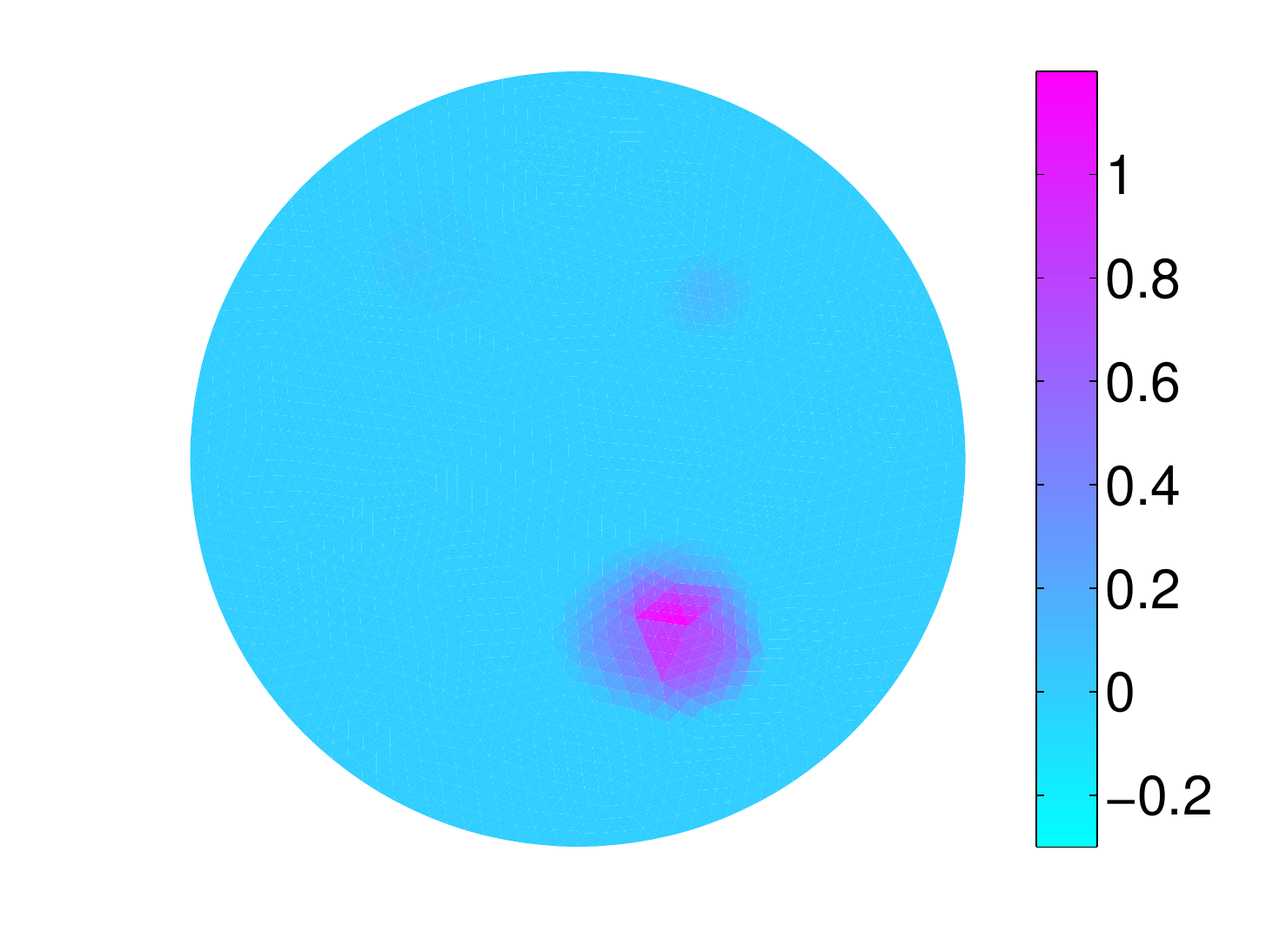}}\hfill
\subfloat[recovered $\delta\sigma_1$]{\includegraphics[trim = 1cm 0.5cm 0.5cm .5cm, clip=true,width = .23\textwidth]{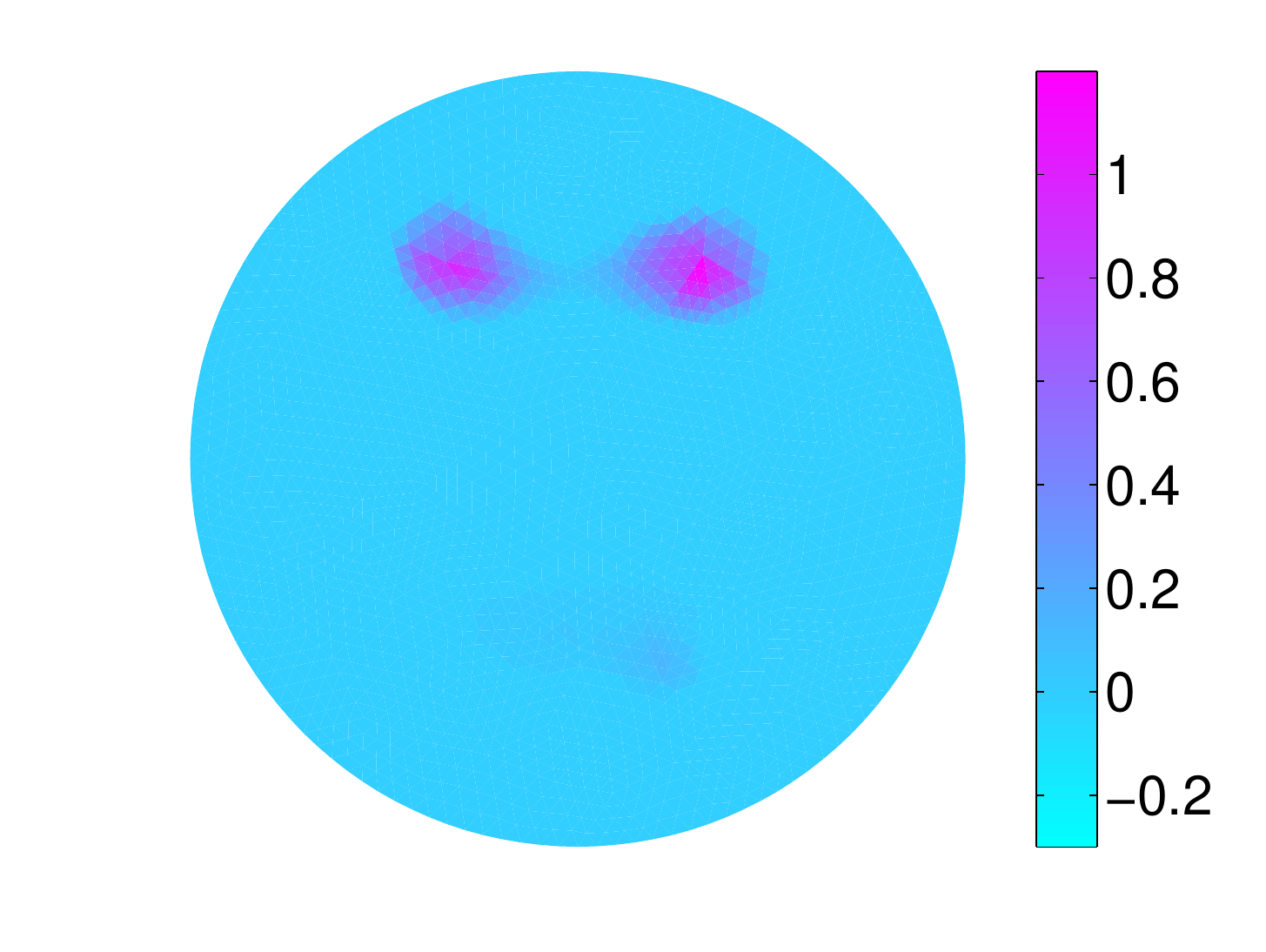}}\hfill
\subfloat[recovered $\delta\sigma_2$]{\includegraphics[trim = 1cm 0.5cm 0.5cm .5cm, clip=true,width = .23\textwidth]{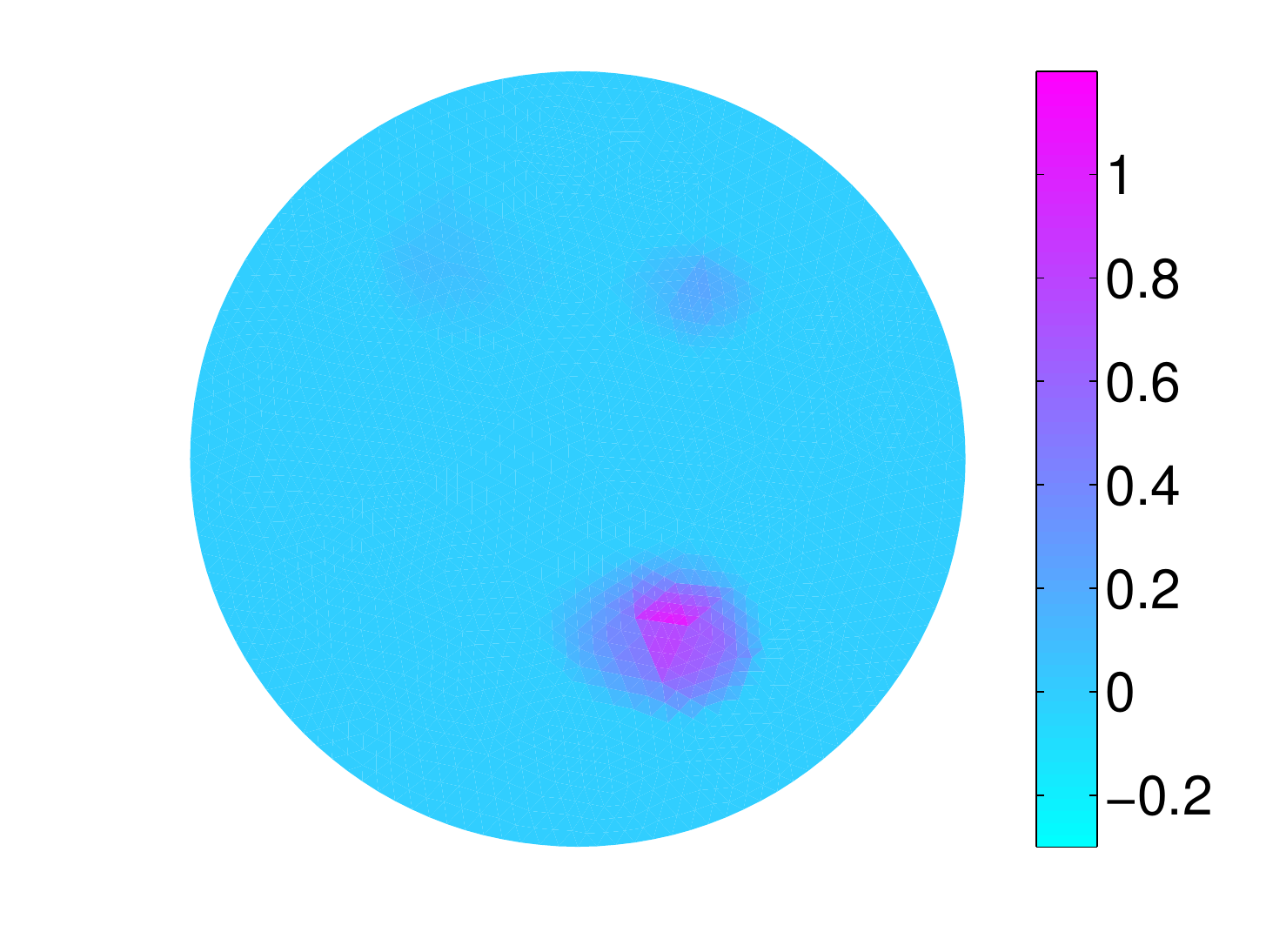}}
\caption{Numerical results for Example~\ref{exam1}(i) with $1\%$ data noise, and imprecisely known
$s_k(\omega)$s. The recoveries in (a) and (b) are obtained with $S$
perturbed by additive Gaussian noise with mean zero and standard deviation $10\%$ of the
entry magnitude, and those in (c) and (d) with $20\%$ noise, both by the direct approach.}
\label{fig:exam1i-perturbed}
\end{figure}

\begin{figure}[htp!]
\subfloat[true $\delta\sigma_k$s]{\includegraphics[trim = 1cm 0.5cm 0.5cm .5cm, clip=true,width = .23\textwidth]{exam1_inc_true.eps}\label{fig:exam1iia}}\hfill
\subfloat[recovered $\delta\sigma_1$]{\includegraphics[trim = 1cm 0.5cm 0.5cm .5cm, clip=true,width = .23\textwidth]{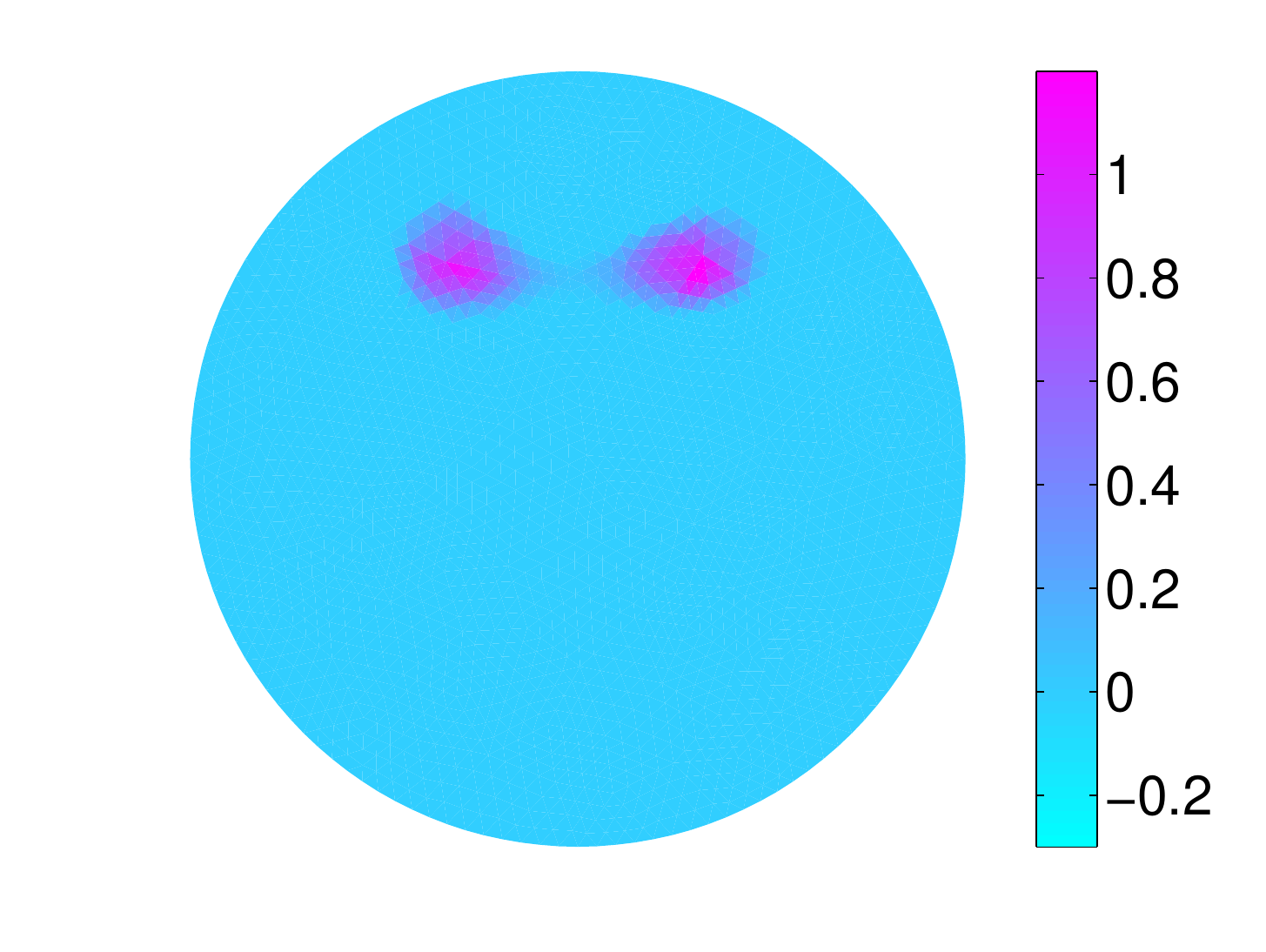}\label{fig:exam1iib}}\hfill
\subfloat[recovered $\delta\sigma_2$]{\includegraphics[trim = 1cm 0.5cm 0.5cm .5cm, clip=true,width = .23\textwidth]{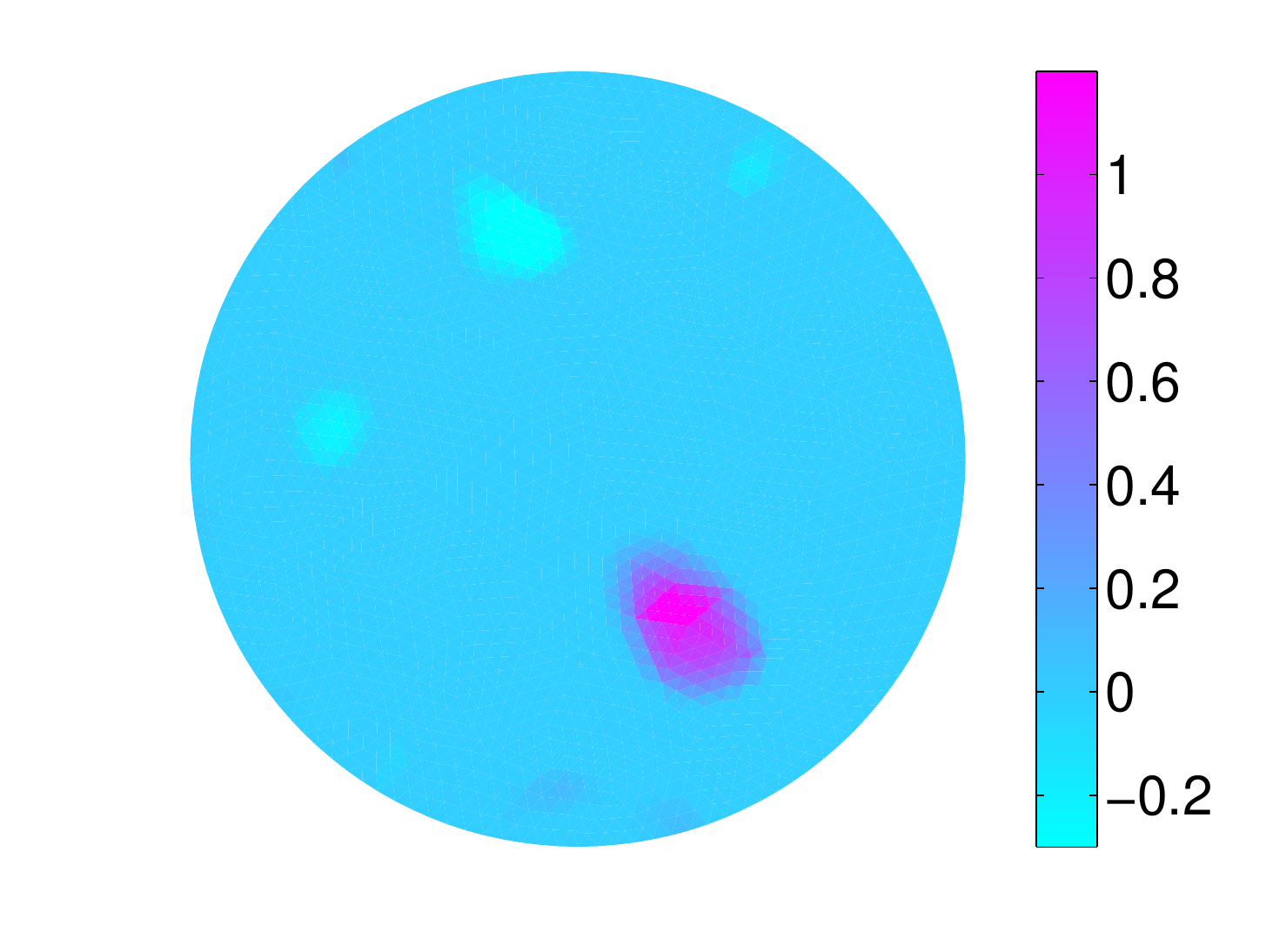}\label{fig:exam1iic}}\hfill
\subfloat[recovered $\delta\sigma_1$]{ \includegraphics[trim = 1cm 0.5cm 0.5cm .5cm, clip=true,width = .23\textwidth]{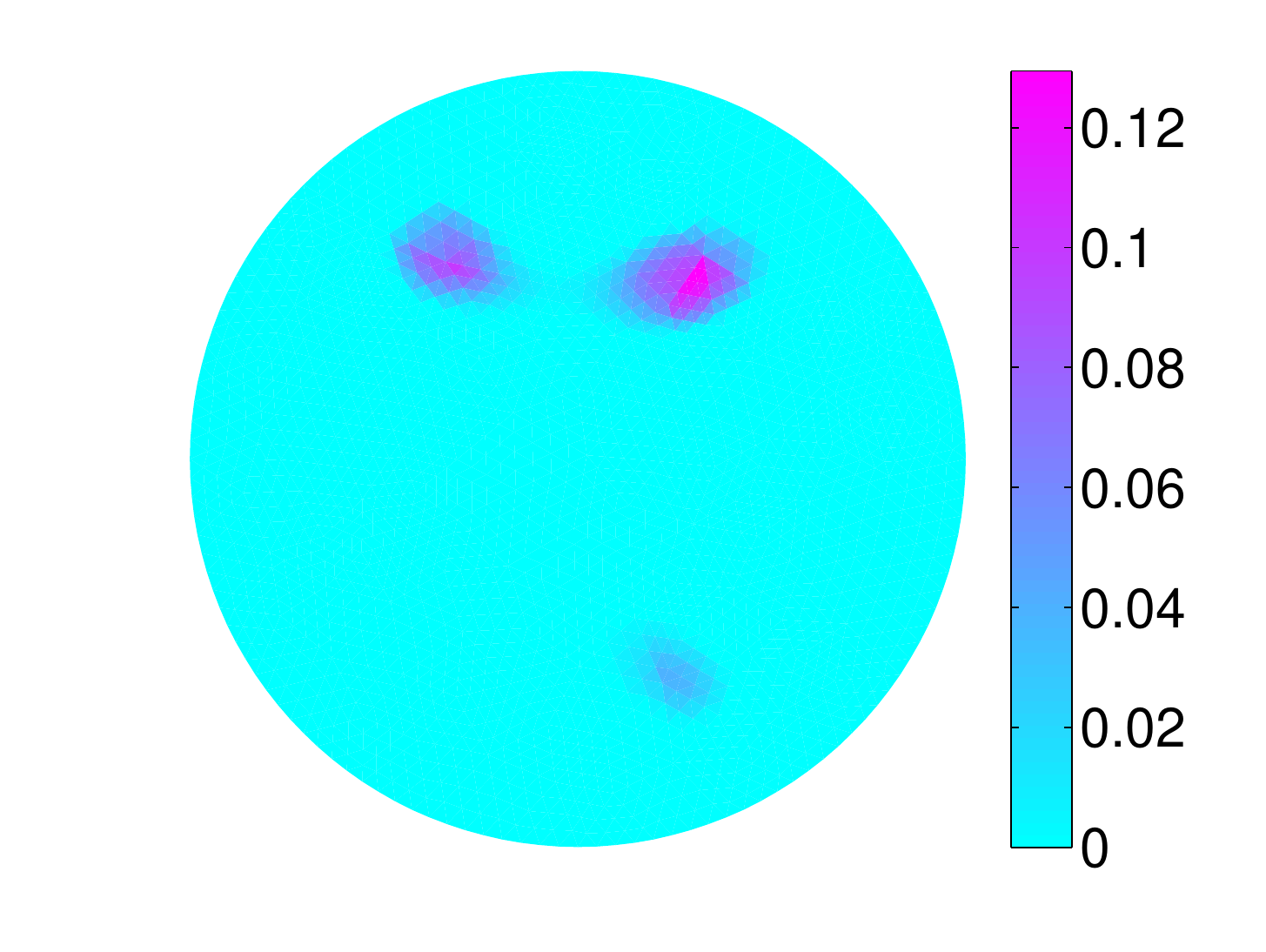}\label{fig:exam1iid}}
\caption{Numerical results for Example~\ref{exam1}(ii) with $1\%$ data noise. The recoveries in
(b) and (c) are obtained with known $s_k(\omega)$s using direct approach, and that in (d)
without knowing $s_k(\omega)$s, using difference imaging.}
\label{fig:exam1ii}
\end{figure}

Throughout, we have assumed a fixed discretization for the linearized model. Due to the ill-posed
nature of the problem, the recovery may vary with the discretization, due to discretization
error, apart from the data noise (and inherent linearization error).
With Example \ref{exam1}(i), we briefly illustrate the dependence of the relative error of the recovery
on the mesh size $h$ used for the inversion (cf.\ \eqref{eqn:cem_lin}), the noise level
$\epsilon$ and the regularization parameter $\alpha$. The results are shown in Table \ref{tab:exam1i}
for various combinations of $h$, $\epsilon$, and $\alpha$.
Just as expected, the relative error increases with $h$ and $\epsilon$, and the convergence is relatively
independent of $\alpha$ within this range. A detailed convergence analysis with respect to the discretization parameter for
EIT imaging with Tikhonov regularization can be found in \cite{GehreJinLu:2014,Rondi:2016}.

\begin{table}[hbt!]
\centering
\caption{The relative errors for Example \ref{exam1}(i) with various mesh size $h$,
noise level $\epsilon$ and regularization parameter $\alpha$, where
$h_1=$1.27e-1, $h_2=$6.36e-2 and $h_3=$3.18e-2. The error is computed with respect to the reference
solution, which is the recovery on the finest mesh.}\label{tab:exam1i}
\begin{tabular}{c|ccc|ccc|ccc}
\hline
  & \multicolumn{3}{|c|}{$\alpha=$5-3} & \multicolumn{3}{c|}{$\alpha=$1e-2} & \multicolumn{3}{c}{$\alpha=$5e-2} \\
    \cline{2-10}
 $\epsilon$ & $h_1$ & $h_2$ & $h_3$ & $h_1$ & $h_2$ & $h_3$ & $h_1$ & $h_2$ & $h_3$\\
   \hline
   1e-3 & 5.57e-2 & 2.72e-2 & 1.22e-2 & 7.80e-2 & 3.85e-2 & 1.74e-2 & 2.42e-1 & 1.26e-1 & 5.87e-2 \\
   3e-3 & 5.61e-2 & 2.75e-2 & 1.23e-2 & 7.83e-2 & 3.87e-2 & 1.75e-2 & 2.42e-1 & 1.26e-1 & 5.88e-2 \\
   1e-2 & 5.77e-2 & 2.82e-2 & 1.27e-2 & 7.95e-2 & 3.94e-2 & 1.78e-2 & 2.42e-1 & 1.26e-1 & 5.89e-2 \\
   \hline
  \end{tabular}
\end{table}

Since $s_2^\prime(\omega)$ is small in Example~\ref{exam1}(ii), we also illustrate difference
imaging in  Section~\ref{subsub:b1}. The recovery of the first
set of inclusions, in the absence of the knowledge of $s_k$s, is shown in  Fig.~\ref{fig:exam1iid}.
The recoveries are free from spurious oscillations. This shows the capability of difference
imaging for spectral profiles with substantially different dependence on $\omega$.

\begin{exam}\label{exam2}
Consider three rectangular inclusions on the top left, top right and bottom of the disk with spectral
profiles $s_1(\omega)$, $s_2(\omega)$ and $s_3(\omega)$, respectively, cf.\ Figure~\ref{fig:exam2aa} for an illustration.
In the experiments, we consider the following two cases:
\begin{itemize}
    \item[(i)] The spectral profiles  are  $s_1(\omega)=0.2\omega+0.2$, $s_2(\omega)=0.1\omega^2$ and $s_3(\omega)=0.2\omega + 0.1$;
    \item[(ii)] The spectral profiles are $s_1(\omega)=0.02\omega+0.02$, $s_2(\omega)=0.1\omega^2$ and $s_3(\omega)=0.2\omega+0.1$.
\end{itemize}
In either case, we take three frequencies, $\omega_1=0$, $\omega_2=0.5$ and $\omega_3=1$.
\end{exam}

The numerical results for Example~\ref{exam2}(i) and \ref{exam2}(ii) are shown in Figs.~\ref{fig:exam2a} and
\ref{fig:exam2b}, respectively. If all three $s_k(\omega)$s are known, the use of three frequencies
yields almost perfect separation of the inclusions using the direct method:
the recovered inclusions are well clustered with correct
supports and magnitudes. For Example~\ref{exam2}(ii), $s_1(\omega)$
is much smaller, and thus the recovered $\delta\sigma_1$ is more susceptible to noise,
whereas the other two are more stable.

The results in Fig.~\ref{fig:exam2b} indicate that with known $s_2(\omega)$
and $s_3(\omega)$ and unknown $s_1(\omega)$, since  $s_1^\prime(\omega)$ is small, difference imaging
can recover accurately both the magnitude and support of  $\delta\sigma_2$ and
$\delta\sigma_3$. These observations fully confirm the discussions in  Section~\ref{sub:case(b)}.

\begin{figure}[hbt!]
  \centering
    \subfloat[true $\delta\sigma_k$s]{\label{fig:exam2aa}\includegraphics[trim = 1.5cm .5cm .5cm .5cm, clip=true,width = .23\textwidth]{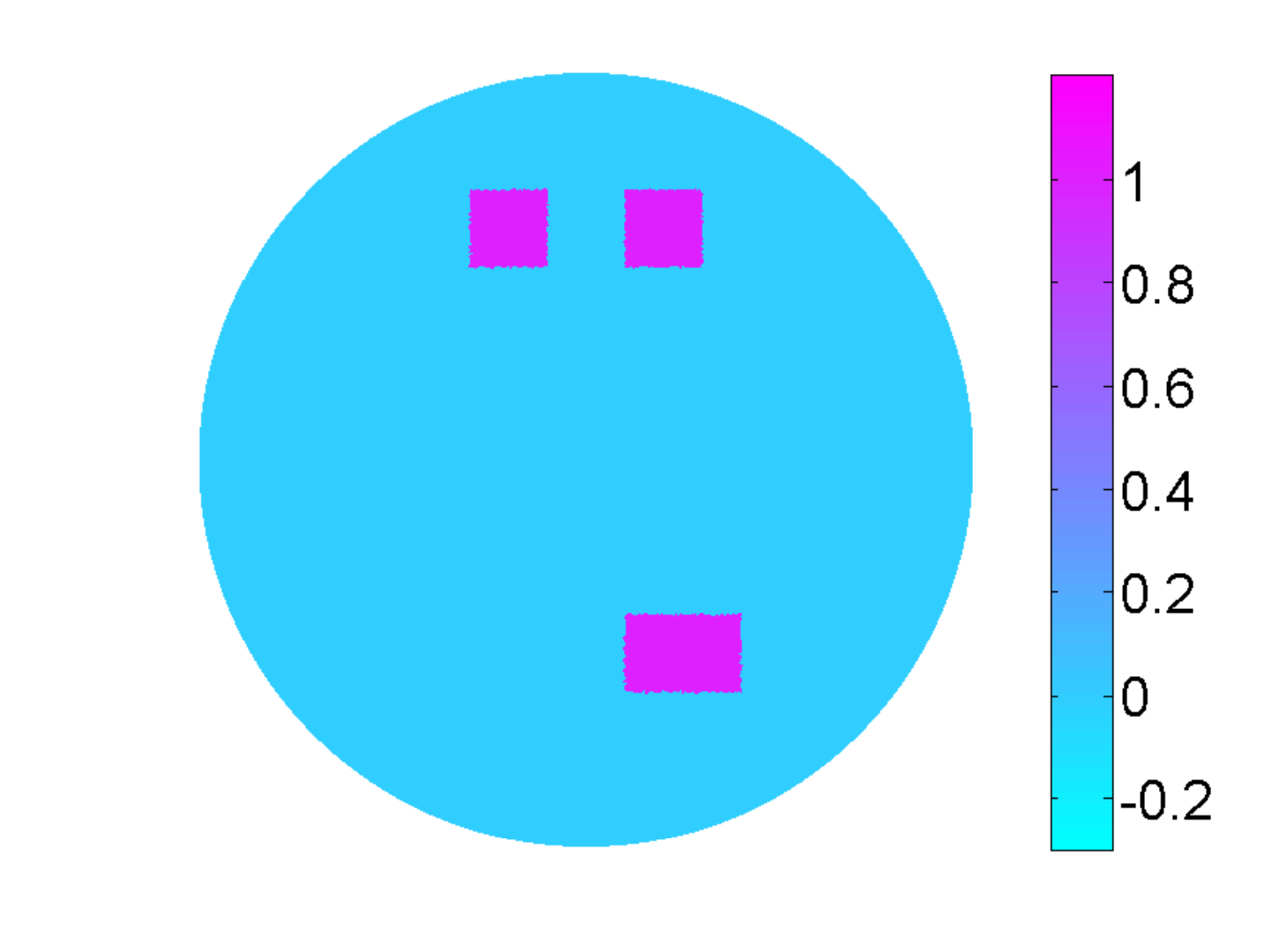}}\hfill
    \subfloat[recovered $\delta\sigma_1$]{\label{fig:exam2ab}\includegraphics[trim = 1.5cm .5cm .5cm .5cm, clip=true,width = .23\textwidth]{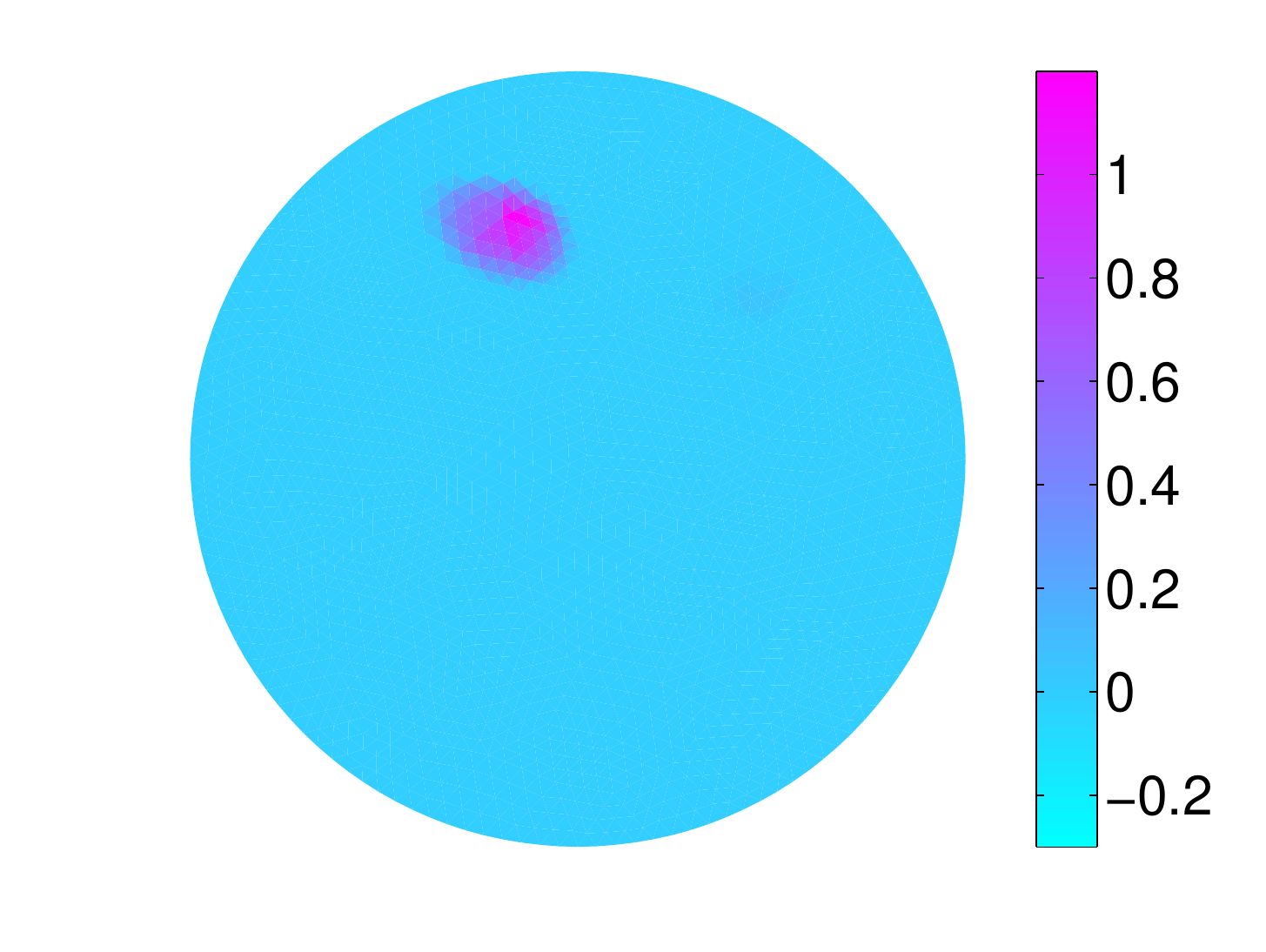}}\hfill
    \subfloat[recovered $\delta\sigma_2$]{\label{fig:exam2ac}\includegraphics[trim = 1.5cm .5cm .5cm .5cm, clip=true,width = .23\textwidth]{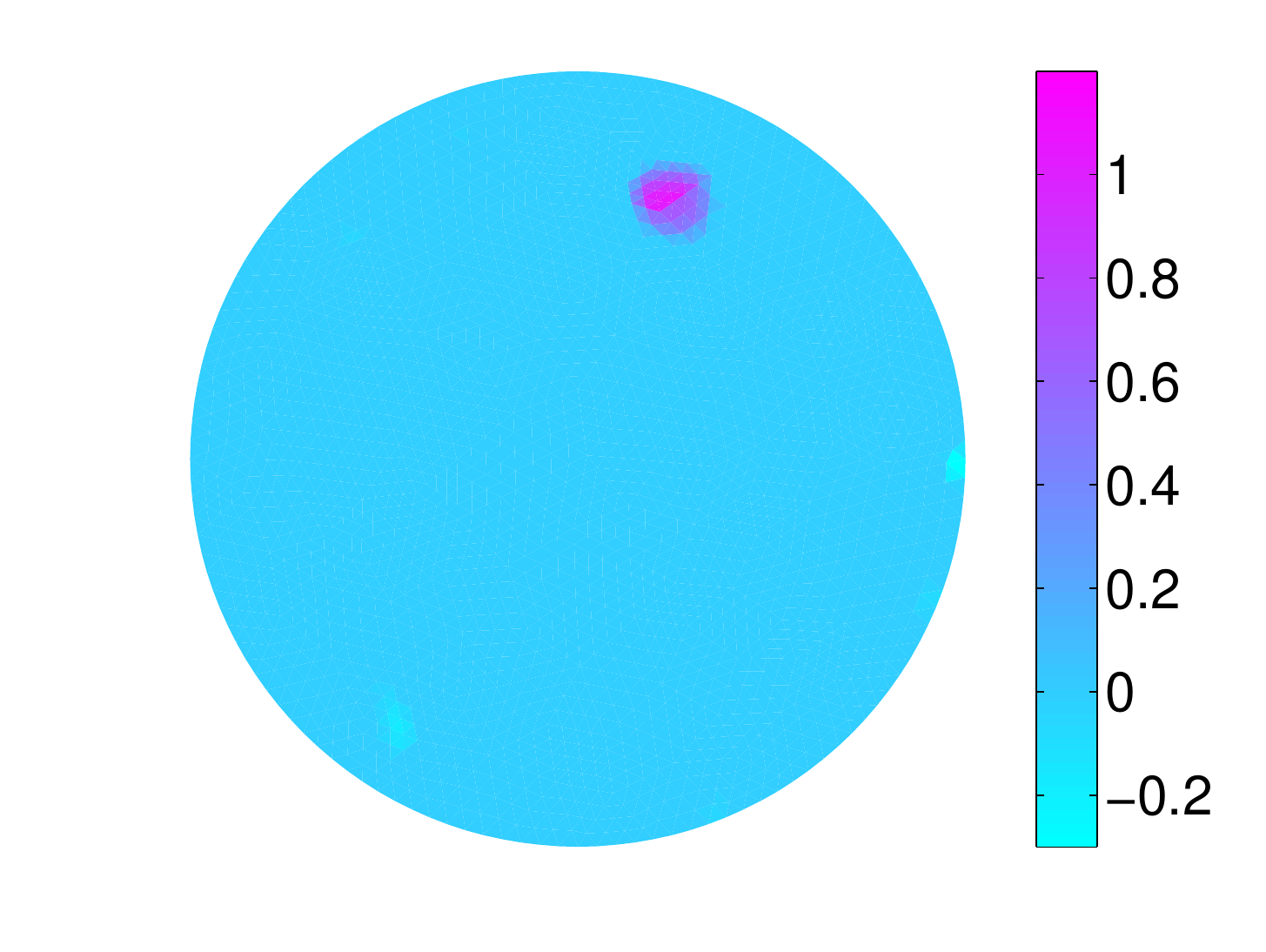}}\hfill
    \subfloat[recovered $\delta\sigma_3$]{\label{fig:exam2ad}\includegraphics[trim = 1.5cm .5cm .5cm .5cm, clip=true,width = .23\textwidth]{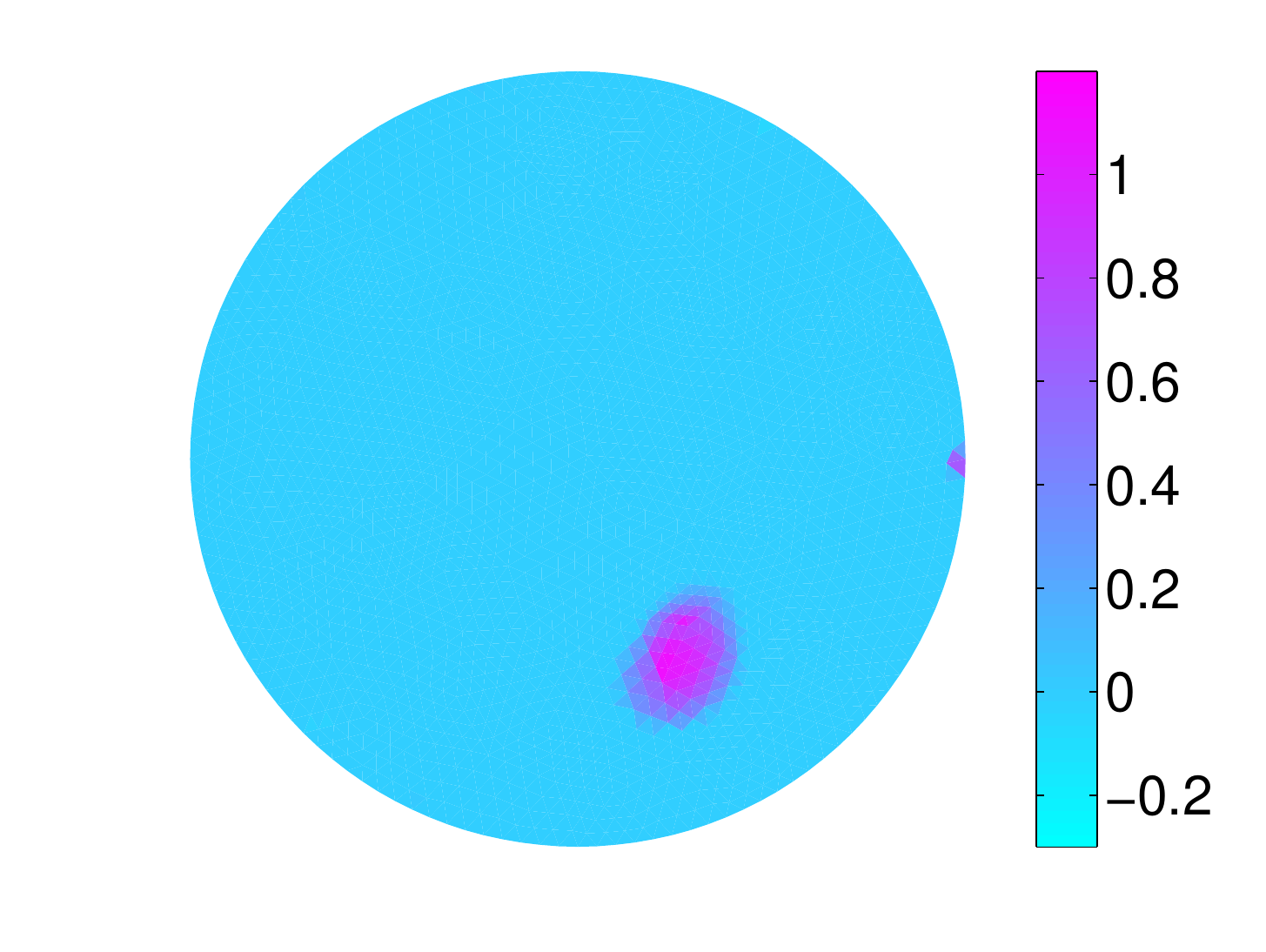}}
  \caption{Numerical results for Example~\ref{exam2}(i) with $1\%$ data noise, with fully known $s_k(\omega)$s. The recoveries
  are obtained by the direct approach.}\label{fig:exam2a}
\end{figure}

\begin{figure}[hbt!]
  \centering
    \subfloat[true $\delta\sigma_k$s]{\includegraphics[width = .3\textwidth]{exam2_inc_true.eps}}\hfill
    \subfloat[recovered $\delta\sigma_1$]{\includegraphics[width = .3\textwidth]{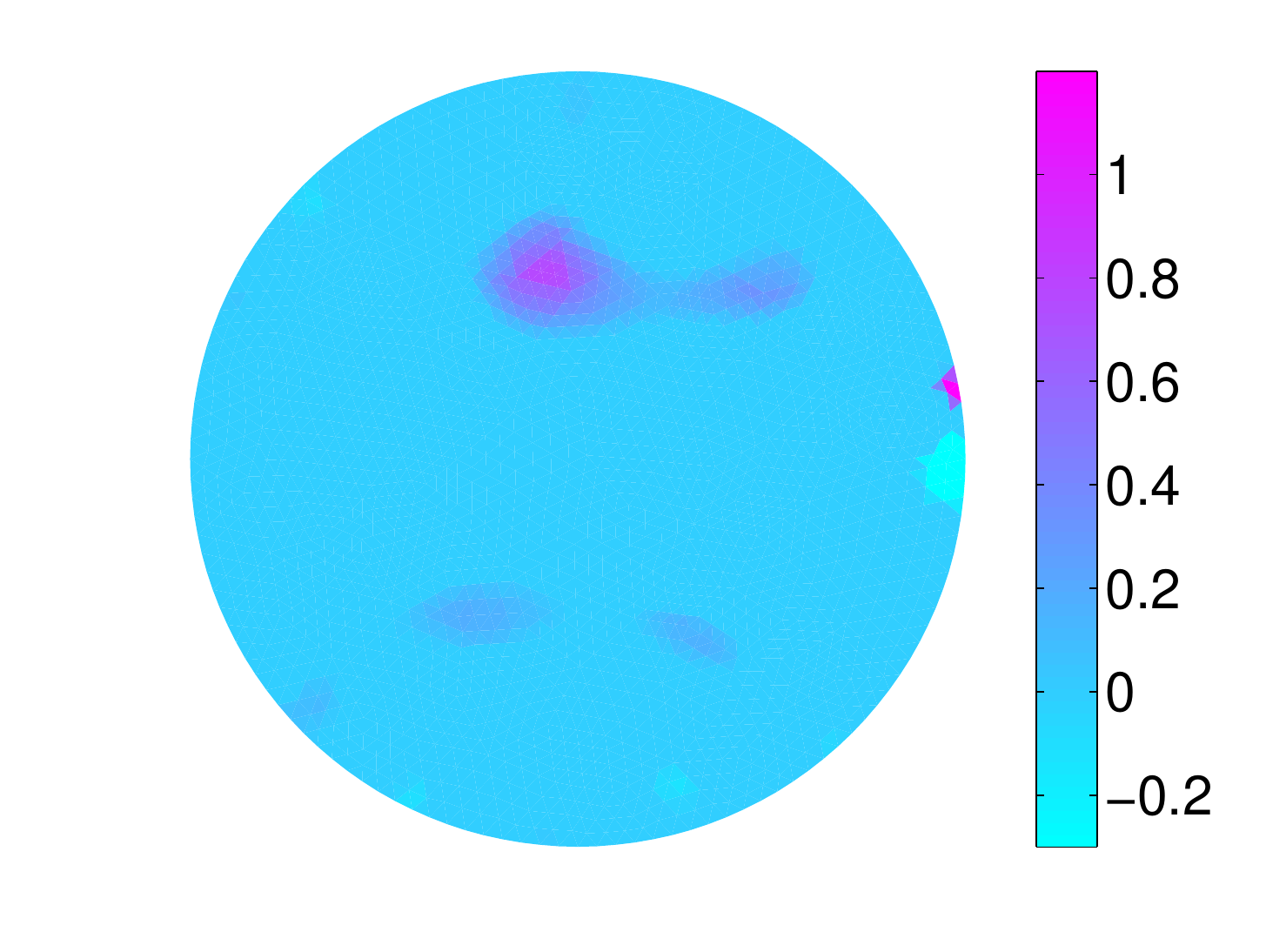}}\hfill
    \subfloat[recovered $\delta\sigma_2$]{\includegraphics[width = .3\textwidth]{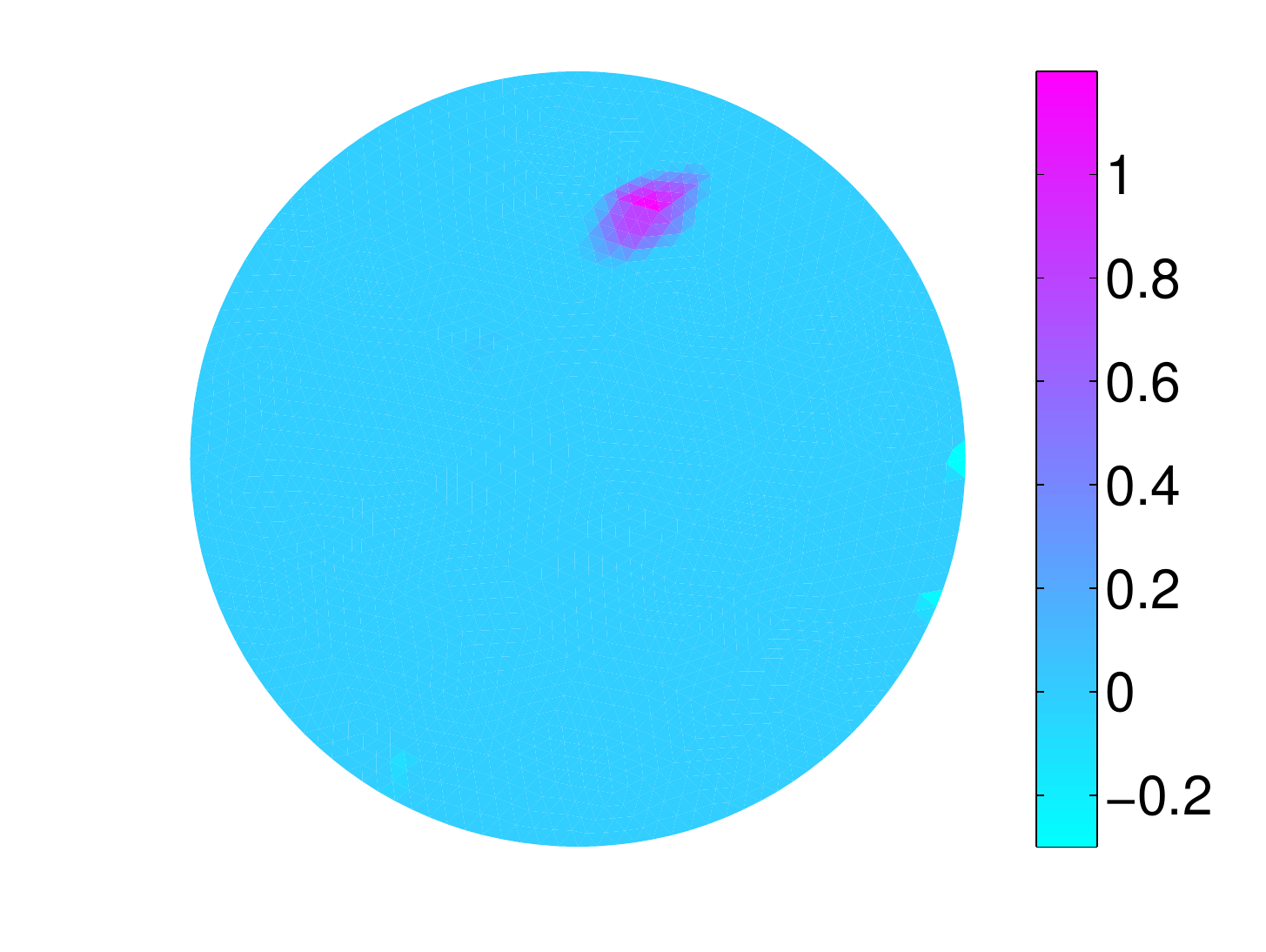}}\hfill

    \subfloat[recovered $\delta\sigma_3$]{\includegraphics[width = .3\textwidth]{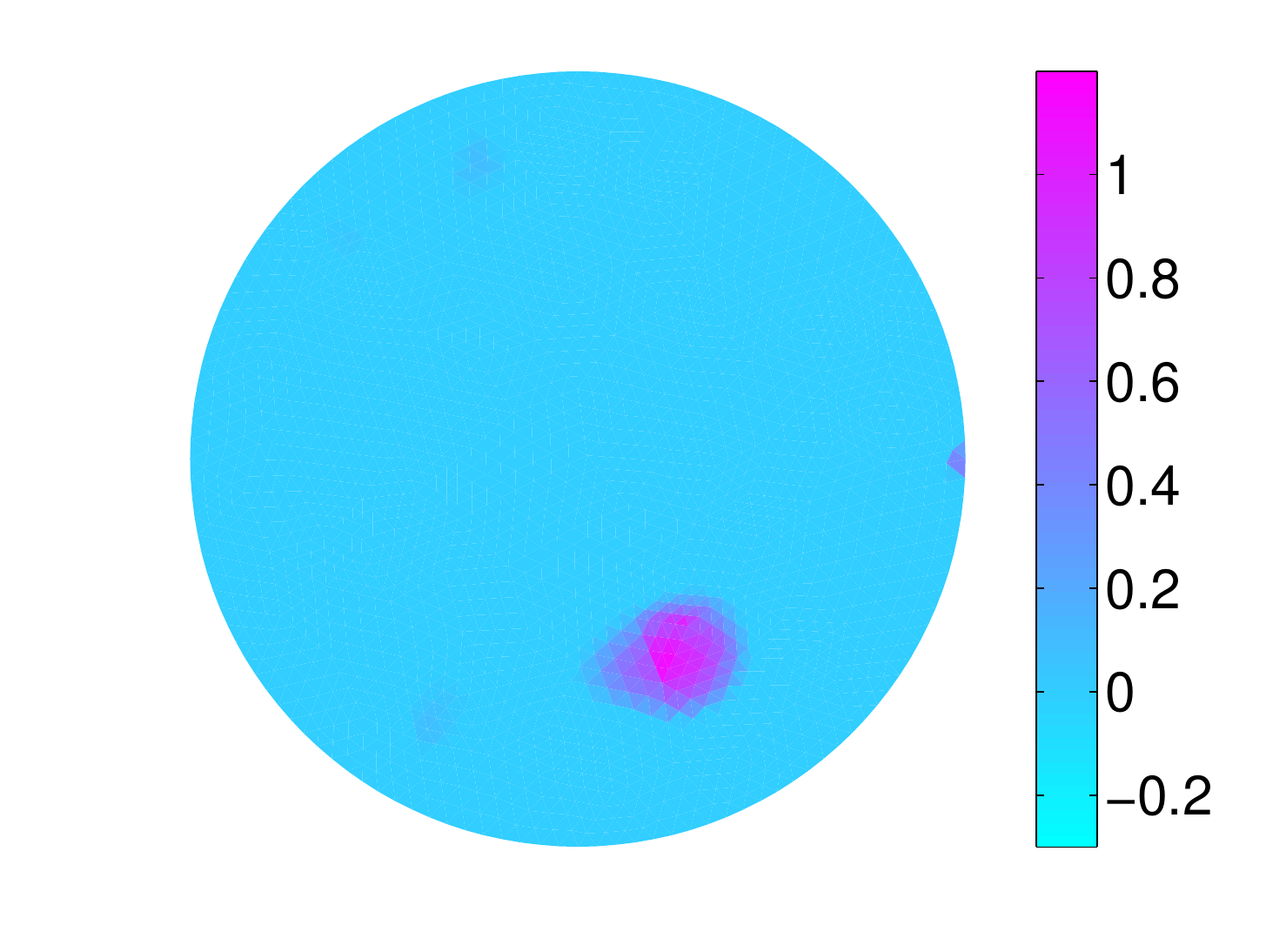}}\hfill
    \subfloat[recovered $\delta\sigma_2$]{\includegraphics[width = .3\textwidth]{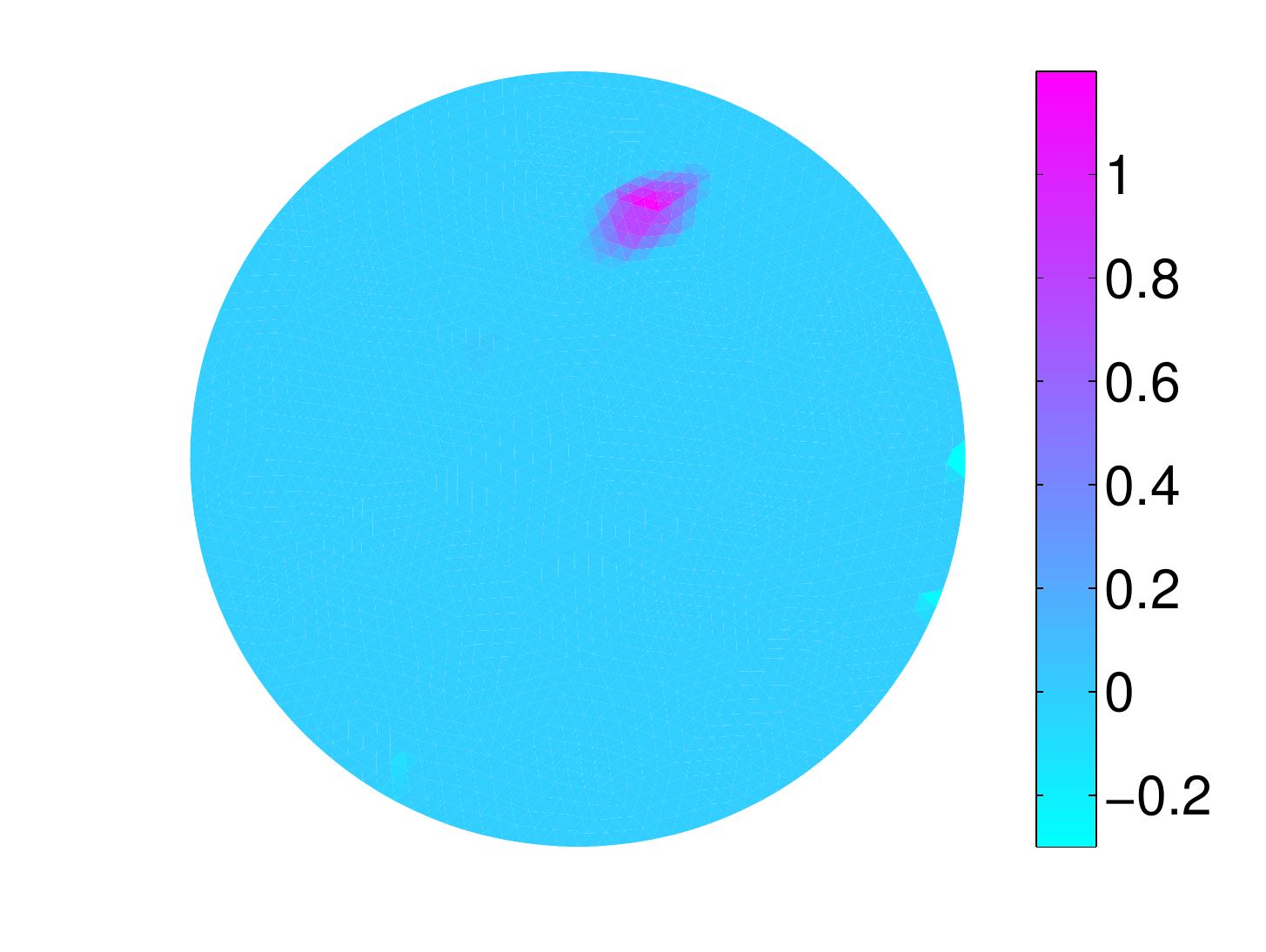}}\hfill
    \subfloat[recovered $\delta\sigma_3$]{\includegraphics[width = .3\textwidth]{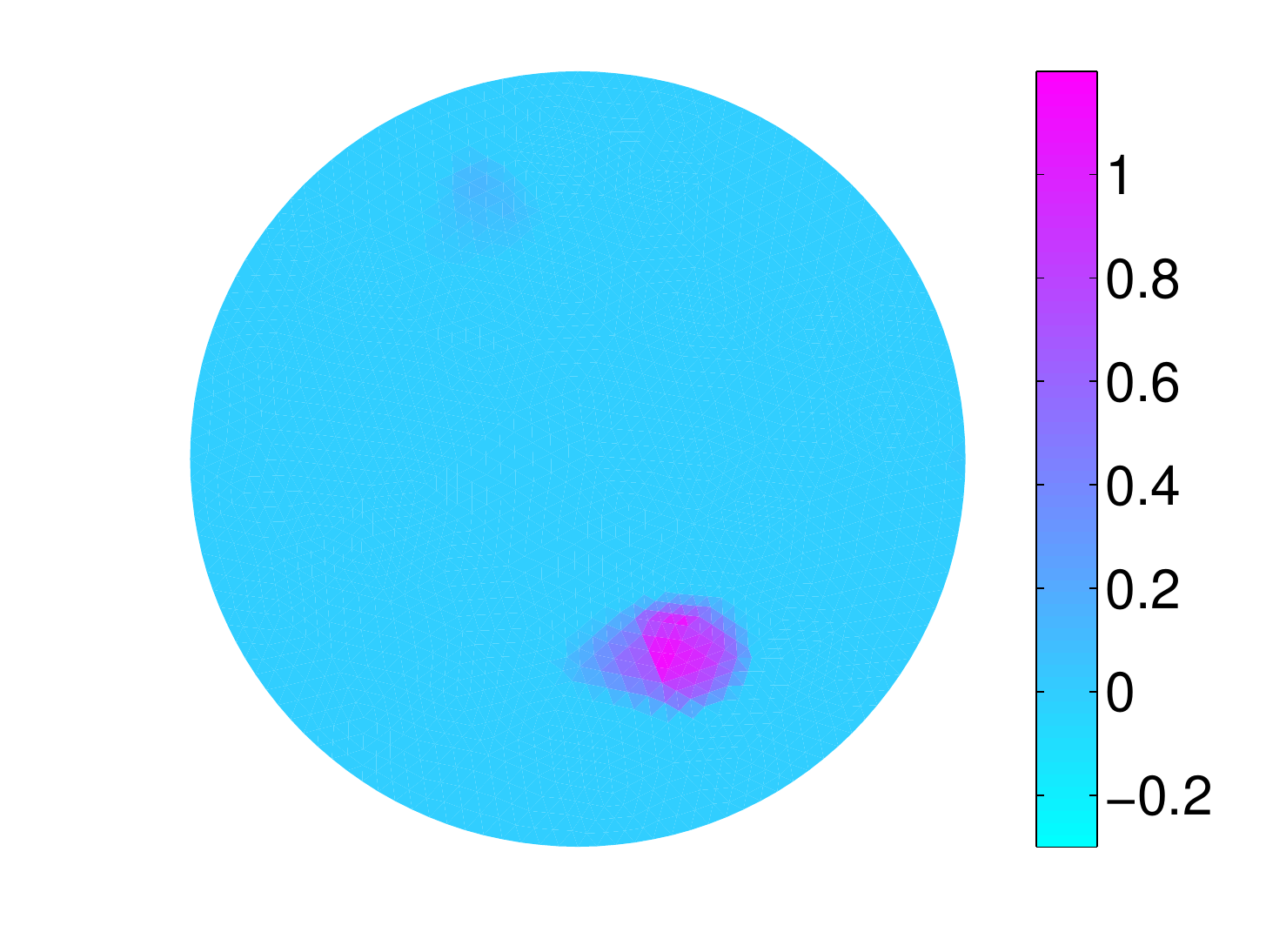}}\hfill
  \caption{Numerical results for Example~\ref{exam2}(ii) with $1\%$ data noise. Here (b)-(d) are the recoveries
  with fully known $s_k(\omega)$ and obtained by the direct approach, while for (e)
  and (f) only $s_2(\omega)$ and $s_3(\omega)$ are known, and the recoveries are obtained by difference imaging.}\label{fig:exam2b}
\end{figure}

Our next example illustrates the case of different conductivities for each inclusion.
\begin{exam}\label{exam2c}
The setup of the example is identical with that of Example \ref{exam2}(i), except that the inclusions on the top
left, top right and bottom have conductivity perturbations of $1.5$, $1$ and $0.5$, respectively.
\end{exam}

The numerical results are presented in Fig. \ref{fig:exam2c}. With different conductivities for
each inclusion, the reconstructions remain fairly reasonable: all three inclusions are well
separated from each other, with their magnitudes accurately estimated, as in Example \ref{exam2}(i).
However, the support of the inclusion on the bottom is slightly distorted, cf. Fig. \ref{fig:exam2cd}.
This is attributed to the smaller magnitude of the inclusion, yielding a higher noise level of
the corresponding linear inversion step, which deteriorates the reconstruction.

\begin{figure}[hbt!]
  \centering
    \subfloat[true $\delta\sigma_k$s]{\label{fig:exam2ca}\includegraphics[trim = 1.5cm .5cm .5cm .5cm, clip=true,width = .23\textwidth]{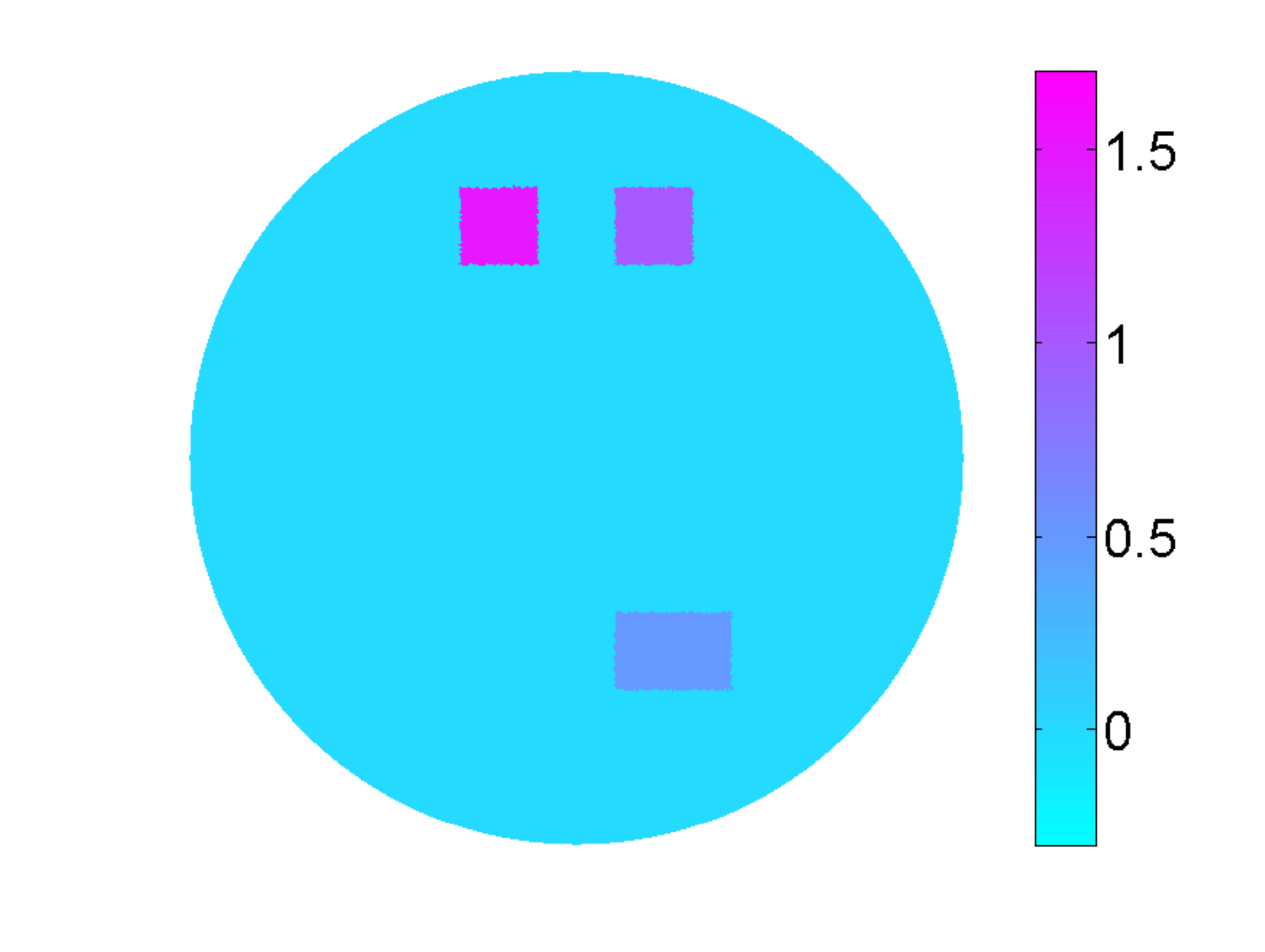}}\hfill
    \subfloat[recovered $\delta\sigma_1$]{\label{fig:exam2cb}\includegraphics[trim = 1.5cm .5cm .5cm .5cm, clip=true,width = .23\textwidth]{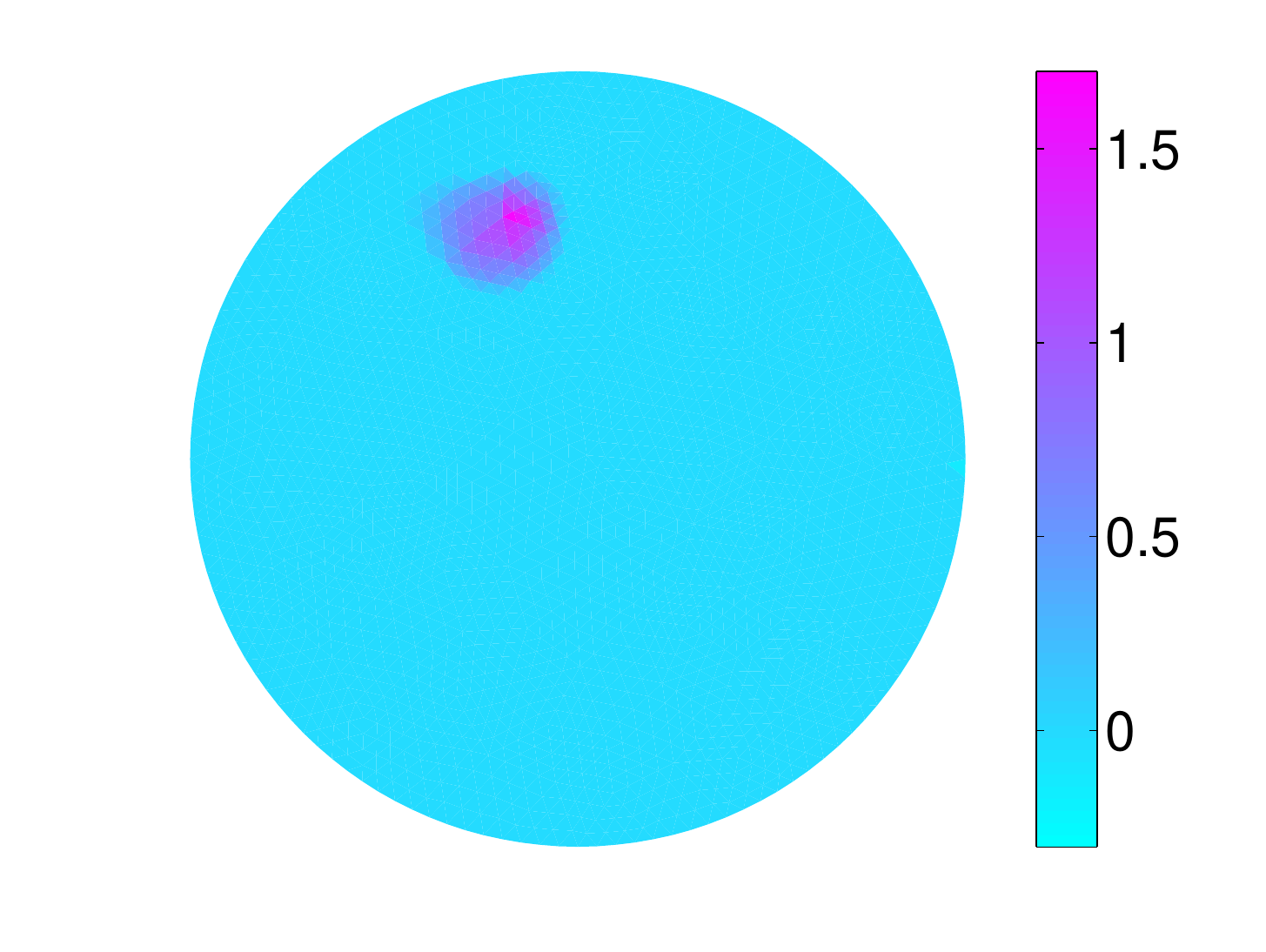}}\hfill
    \subfloat[recovered $\delta\sigma_2$]{\label{fig:exam2cc}\includegraphics[trim = 1.5cm .5cm .5cm .5cm, clip=true,width = .23\textwidth]{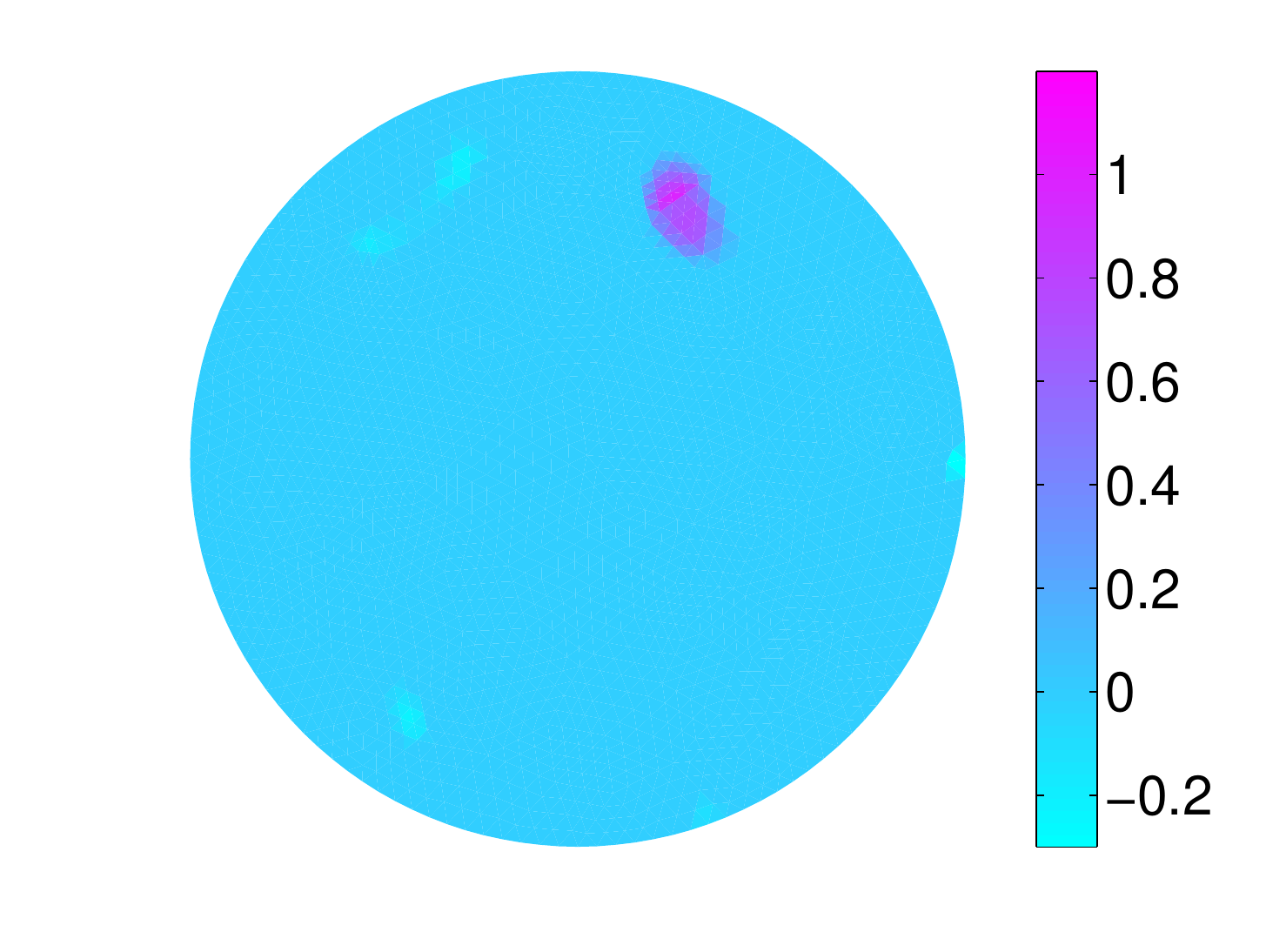}}\hfill
    \subfloat[recovered $\delta\sigma_3$]{\label{fig:exam2cd}\includegraphics[trim = 1.5cm .5cm .5cm .5cm, clip=true,width = .23\textwidth]{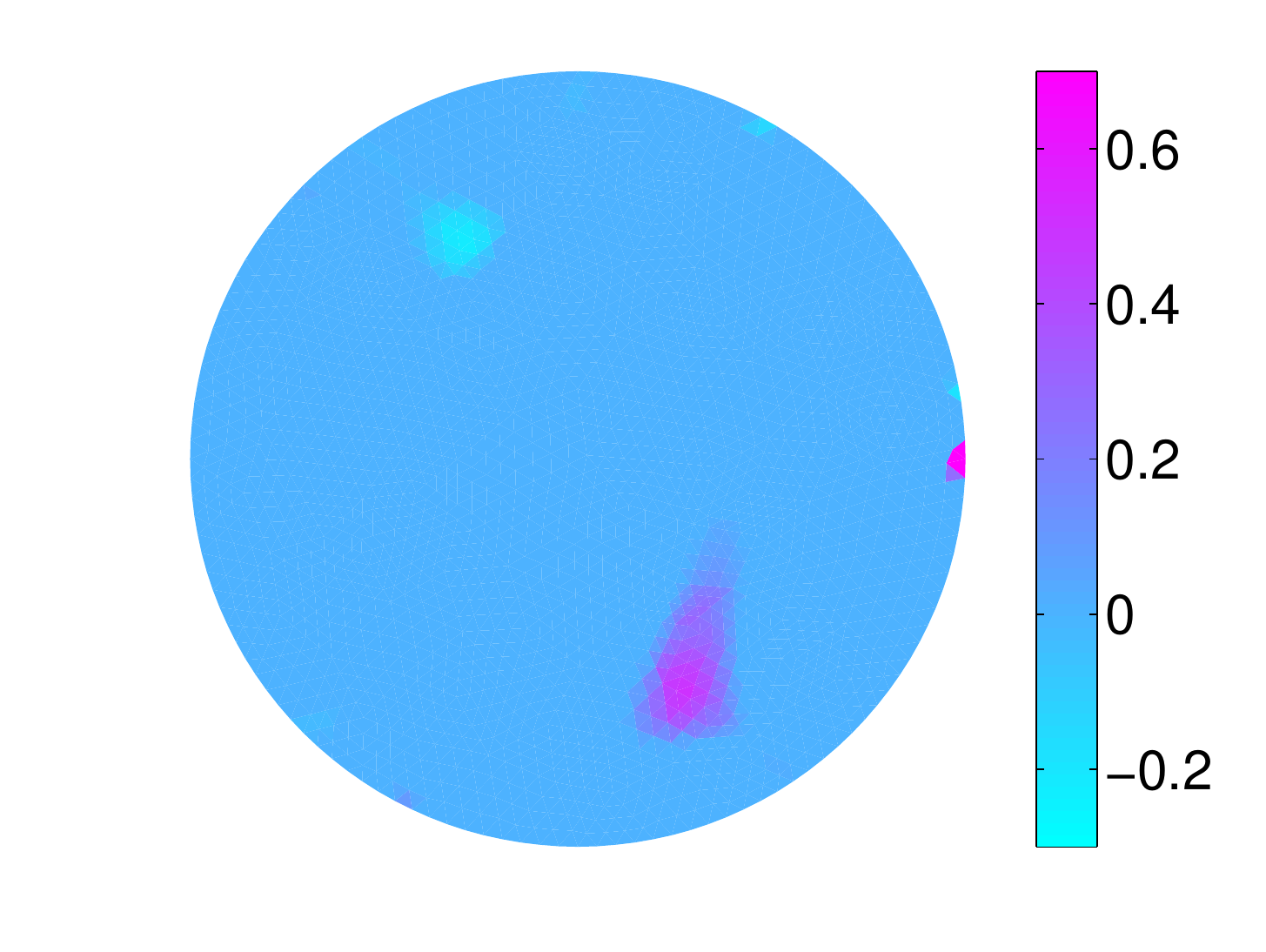}}
  \caption{Numerical results for Example~\ref{exam2c} with $1\%$ data noise, with fully known $s_k(\omega)$s. The recoveries
  are obtained by the direct approach.}\label{fig:exam2c}
\end{figure}

\subsection{Imperfectly Known Boundary}\label{sec:numer-imperf}

Now we illustrate the approach in the case of an imperfectly known boundary. In the first example, the unknown
true domain $\widetilde\Omega$ is an ellipse
centered at the origin with semi-axes $a$ and $b$, $\mathcal{E}_{a,b}=\{(x_1,x_2): x_1^2/a^2+x_2^2/b^2<1\}$, and
the computational domain $\Omega$ is taken to be the unit circle.
\begin{exam}\label{exam3}
Consider two square inclusions on the top and the bottom of the ellipse, with
$s_1(\omega)=0.2\omega+0.2$ and $s_2(\omega)=0.1\omega^2$ (Fig.~\ref{fig:exam3a}). We consider
the following two cases:
\begin{itemize}
  \item[(i)] The true domain $\widetilde\Omega$ is $\mathcal{E}_{a,b}$ with $a=1.1$ and $b=0.9$;
  \item[(ii)] The true domain $\widetilde\Omega$ is $\mathcal{E}_{a,b}$ with $a=1.2$ and $b=0.8$.
\end{itemize}
In either case, we take three frequencies, $\omega_1=0$, $\omega_2=0.5$ and $\omega_3=1$.
\end{exam}

\begin{figure}[hbt!]
  \centering
    \subfloat[true $\delta\sigma_k$s]{\includegraphics[width = .3\textwidth]{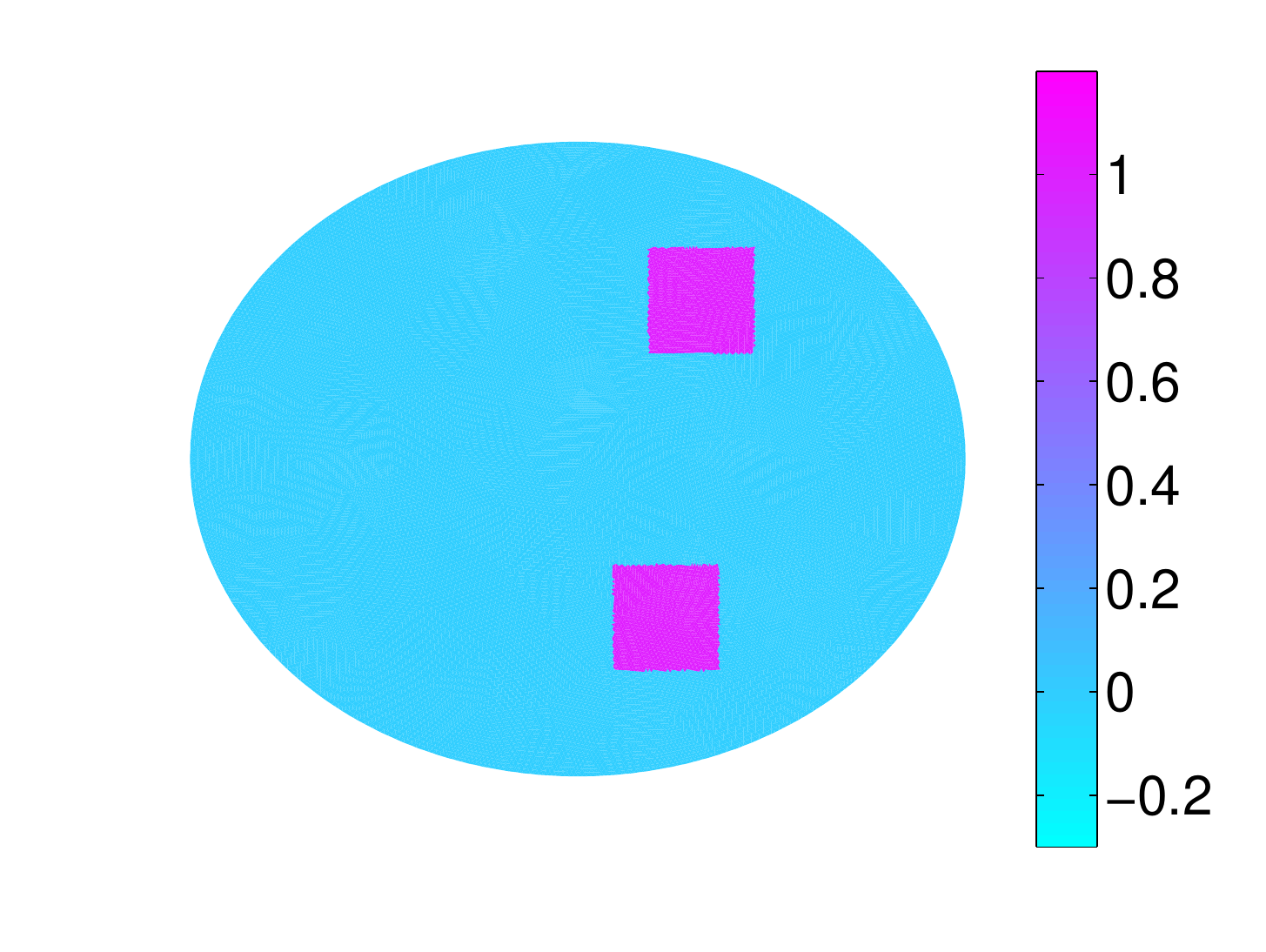}}\hfill
    \subfloat[recovered $\delta\sigma_1$]{\includegraphics[width = .3\textwidth]{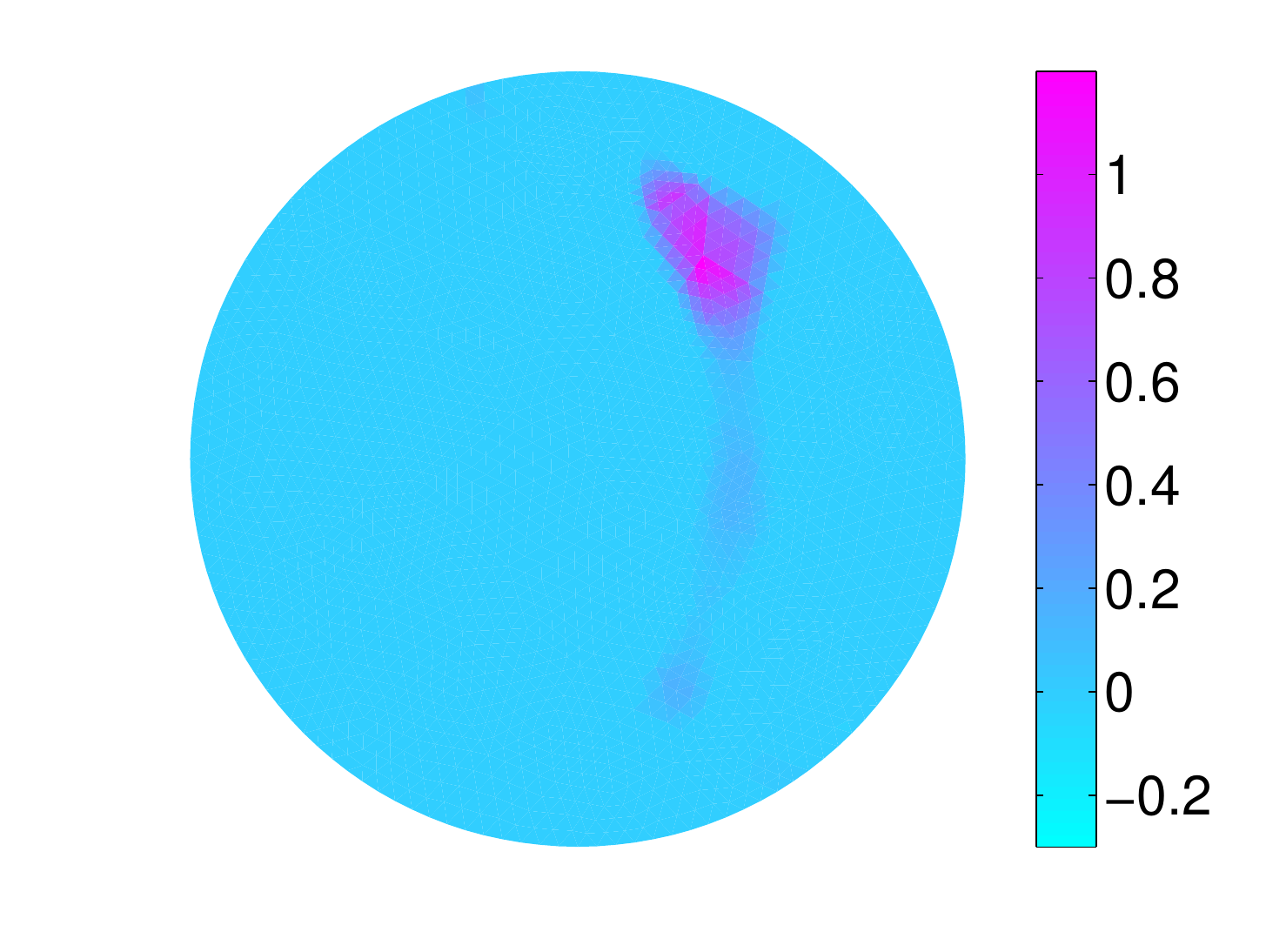}}\hfill
    \subfloat[recovered $\delta\sigma_2$]{\includegraphics[width = .3\textwidth]{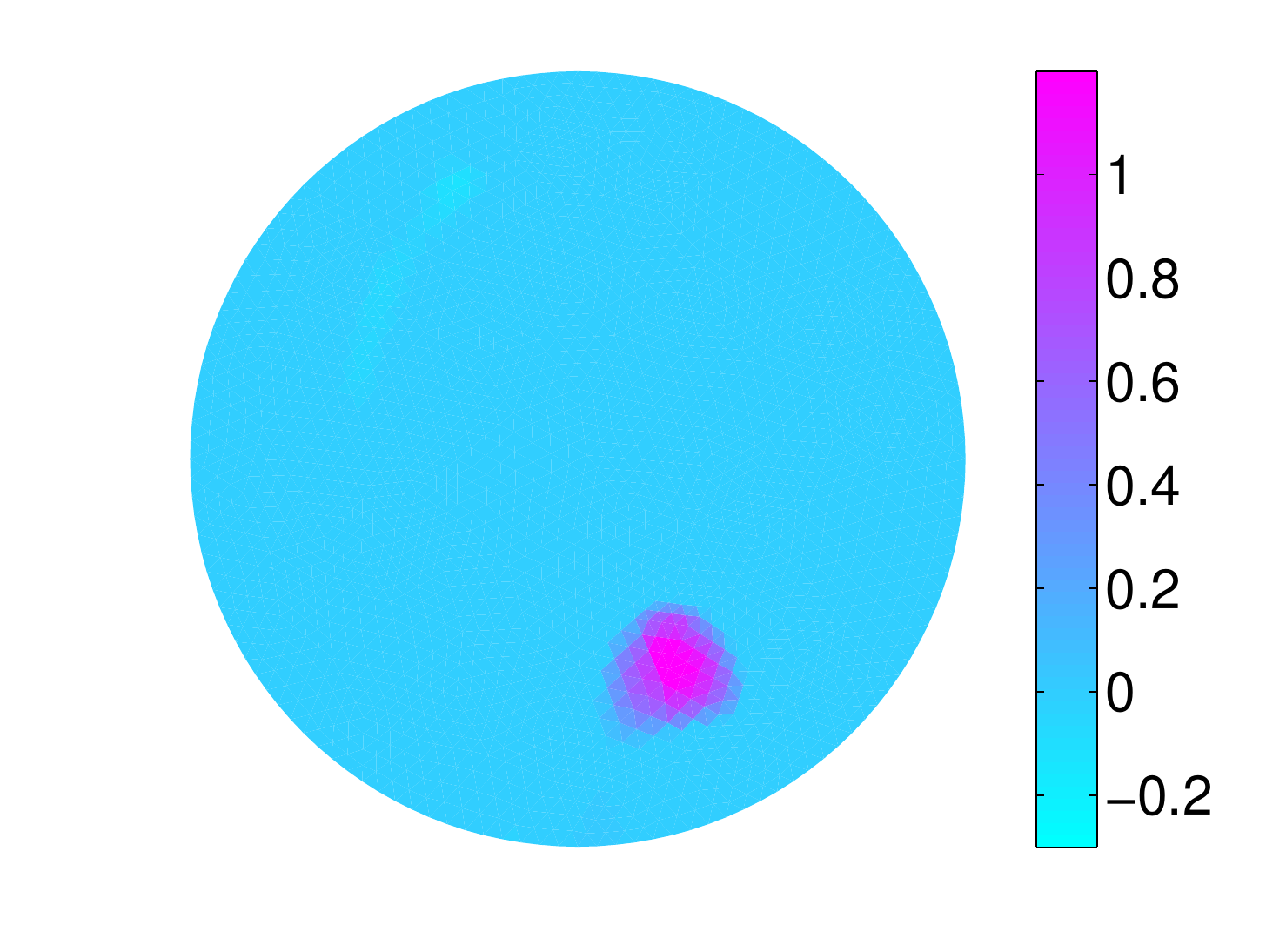}}
  \caption{Numerical results for Example~\ref{exam3}(i) with $0.1\%$ data noise, fully known
  $s_k(\omega)$. The recoveries are obtained using difference imaging.}\label{fig:exam3a}
\end{figure}

\begin{figure}[hbt!]
  \centering
   \subfloat[true $\delta\sigma_k$s]{\label{fig:exam3ba}\includegraphics[width = .3\textwidth]{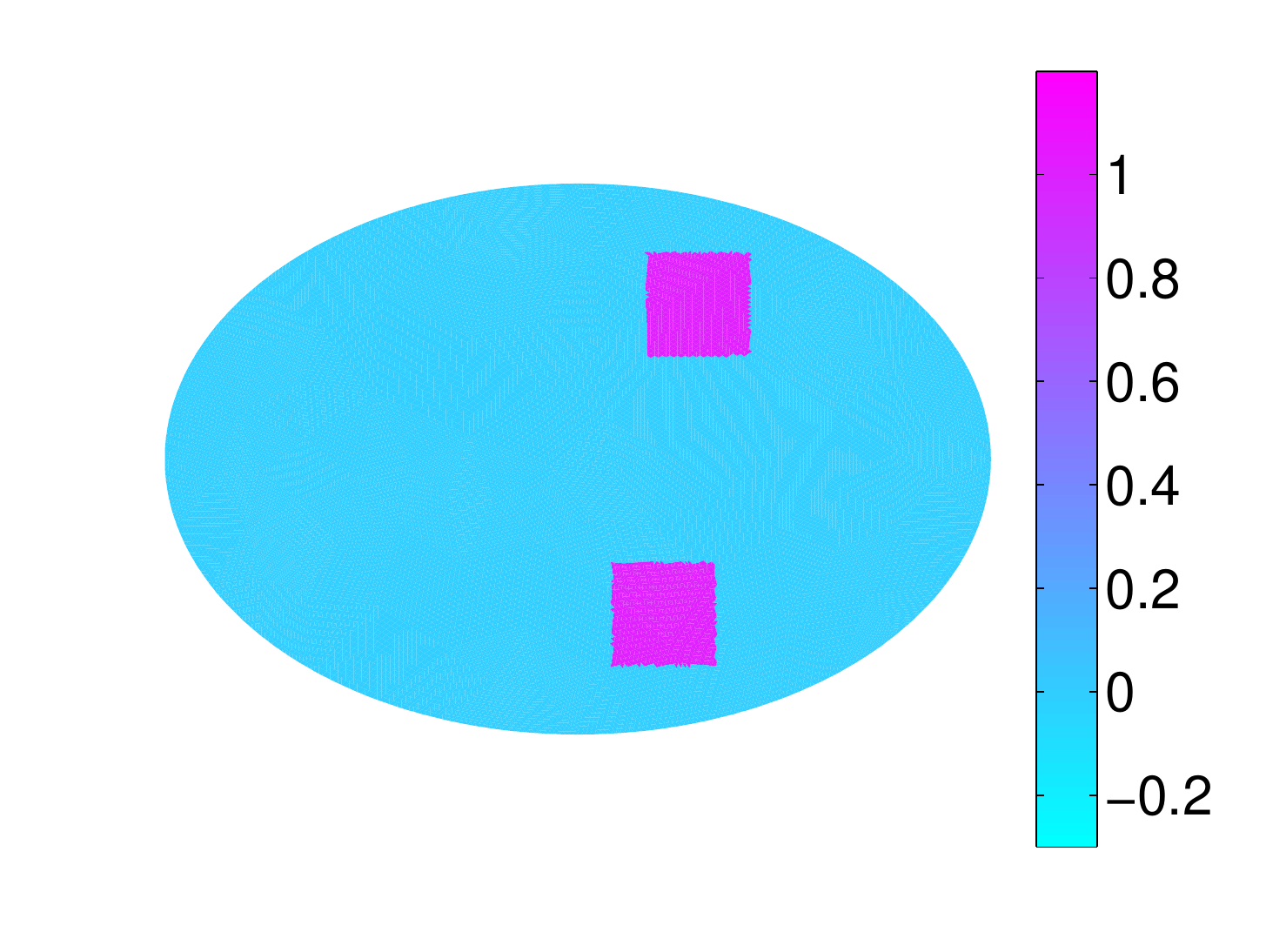}}\hfill
   \subfloat[recovered $\delta\sigma_1$]{\label{fig:exam3bb}\includegraphics[width = .3\textwidth]{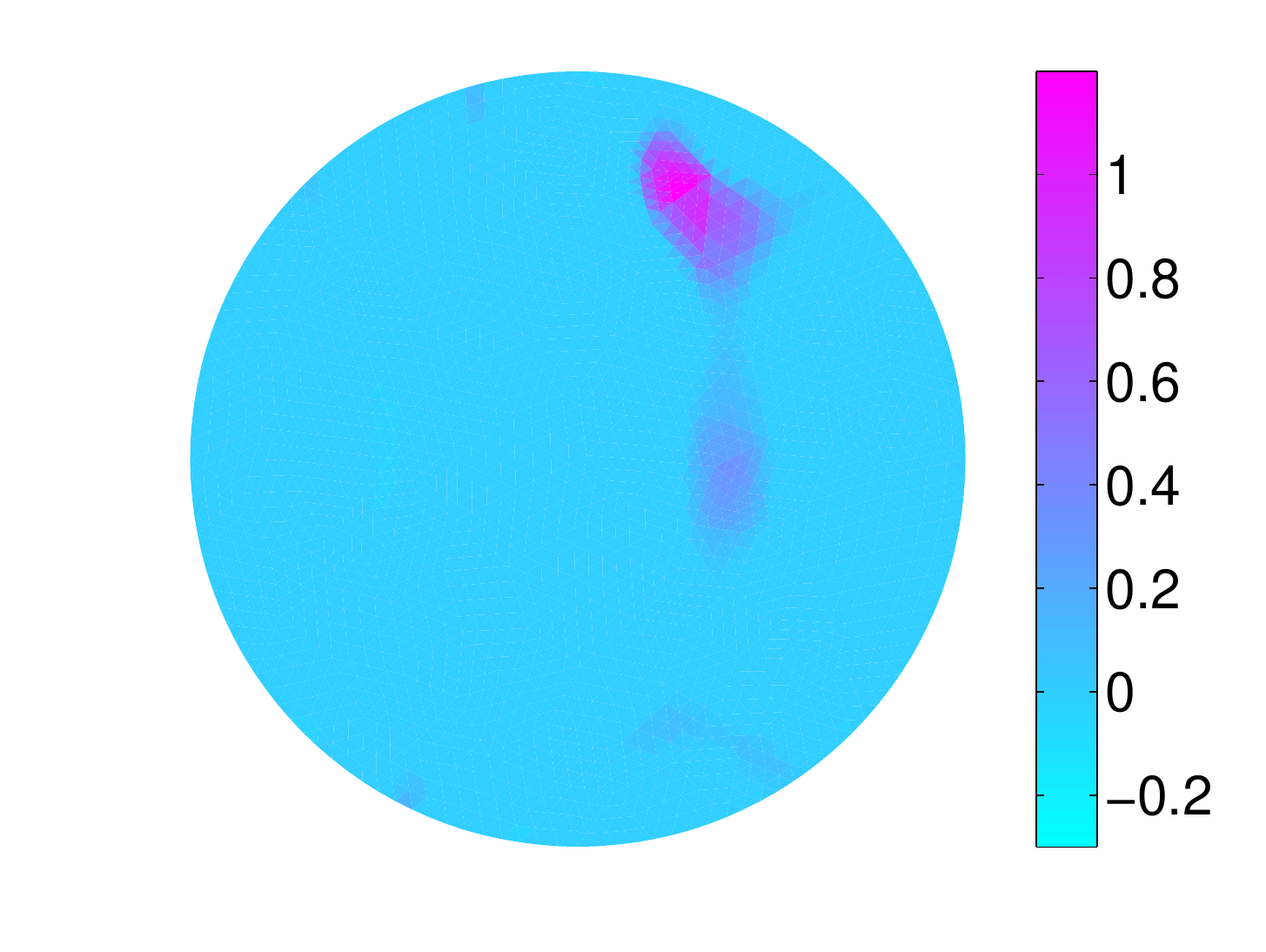}}\hfill
   \subfloat[recovered $\delta\sigma_2$]{\includegraphics[width = .3\textwidth]{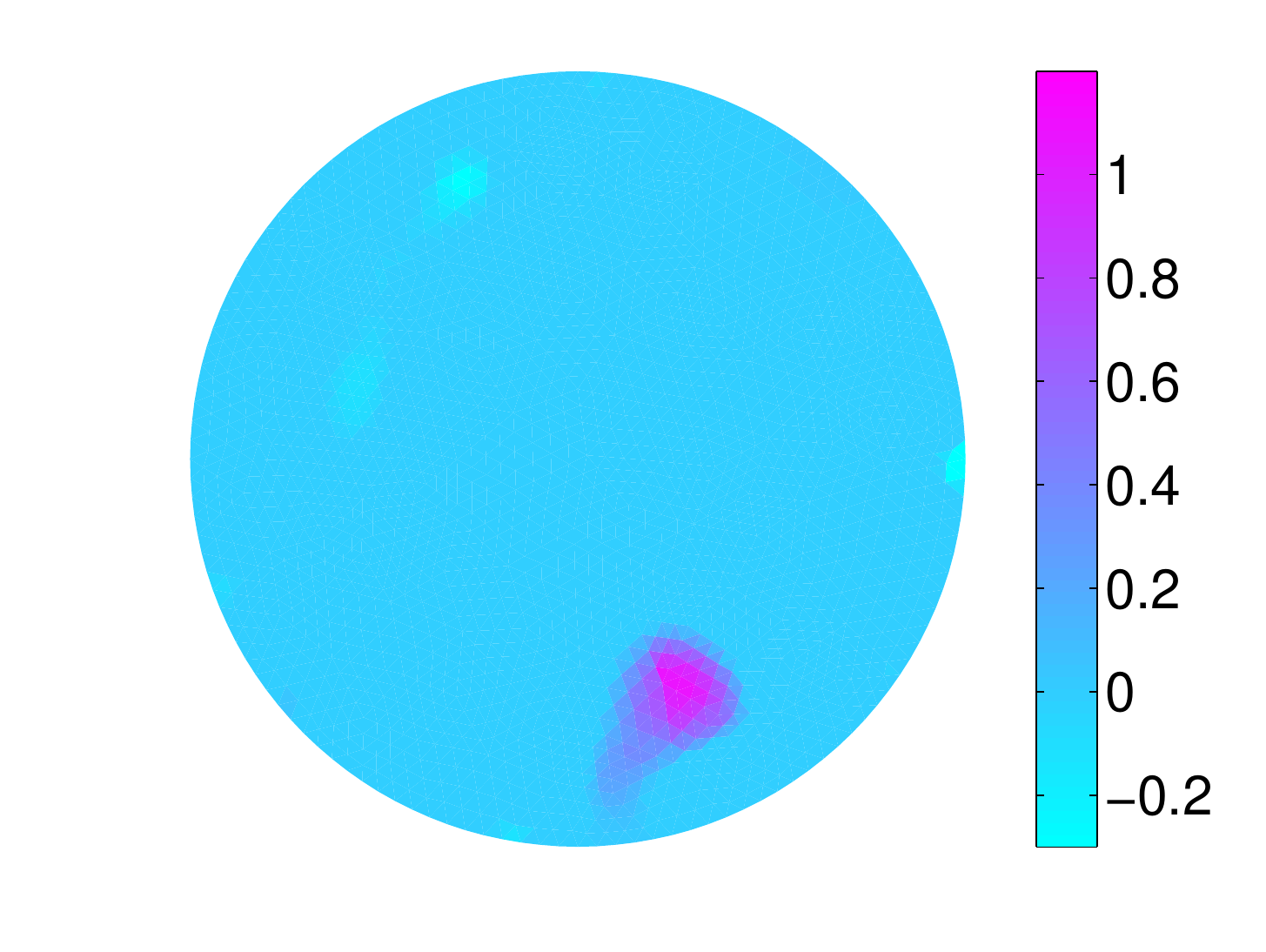}}

   \subfloat[recovered $\delta\sigma_0$]{\includegraphics[width = .3\textwidth]{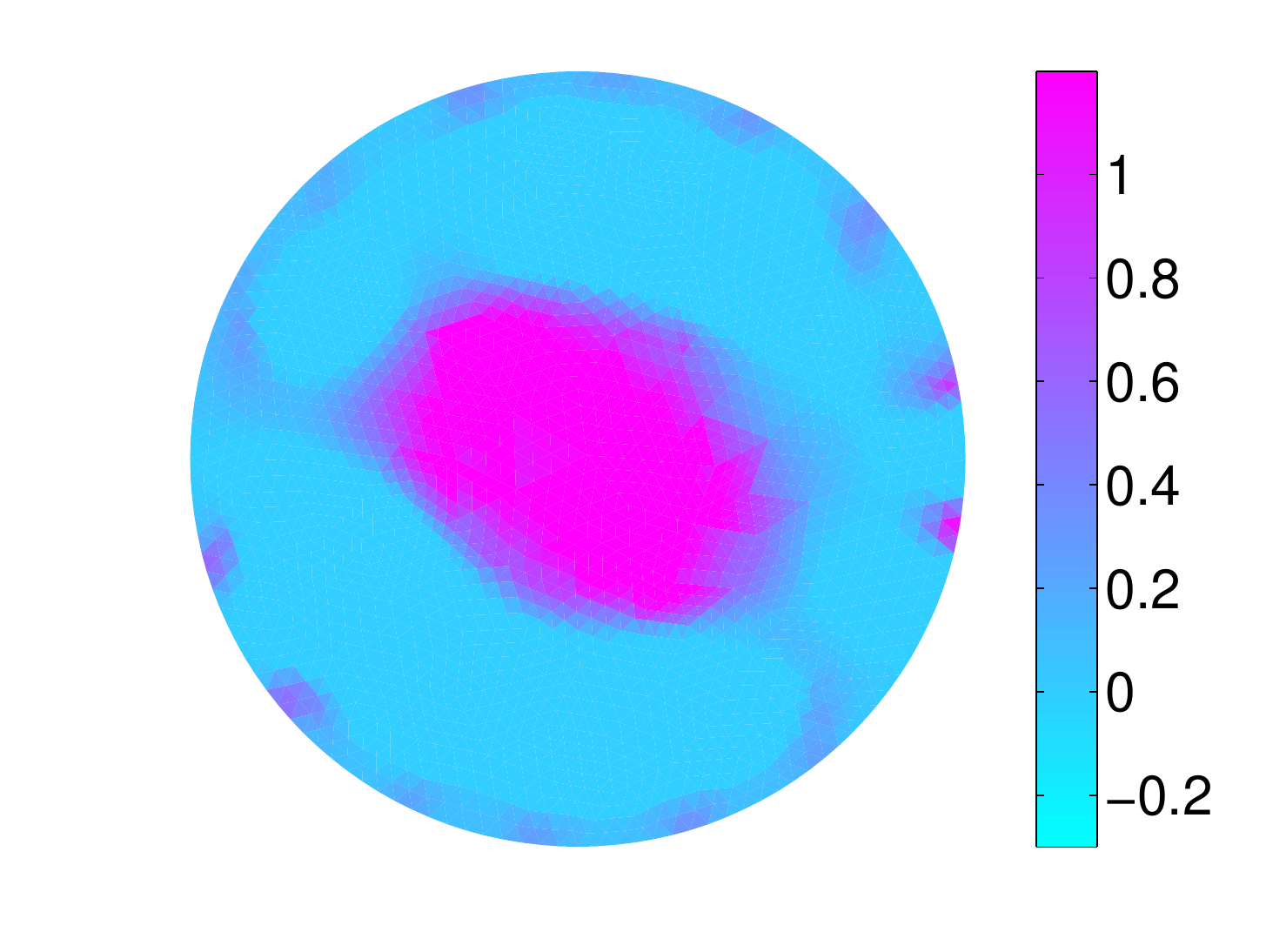}}\hfill
   \subfloat[recovered $\delta\sigma_1$]{\includegraphics[width = .3\textwidth]{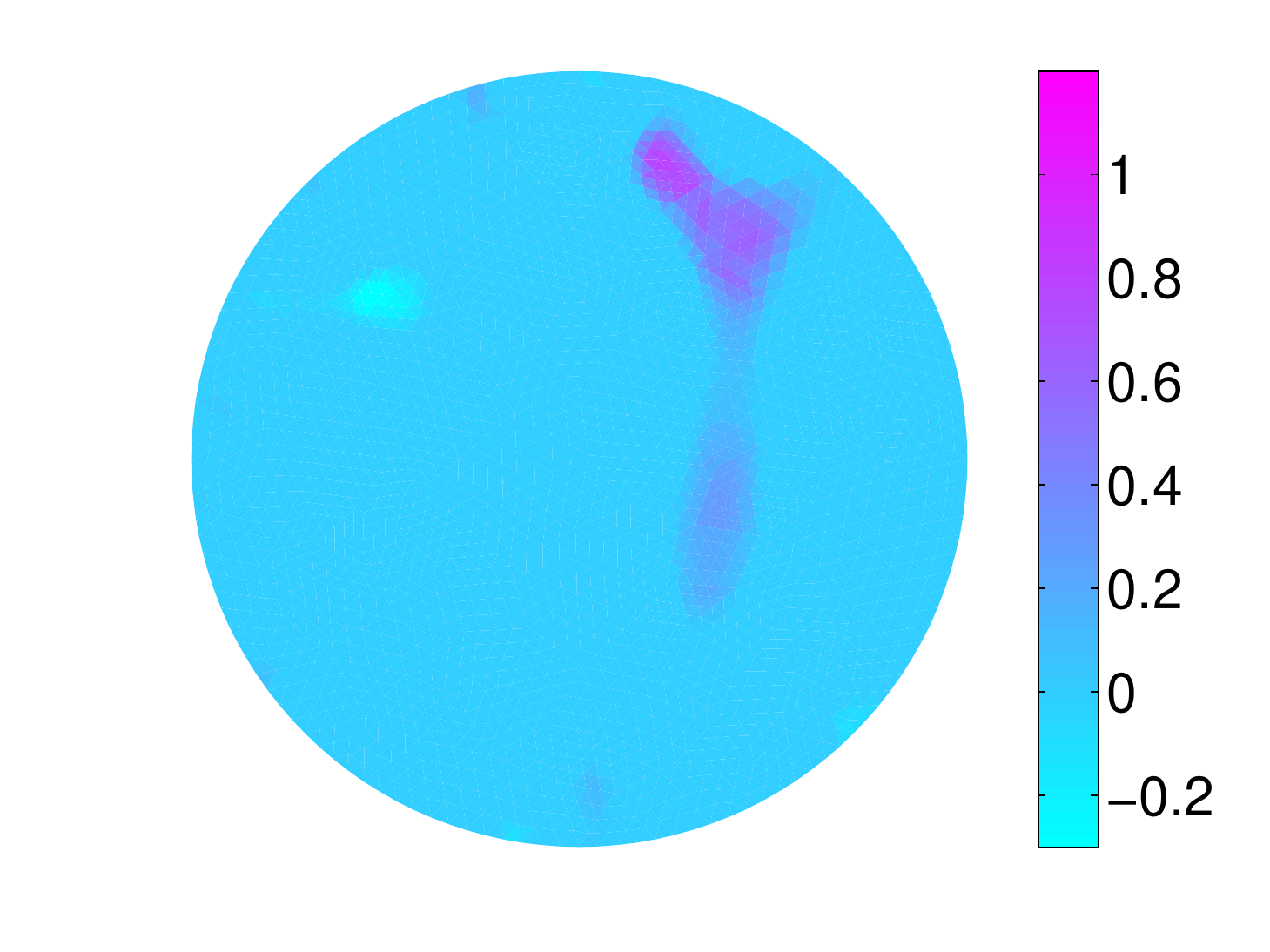}}\hfill
   \subfloat[recovered $\delta\sigma_2$]{\includegraphics[width = .3\textwidth]{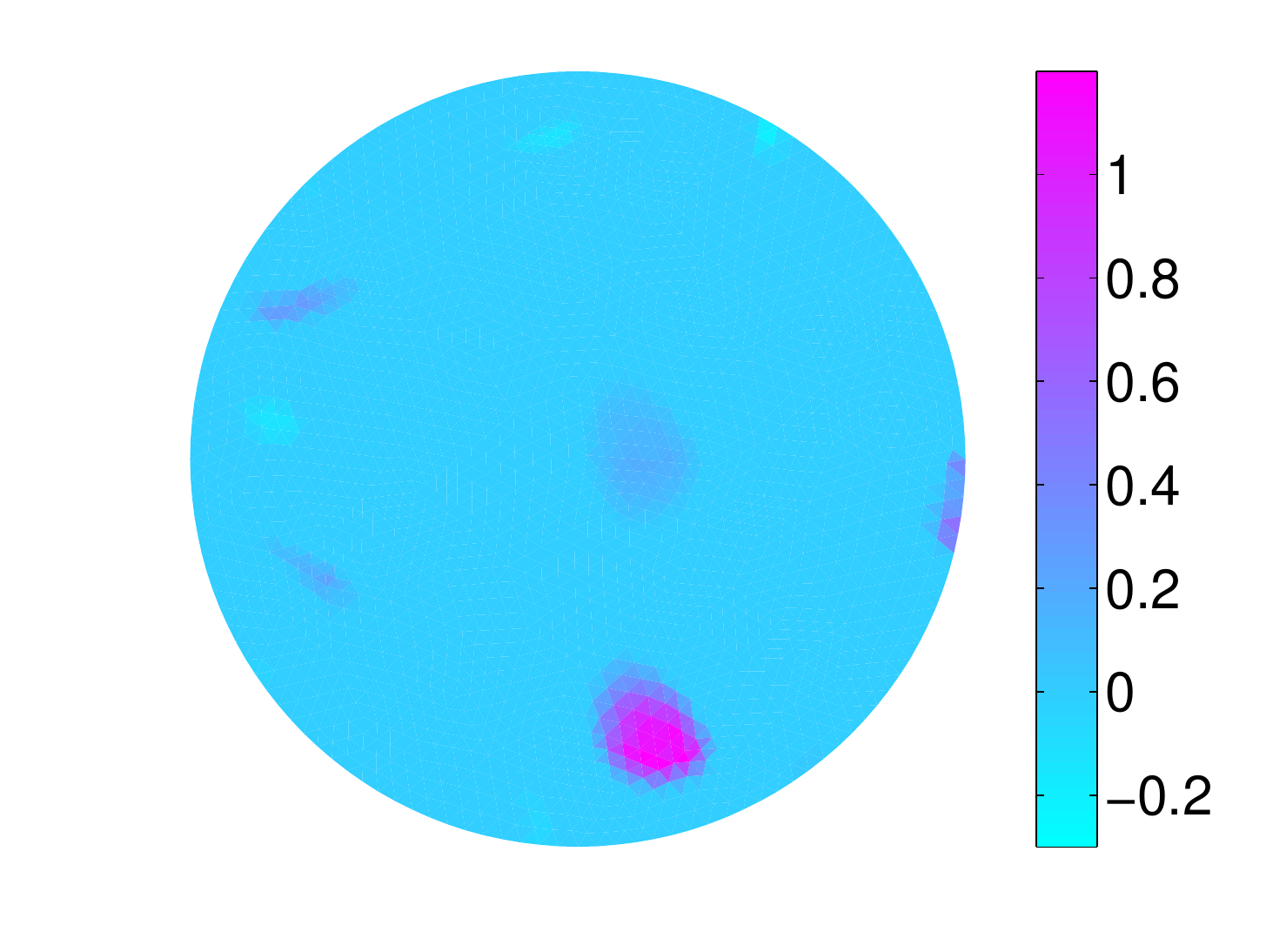}}

  \caption{Numerical results for Example~\ref{exam3}(ii) with $0.1\%$ data noise, fully known $s_k(\omega)$s. The recoveries
  in (b)-(c) are based on difference imaging,
  and those in (d)-(f) the direct approach. }\label{fig:exam3b}
\end{figure}

The results are given in Figs.~\ref{fig:exam3a} and \ref{fig:exam3b} with $0.1\%$ noise
in the data, for (i) and (ii), respectively. Although not presented, we note that
the static imaging can only produce useless recoveries, due to
the presence of significant modeling errors. Numerically one can verify that for both cases,
the contribution from domain deformation is much larger than that of the
inclusions, which justifies the smaller noise level $0.1\%$. By exploiting the spectral
incoherence, mfEIT can separate different contributions, and hence recover each inclusion accurately.

From Fig.~\ref{fig:exam3a}, difference imaging
can recover the inclusions accurately, and they are well separated, due to their
incoherent $s_k(\omega)$s. However, the shape and location of the recovery tend to
be slightly deformed. This concurs with the
analysis in  Section~\ref{sec:cem}: the unknown boundary induces deformed
conductivity of the inclusions, in addition to the anisotropic component.

In Fig.~\ref{fig:exam3b} we present  the results related to Example~\ref{exam3}(ii).
The preceding observations on difference imaging still hold, cf.\
Figs.~\ref{fig:exam3ba} and \ref{fig:exam3bb}.  The direct
approach works equally well: the recovered $\delta\sigma_1$
and $\delta\sigma_2$ are fairly accurate;
and the results are comparable with those obtained by difference imaging.
The recovered $\delta\sigma_0$ contains only the spurious conductivity induced
by the domain deformation. Should there be any true inclusion $\delta\sigma_0$ corresponding
to $s_0(\omega)$, it will be washed away by the error
$\epsilon\Psi$, cf. \eqref{eqn:lin-inv-cem-deformed-2}. The preceding discussions fully
confirm the analysis in  Section~\ref{sec:cem}: mfEIT is capable
of discriminating the perturbation due to domain deformation from the
inclusions by either the direct approach or
difference imaging.

Last we present one example where the electrodes are misplaced, but their lengthes
do not change, i.e., the factor $z$ in the boundary integral
can be set to the unit (see Example~\ref{ex:z=1}). This is a special case of the imperfectly known boundary case,
where the forward map $F$ maps the domain $\Omega$ onto itself. However, the forward
map is not the identity or a rotation operator, and thus it will induce an anisotropic
conductivity, especially in the regions near the boundary.
\begin{exam}\label{exam4}
The true domain $\widetilde\Omega$ is identical with the computational domain $\Omega$,
the unit circle, but every other electrode is shifted by an angle of $\pi/32$, while
the length of each electrode remains unchanged; see Fig.~\ref{fig:electrodeb} for a
schematic illustration. There are two rectangular inclusions, on the top and on the
bottom of the ellipse, with spectral profiles $s_1(\omega)=0.2\omega+0.2$ and $s_2(\omega)=0.1\omega^2$, respectively.
We take three frequencies $\omega_1=0$, $\omega_2=0.5$ and $\omega_3=1$.
\end{exam}

\begin{figure}[htb!]
  \subfloat[true $\delta\sigma_i$s]{\includegraphics[width=.3\textwidth]{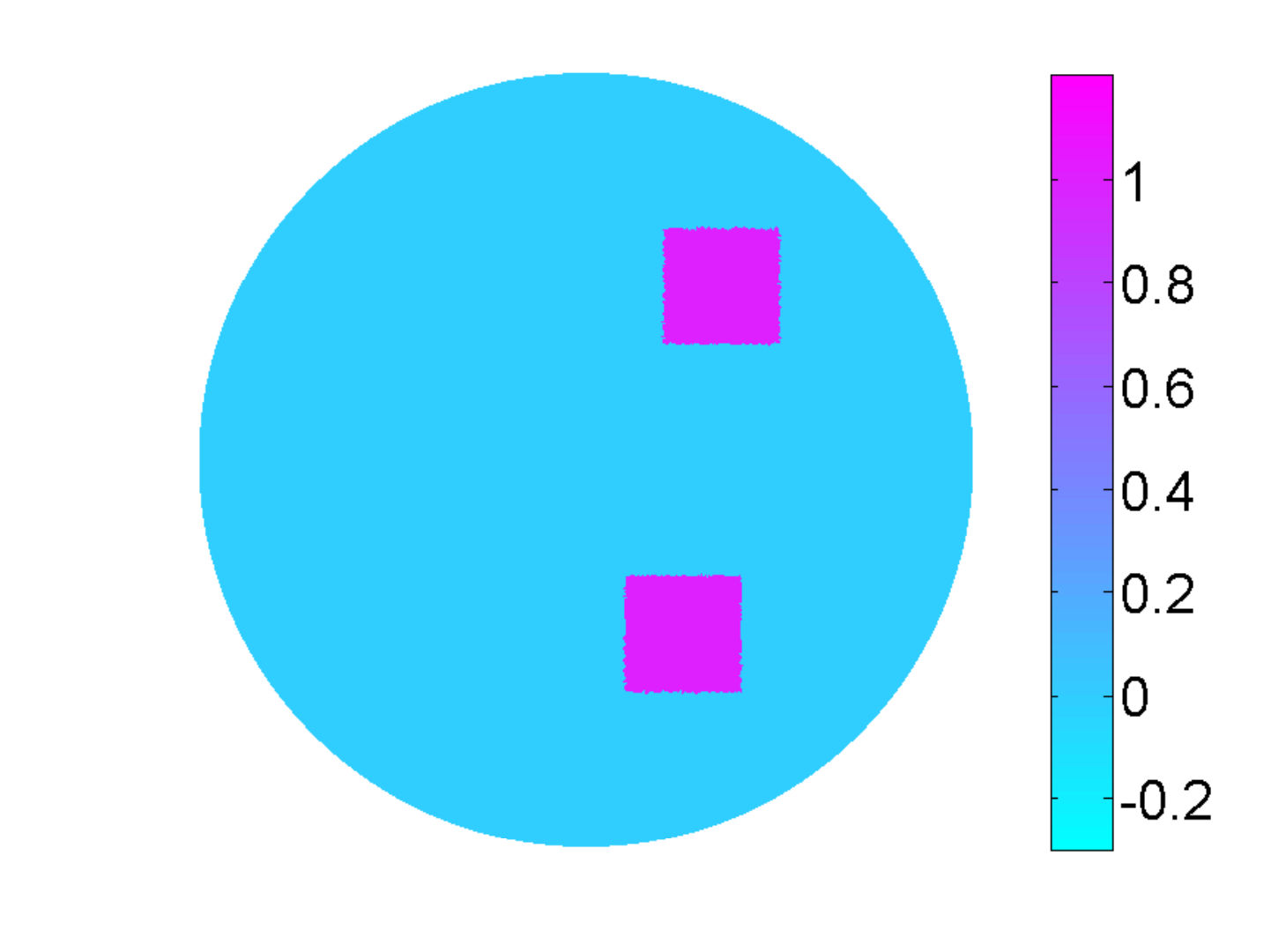}}\hfill
  \subfloat[recovered $\delta\sigma_1$]{\includegraphics[width=0.3\textwidth]{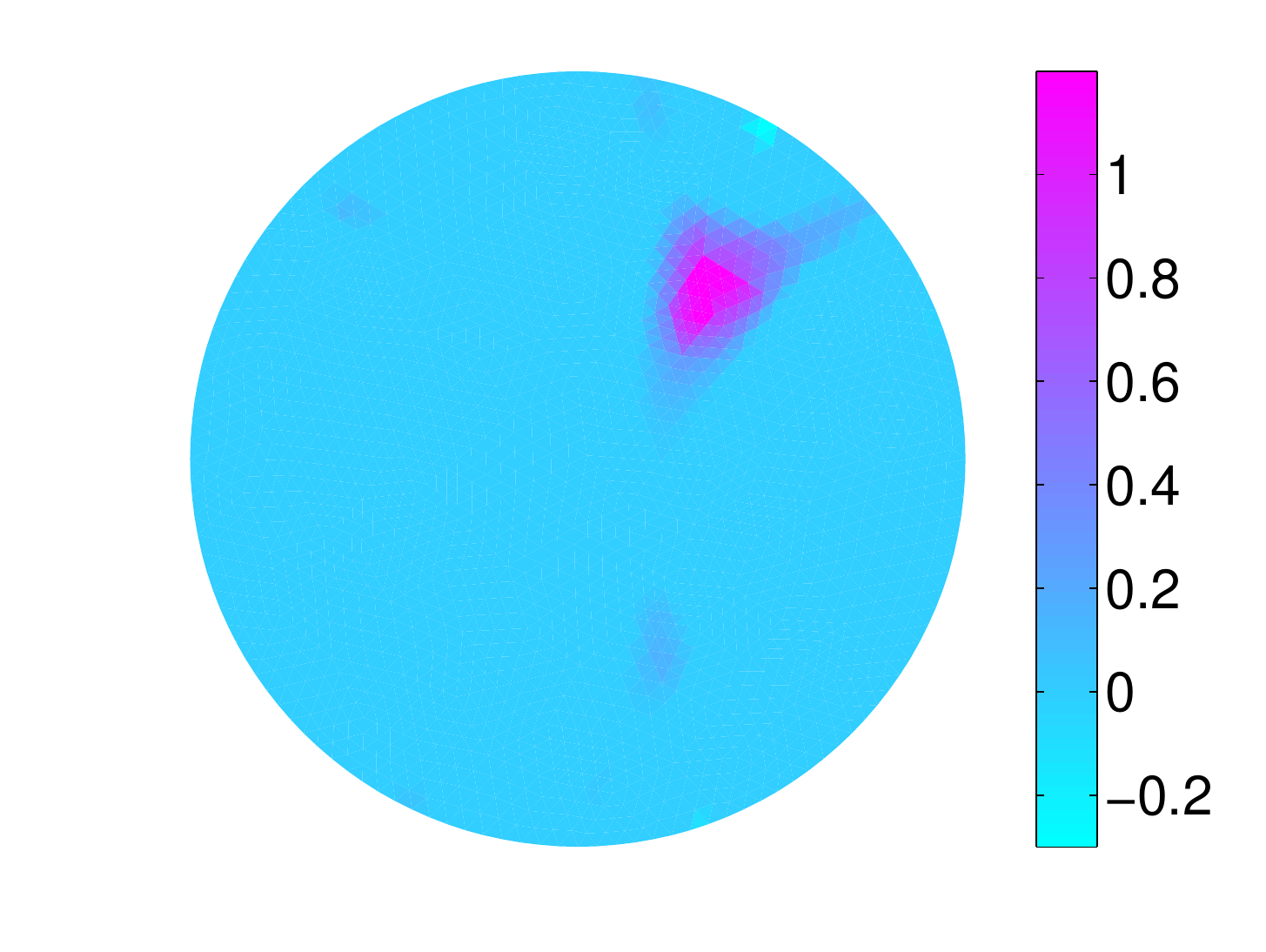}}\hfill
  \subfloat[recovered $\delta\sigma_2$]{\includegraphics[width=0.3\textwidth]{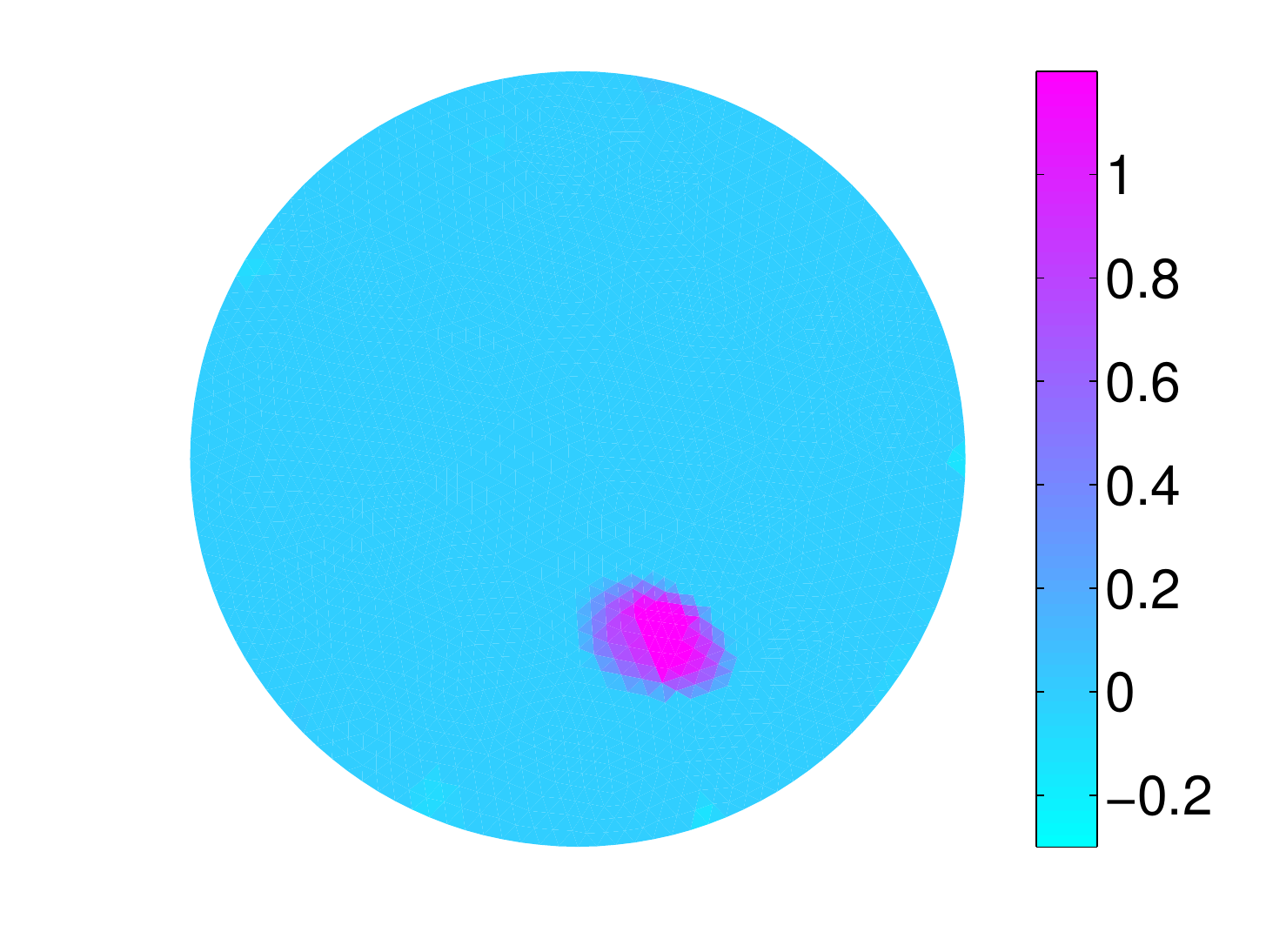}}

  \subfloat[recovered $\delta\sigma_0$]{\includegraphics[width=0.3\textwidth]{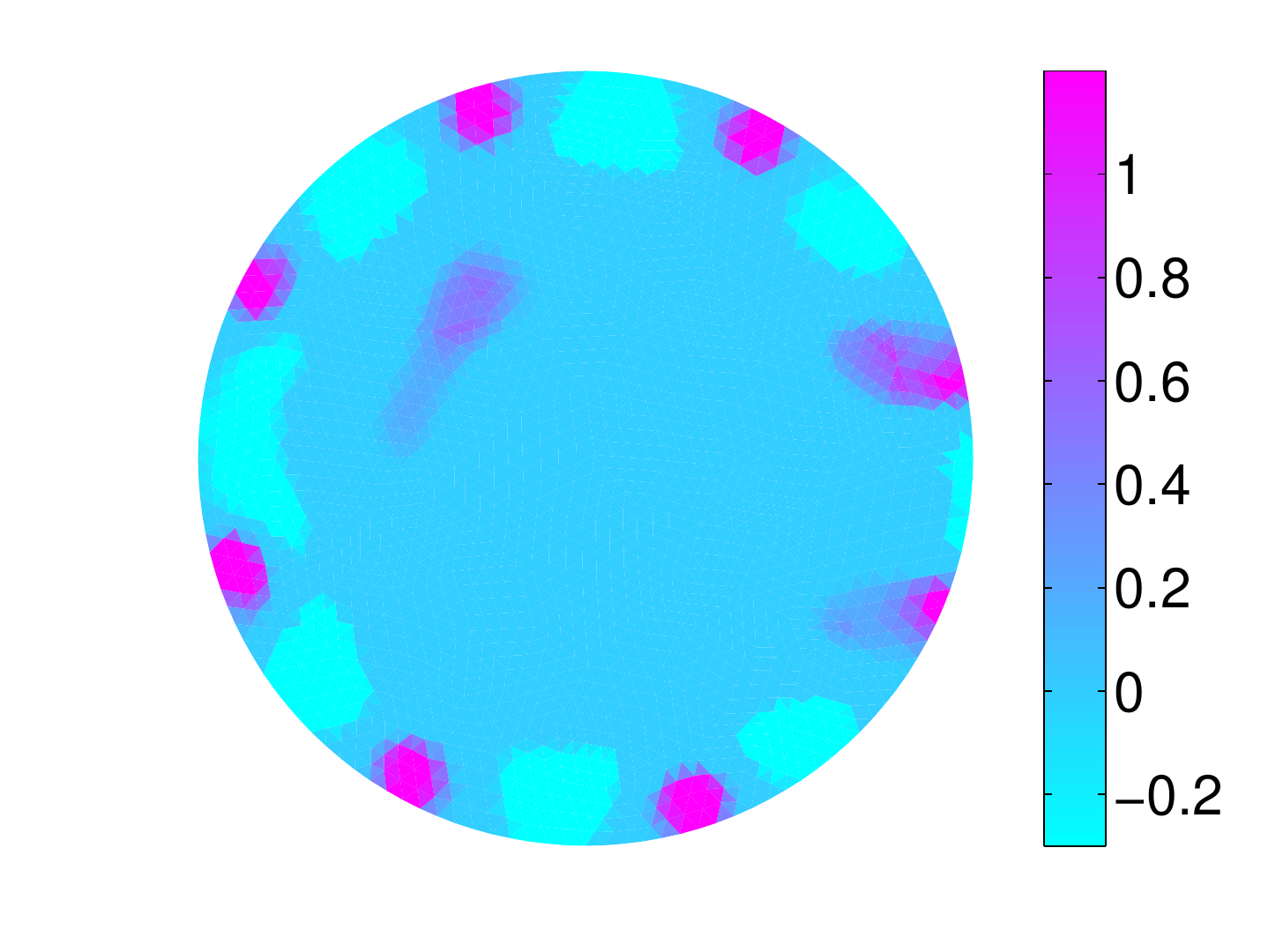}}\hfill
  \subfloat[recovered $\delta\sigma_1$]{\includegraphics[width=0.3\textwidth]{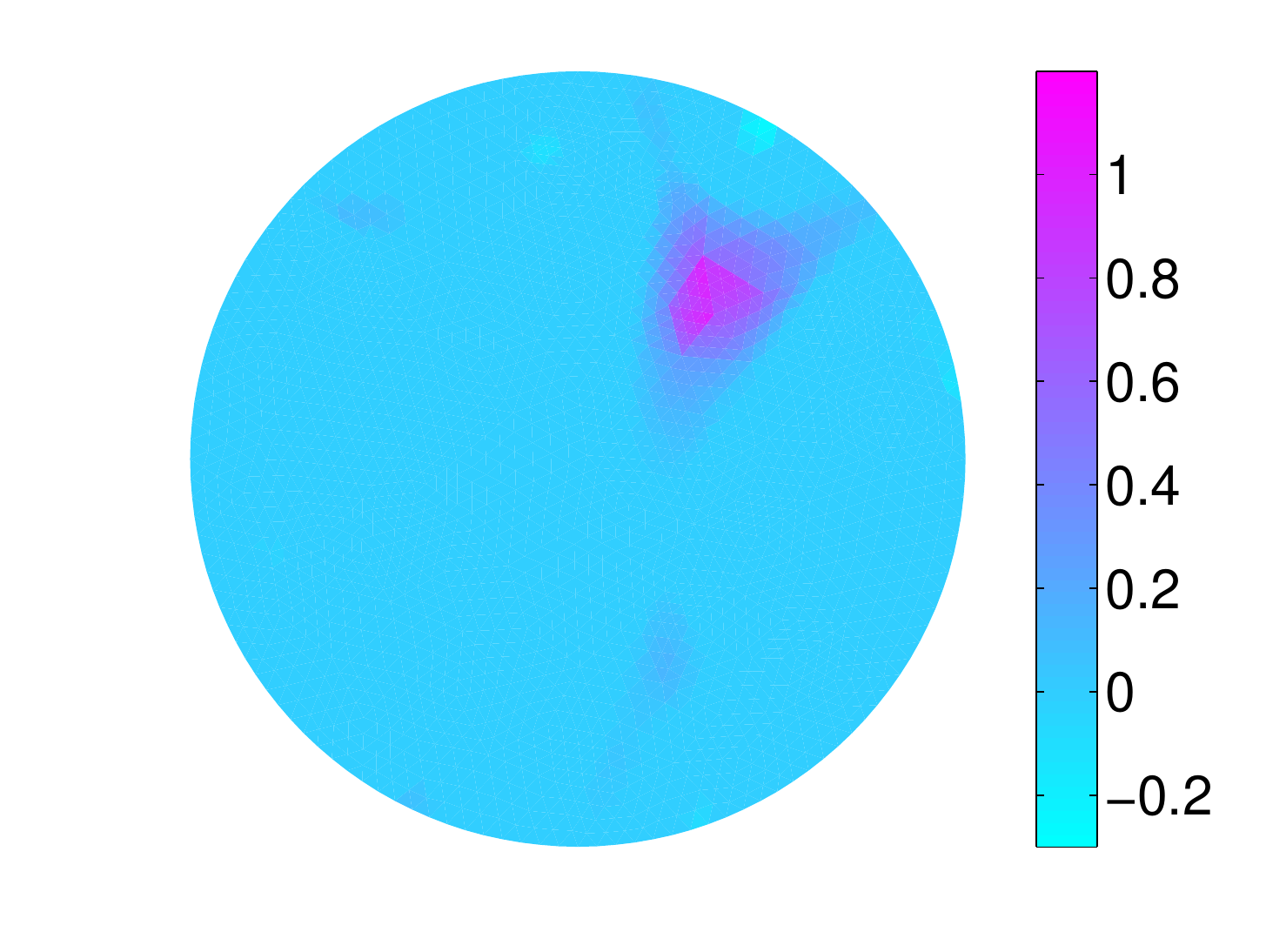}}\hfill
  \subfloat[recovered $\delta\sigma_2$]{\includegraphics[width=0.3\textwidth]{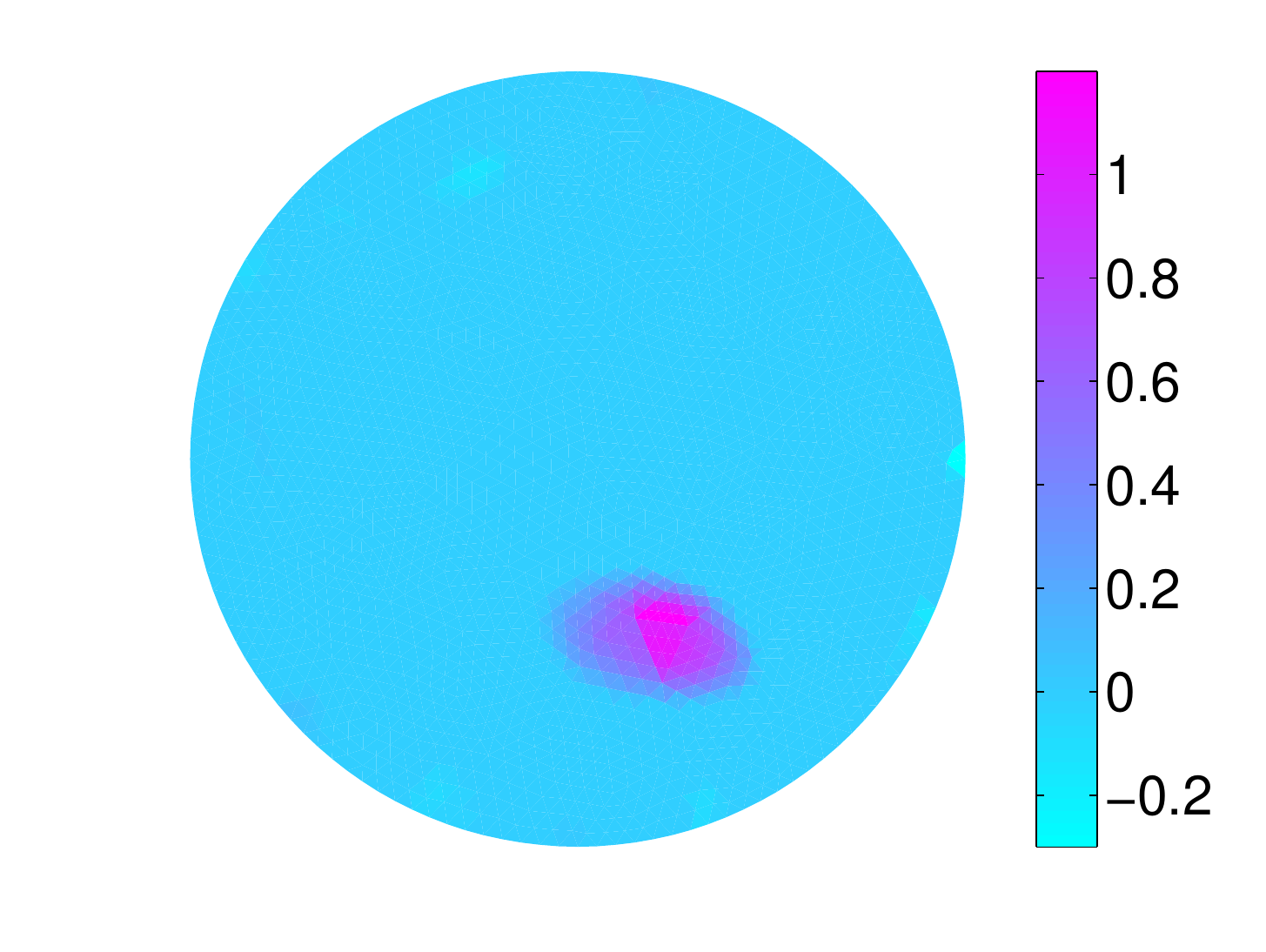}}

  \caption{Numerical results for Example~\ref{exam4} with $0.1\%$ data noise, fully known
  $s_k(\omega)$s. The recoveries in (b)-(c) are based on difference
  imaging, and those in (d)--(f) direct approach.}\label{fig:exam4}
\end{figure}

The results for Example~\ref{exam4} are given in Fig.~\ref{fig:exam4}. The
analysis in Section \ref{sub:cem-unknown} indicates that
the conductivity perturbation can be lumped to
$\delta\sigma_0$. The results confirm the analysis: when using the direct
approach, there are pronounced oscillations around
the boundary in the recovered $\delta\sigma_0$. However,
the recovered $\delta\sigma_1$ and $\delta\sigma_2$ are reasonable in both location
and size. The difference imaging can
also remove the contributions due to unknown electrode locations, since $s_k(\omega)$s are
incoherent both before and after differentiation.

In summary, as expected from the analysis of Sections~\ref{sec:unknownboundary} and \ref{sub:cem-unknown},
the mfEIT technique has significant potentials in handling modeling errors. The inclusion $\delta\sigma_0$
corresponding to $s_0$ may not be recovered. However, by mfEIT,
$\{\delta\sigma_k\}_{k=1}^K$ can be correctly recovered by either the
direct approach or difference imaging,
provided that $s_k$s or $s_k^\prime$s are sufficiently incoherent.

\section{Concluding Remarks}\label{sec:conclusion}
In this paper we have presented novel reconstruction methods in multifrequency EIT.
In particular, we have illustrated both analytically and numerically the significant
potentials of mfEIT in handling the modeling error due to an imperfectly known boundary shape.
We have also introduced a new and efficient group sparse reconstruction algorithm for the linearized EIT problem.
The techniques may be extended to quantitative photoacoustic imaging
from multispectral measurements \cite{daniel}.

This work represents only a first step towards the rigorous mathematical and numerical analysis of
mfEIT. There are a few questions deserving further research. For instance, beyond the linearized regime, the nonlinear
approach may be more appropriate, but it comes with  significant computational overhead, due to a large number
of PDEs involved. It is imperative to develop fast image reconstruction algorithms and to provide theoretical justifications.
Moreover, in this work we have mainly focused on the
recovery of the abundances. It would be of great interest to derive
sufficient conditions for the simultaneous recovery of partial spectral profiles, under suitable structural
prior knowledge, e.g., the (disjoint) sparsity of abundances. It is expected that this issue may have different features in the nonlinear regime.

\section*{Acknowledgements}
The authors are grateful to the anonymous referees and the associate editor Otmar Scherzer
for their constructive comments, which have led to an improved presentation of the work.
This work was partially supported by the ERC Advanced Grant Project
MULTIMOD-267184 and the EPSRC grant EP/M025160/1.

\bibliographystyle{abbrv}

\bibliography{eit}

\end{document}